\documentclass[a4paper]{amsart}
\usepackage{latexsym}\usepackage{ifthen}\usepackage[leqno]{amsmath}
\usepackage{enumerate}\usepackage{calc}\usepackage{hyphenat}
\usepackage{amstext,amsbsy,amsopn,amsthm,amsgen}
\usepackage{amsfonts,amscd,amsxtra,upref}
\usepackage{mathrsfs}\usepackage{euscript}\usepackage{amssymb}
\usepackage{hyperref}
\usepackage{color}
\usepackage{tkz-graph,stackrel}
\swapnumbers
\theoremstyle{plain}
\makeatletter\@namedef{subjclassname@2010}{\textup{2010} Mathematics Subject Classification}
\makeatother
\newtheorem{Thm}{Theorem}[section]
\newtheorem{Lem}[Thm]{Lemma}
\newtheorem{Cor}[Thm]{Corollary}
\newtheorem{Pro}[Thm]{Proposition}

\theoremstyle{definition}
\newtheorem{Def}[Thm]{Definition}
\newtheorem{Exm}[Thm]{Example}

\newtheorem{Prb}[Thm]{Problem}
\theoremstyle{remark}
\newtheorem{Rem}[Thm]{Remark}
\numberwithin{equation}{section}

\newcommand{\ITE}[3]{\ifthenelse{#1}{#2}{#3}}\newcommand{\ITEE}[4][]{\ITE{\equal{#2}{#3}}{#4}{#1}}
 \ITE{\isundefined{\texorpdfstring}}{\newcommand{\texorpdfstring}[2]{#1}}{}

\newenvironment{cor}[2][]{\ITEE[{\begin{Cor}[#1]}]{#1}{}{\begin{Cor}}\label{cor:#2}}{\end{Cor}}
\newenvironment{dfn}[2][]{\ITEE[{\begin{Def}[#1]}]{#1}{}{\begin{Def}}\label{def:#2}}{\end{Def}}
\newenvironment{exm}[2][]{\ITEE[{\begin{Exm}[#1]}]{#1}{}{\begin{Exm}}\label{exm:#2}}{\end{Exm}}
\newenvironment{lem}[2][]{\ITEE[{\begin{Lem}[#1]}]{#1}{}{\begin{Lem}}\label{lem:#2}}{\end{Lem}}
\newenvironment{prb}[2][]{\ITEE[{\begin{Prb}[#1]}]{#1}{}{\begin{Prb}}\label{prb:#2}}{\end{Prb}}
\newenvironment{pro}[2][]{\ITEE[{\begin{Pro}[#1]}]{#1}{}{\begin{Pro}}\label{pro:#2}}{\end{Pro}}
\newenvironment{rem}[2][]{\ITEE[{\begin{Rem}[#1]}]{#1}{}{\begin{Rem}}\label{rem:#2}}{\end{Rem}}
\newenvironment{thm}[2][]{\ITEE[{\begin{Thm}[#1]}]{#1}{}{\begin{Thm}}\label{thm:#2}}{\end{Thm}}

\newcommand{\COR}[2][!]{\ITEE{#1}{!}{Corollary~}\ITEE{#1}{s}{Corollaries~}\textup{\ref{cor:#2}}}
\newcommand{\DEF}[2][!]{\ITEE{#1}{!}{Definition~}\ITEE{#1}{s}{Definitions~}\textup{\ref{def:#2}}}
\newcommand{\EXM}[2][!]{\ITEE{#1}{!}{Example~}\ITEE{#1}{s}{Examples~}\textup{\ref{exm:#2}}}
\newcommand{\LEM}[2][!]{\ITEE{#1}{!}{Lemma~}\ITEE{#1}{s}{Lemmas~}\textup{\ref{lem:#2}}}
\newcommand{\PRO}[2][!]{\ITEE{#1}{!}{Proposition~}\ITEE{#1}{s}{Propositions~}\textup{\ref{pro:#2}}}
\newcommand{\REM}[2][!]{\ITEE{#1}{!}{Remark~}\ITEE{#1}{s}{Remarks~}\textup{\ref{rem:#2}}}
\newcommand{\THM}[2][!]{\ITEE{#1}{!}{Theorem~}\ITEE{#1}{s}{Theorems~}\textup{\ref{thm:#2}}}

\newcommand{\FFF}{\mathbb{F}}
\newcommand{\NNN}{\mathbb{N}}
\newcommand{\RRR}{\mathbb{R}}
\newcommand{\SSS}{\mathbb{S}}
\newcommand{\UUU}{\mathbb{U}}
\newcommand{\ZZZ}{\mathbb{Z}}
\newcommand{\FFf}{\CMcal{F}}
\newcommand{\LLl}{\CMcal{L}}
\newcommand{\PPp}{\CMcal{P}}
\newcommand{\CcC}{\EuScript{C}}
\newcommand{\DdD}{\EuScript{D}}
\newcommand{\Cc}{\mathfrak{C}}
\newcommand{\Dd}{\mathfrak{D}}
\newcommand{\Ee}{\mathfrak{E}}
\newcommand{\Ff}{\mathfrak{F}}
\newcommand{\Gg}{\mathfrak{G}}
\newcommand{\Hh}{\mathfrak{H}}
\newcommand{\Kk}{\mathfrak{K}}
\newcommand{\Pp}{\mathfrak{P}}
\newcommand{\Qq}{\mathfrak{Q}}
\newcommand{\Rr}{\mathfrak{R}}
\newcommand{\Ss}{\mathfrak{S}}
\newcommand{\Tt}{\mathfrak{T}}
\newcommand{\Ww}{\mathfrak{W}}
\newcommand{\Xx}{\mathfrak{X}}
\newcommand{\Zz}{\mathfrak{Z}}
\newcommand{\aA}{\mathfrak{a}}
\newcommand{\bB}{\mathfrak{b}}
\newcommand{\cC}{\mathfrak{c}}
\newcommand{\gG}{\mathfrak{g}}
\newcommand{\hH}{\mathfrak{h}}
\newcommand{\mM}{\mathfrak{m}}
\newcommand{\nN}{\mathfrak{n}}
\newcommand{\uU}{\mathfrak{u}}
\newcommand{\vV}{\mathfrak{v}}
\newcommand{\wW}{\mathfrak{w}}
\newcommand{\xX}{\mathfrak{x}}
\newcommand{\yY}{\mathfrak{y}}
\newcommand{\zZ}{\mathfrak{z}}
\newcommand{\ccC}{\mathscr{C}}
\newcommand{\eeE}{\mathscr{E}}
\newcommand{\ffF}{\mathscr{F}}
\newcommand{\hhH}{\mathscr{H}}
\newcommand{\kkK}{\mathscr{K}}
\newcommand{\mmM}{\mathscr{M}}
\newcommand{\ssS}{\mathscr{S}}
\newcommand{\xxX}{\mathscr{X}}
\newcommand{\dd}{\colon}
\newcommand{\df}{\stackrel{\textup{def}}{=}}
\newcommand{\epsi}{\varepsilon}
\newcommand{\grp}[1]{\langle#1\rangle}
\newcommand{\scalar}[2]{\left\langle#1,#2\right\rangle}
\newcommand{\scalarr}{\langle\cdot,\mathrm{-}\rangle}
\newcommand{\varempty}{\varnothing}
\newcommand{\OPN}[1]{\operatorname{#1}}
\newcommand{\card}{\operatorname{card}}
\newcommand{\Iso}{\operatorname{Iso}}
\newcommand{\UP}[1]{\textmd{\textup{#1}}}
\newcommand{\tfcae}{the following conditions are equivalent:}
\newcommand{\iaoi}{if and only if}

\newcommand{\ORD}{\preccurlyeq}
\newcommand{\MBOX}[1]{\mbox{\(#1\)}}
\newcommand{\SUPP}[1]{\MBOX{\OPN{supp}(#1)}}
\newcommand{\AUT}[2][]{\MBOX{\OPN{Aut}_{#1}(#2)}}
\newcommand{\COL}[1]{\MBOX{\OPN{Col}(#1)}}
\newcommand{\FIN}[2][]{\MBOX{\OPN{Fin}_{#1}(#2)}}
\newcommand{\PROd}[3]{\MBOX{{}^{#2}\!\prod_{#1}^{#3}}}
\newcommand{\PROD}[2][\ORD]{\MBOX{\PROd{#2}{#1}{*}}}
\newcommand{\DIF}[2]{\MBOX{\OPN{Diff}(#1,#2)}}
\newcommand{\LAB}[3][\ORD]{\MBOX{\ell_{#1}(#2,#3)}}
\newcommand{\AP}[1][]{\mbox{\textup{(#1AP)}}}
\newcommand{\MBP}{\mbox{\textup{(MBP)}}}
\newcommand{\OBP}{\mbox{\textup{(OBP)}}}
\newcommand{\SBP}{\mbox{\textup{(SBP)}}}
\newcommand{\ITP}{\mbox{\textup{(ITP)}}}
\newcommand{\OB}[1][\ccC]{\MBOX{\textup{\textsf{Ob}}(#1)}}
\newcommand{\ARR}[3][\ccC]{\MBOX{\textup{\textsf{Mor}}_{#1}(#2,#3)}}
\newcommand{\SYS}[3]{\MBOX{\langle#1,#2\rangle_{#3}}}
\newcommand{\CARD}{cardinal}

\newcommand{\VERT}[1][]{%
 \SetUpVertex[MinSize=1pt,LabelOut,Ldist=0pt,FillColor=white,Ldist=-1mm,%
 NoLabel#1]%
}
\newcommand{\VERTCAP}[3][]{%
 \VERT[,LineColor=white]\SOWE[Math,NoLabel=false,L={#3}#1](#2){Q}
}
\newcommand{\EDGE}[4][]{\Edge[#1label=\mbox{$#2$}](#3)(#4)}
\newcommand{\DIAG}[3][1.5]{%
 \begin{minipage}{#2\textwidth}\centering\begin{tikzpicture}[scale=#1]\VERT%
 \tikzstyle{LabelStyle}=[fill=white]\SetVertexMath{}#3%
 \end{tikzpicture}\end{minipage}%
}
\newcommand{\DIAGF}[3]{%
 \DIAG{0.3}{\Vertices{circle}{Z,W,Y}\VERT[,Lpos=30,Ldist=-1mm]\WE(Z){X}#1%
 \VERTCAP{#2}{#3}}%
}
\newcommand{\DIAGK}[2]{%
 \DIAGF{\EDGE{a}{Z}{W}\EDGE{a}{W}{X}\EDGE{a}{Y}{X}\EDGE{b}{W}{Y}\EDGE{b}{Z}{Y}%
 \EDGE{#1}{X}{Z}}{Z}{#2}%
}
\newcommand{\DIAGP}[9]{%
 \DIAG[1.8]{0.365}{\Vertices{circle}{X,Y,Z,W,V}\AddVertexColor{black}{Y,W}%
 \EDGE{a}{X}{Y}\EDGE{#1}{X}{Z}\EDGE{#2}{X}{W}\EDGE{#3}{X}{V}\EDGE{#4}{Y}{Z}%
 \EDGE{#5}{Y}{V}\EDGE{#6}{Z}{W}\EDGE{#7}{Z}{V}\EDGE{#8}{W}{V}%
 \VERTCAP[,Ldist=8mm]{X}{#9}}%
}
\newcommand{\DIAGQ}{%
 \DIAG[1.8]{0.365}{\Vertices{circle}{X,Y,Z,W,V}%
 \EDGE{a}{X}{Y}\EDGE{a}{Y}{Z}\EDGE{a}{Z}{W}\EDGE{a}{W}{V}\EDGE{a}{V}{X}%
 \EDGE{b}{X}{Z}\EDGE{b}{Y}{W}\EDGE{b}{Z}{V}\EDGE{b}{W}{X}\EDGE{b}{V}{Y}}%
}

\begin{document}

\title[Absolute homogeneity]{Extensive approach to absolute homogeneity}
\author[P.\ Niemiec]{Piotr Niemiec}
\address{Wydzia\l{} Matematyki i~Informatyki\\Uniwersytet Jagiello\'{n}ski\\%
 ul.\ \L{}o\-ja\-sie\-wi\-cza 6\\30-348 Krak\'{o}w\\Poland}
\email{piotr.niemiec@uj.edu.pl}
\thanks{Research supported by the National Center of Science [grant no
 2021/03/Y/ST1/00072].}
\begin{abstract}
The main aim of the paper is to study in greater detail absolutely homogeneous
structures (that is, objects with the property that each partial isomorphism
extends to a global automorphism), with special emphasis on metric spaces and
(possibly infinite, full) graphs with edge-coloring. Besides, a general
categorical approach to this concept is presented. The main achievement of
the paper is the discovery of one-to-one correspondence between absolutely
homogeneous objects and certain classes (that become sets when isomorphic
objects are identified) of ``finite'' objects that satisfy a few quite general
axioms (such as amalgamation and heredity). It is also introduced and discussed
in detail the concept of products for graphs with edge-coloring (that produces
an absolutely homogeneous graph provided all factors are so). Among the most
significant results of the paper, it is worth mentioning a full classification
(up to isometry) of all absolutely homogeneous ultrametric spaces as well as of
all absolutely homogeneous graphs with edge-coloring in which all triangles are
isosceles or in which all triangles are (precisely) tricolored.
\end{abstract}
\maketitle

\tableofcontents

\section{Introduction}
Homogeneity belongs to the most interesting notions in topology and occurs when
the action of a certain group of transformations of the space under
consideration is transitive. For example, a topological space is homogeneous if
the natural action of its homeomorphism group is transitive, whereas a metric
space is \emph{metrically} homogeneous if the action of the isometry group is
so. In recent years homogeneity and its higher level variants (such as
ultrahomogeneity) are widely studied by many mathematicians, including van Mill,
Arhangel'skii, Kechris, Solecki, Kubi\'{s}, Malicki, Doucha and many others.
Special attention is drawn to the concept of Fra\"{\i}ss\'{e} classes and limits
(\cite{fra}) where a higly homogeneous object appears naturally. Although
ultrahomogeneous spaces and other such objects or structures are widely studied,
not much is known about absolutely homogeneous spaces. As ultrahomogeneity says
about extending isomorphisms between finite subsets/subobjects to global
automorphisms, absolute homogeneity has no restrictions in this matter: it says
that each isomorphism between any subsets/subobjects extends to an automorphism
of the whole space/object under consideration. Until now there are known only
a few examples of infinite absolutely homogeneous metric space: Euclidean
spaces; Euclidean spheres; hyperbolic spaces; the rationals, the integers and
similar subspaces of the reals. Notice that among these examples only the first
three deserve not to be considered as trivial. It is highly probable that these
are the only known non-trivial examples of infinite absolutely homogeneous
metric spaces. (In particular, to the best of our knowledge, it is known no
example of such a non-separable space). In other branches of mathematics it is
also hard to find any example of such structures. So, a natural question arises
about the reason for the absence of other examples. This motivated us to
investigate absolute homogeneity on a broader scale. As an effect, we have
obtained results that we consider mysterious. Namely, when the issue is settled
in a category where ``finiteness'' is quite well understood (a formal definition
will be given is Section~\ref{sec:cat}), an absolutely homogeneous object is
completely determined by the class of its finite subobjects. To be more precise,
let us illustrate this phenomenon with an example of metric spaces. One of our
main results (\THM{embed}; consult also \THM[s]{embedd} and \ref{thm:AB-HO-cat})
includes the property that a totally arbitrary metric space isometrically embeds
into a fixed absolutely homogeneous metric space if (and only if) all its finite
subsets do so. In this way we may assign to each absolutely homogeneous metric
space a collection, called by us \emph{skeleton}, of all its finite subsets and
try to `translate' topological properties of the whole space in terms of its
skeleton. Since the above assignment is one-to-one, it is also natural to ask
about the range of this assignment, that is: which collections of finite metric
spaces come from an absolutely homogeneous metric space (i.e., coincide with
some skeleton)? In this way we arrive at the concept of
a \emph{skeletoid}---a collection of finite metric spaces that `resembles'
skeletons, but is defined by a list of purely intrinsic properties (such as
amalgamation). It turns out that one of the main tools in this approach is
a classical Erd\H{o}s-Rado theorem (on, roughly speaking, monochromatic
subgraphs in huge infinite graphs). Our main result on skeletoids
(\THM{skeleton}; consult also \THM[s]{dskeleton} and \ref{thm:skton-cat}) gives
an equivalent condition for a skeletoid to be a skeleton. In a special case,
when the skeloetoid is countable, this criterion may be highly simplified and
expressed in terms of countable chains of objects from the skeletoid.\par
To construct innumerable amounts of absolutely homogeneous structures (more
specifically, graphs with edge-coloring), we develop a method of producing new
such objects from pre-existing ones of the same kind. They are created as
certain Cartesian-like products (see \DEF{product} and \THM{prod-AB-HO}) where
well-orders appear naturally (although one starts from an arbitrary total
order). This tool enables us to give a full classification of all absolutely
homogeneous ultrametric spaces (\COR{ultrametric}) as well as such graphs with
edge-coloring in which all triangles are isosceles (\THM{classification}). In
each of these two cases we can speak about \emph{primary decomposition} as it
has the form of the product of decomposable factors and this form is unique (up
to isomorphism of respective factors). What is worth underlying here is that,
contrary to ordinary Cartesian products (or, e.g., prime decomposition),
the order of factors in the primary decomposition does matter.\par
Other results that we consider worth mentioning here are: a full classification
of absolutely homogeneous graphs with edge-coloring in which all triangles are
(precisely) tricolors (\THM{skton-grp}) and its counterpart for directed graphs
with edge-coloring (\THM{dskton-grp}). By the way, it turned out that each
group possesses four natural structures of absolutely homogeneous undirected
graphs with edge-coloring and when equipped with such a structure a group
becomes an undirected graph precisely when it is Boolean. So, directed graphs
are much more restrictive than undirected. And we may consider absolutely
homogeneous (un)directed graphs with edge-coloring as ``generalised (Boolean)
groups''. (However, groups are quite simple, or even primitive, as such
structures since their graph-automorphisms are uniquely determined by the value
at a single point.)\par
Finally, we have managed to formulate (and prove) a quite general theorem on
absolutely homogeneous objects and their finitary structures (that is,
collections of finite subobjects) for abstract categories that satisfy certain
quite natural axioms (resembling properties of the category of sets).\par
The paper is organised as follows. Below we fix the notation and terminology.
In Section~\ref{sec:metr} we briefly report what is known about absolutely
homogeneous metric spaces and discuss three classical classes of such spaces.
Section~\ref{sec:gec} is the longest part of the paper and deals with graphs
with edge-coloring. Here we define skeletons, skeletoids and amalgamation
property (\DEF[s]{skelet} and \ref{def:amalgam}) as well as prove two of
the main results of the paper (\THM[s]{embed} and \ref{thm:skeleton}; see also
\COR{countable}). Subsection~\ref{sec:metric} is devoted to skeletoids
consisting of metric spaces. Among other things, we show there that
an absolutely homogeneous metric space is either locally connected or locally
totally disconnected (\PRO{prop-AHM}) and in \THM{H-B-AB-HO} give a (convenient)
equivalent condition for a skeletoid to be a skeleton of an absolutely
homogeneous \emph{Heine-Borel} metric space (that is, in which all closed balls
are compact).\par
In Section~\ref{sec:products} we introduce a new concept of products of families
of graphs with edge-coloring (\DEF{product}) and prove \THM{prod-AB-HO} which
says that such a product of a (productable) collection of absolutely homogeneous
graphs is absolutely homogeoenous as well. We the aid of this tool we give
examples of absolutely homogeneous graphs with edge-coloring of arbitrarily
large size (\COR{2toY}). We fully discuss (in \PRO{prod-metric}) the case when
all factors are metric spaces and their product is also metric (which is not
automatic). Based on this result, we give examples of absolutely homogeneous
metric spaces of maximal possible size with a prescribed set of attainable
distances by a metric (\THM{exist-metric}). Examples of Heine-Borel such spaces
are given in \THM{exm-AB-HO}. Section~\ref{sec:classification} is devoted to
a full classification of absolutely homogeneous graphs with edge-coloring in
which all triangles are isosceles (\THM{classification}) as well as of all
absolutely homogeneous ultrametric spaces (\COR{ultrametric}). These are done
with the aid of products introduced in an earlier part of the paper.\par
In a short Section~\ref{sec:card} we present two characteristic cardinals of
skeletoids and discuss them on the example of three classical classes of
absolutely homogeneous metric spaces (Euclidean, hiperbolic, spheres).
\THM{ht-hyp-sph}, formulated in this part, is a counterpart of Menger's theorem
\cite{mg1,mg2} (on isometric embeddability into Euclidean spaces) for hyperbolic
spaces and Euclidean spheres. Section~\ref{sec:directed} deals with directed
graphs with edge-coloring. In that case a new phenomenon occurs (denoted by us
(OBP); see \DEF{OBP}). A property of skeletons related to this phenomenon is
formulated and proved in \THM{OBP}, which is another consequence of
the Erd\H{o}s-Rado theorem. Then we repeat main results of previous parts
(obtained for undirected graphs) and discuss absolutely homogenenous directed
graphs with edge-coloring which are \emph{uniquely homogeneous} (this means that
each partial morphism defined on a non-empty subgraph extends to a unique
automorphism). Contrary to undirected graphs (where all such graphs with
the aforementioned property are isomorphic to Boolean groups), directed graphs
with such a property can be non-isomorphic to groups, but even though they are
very `close' to them (consult \THM{uniq-homo}).\par
In Section~\ref{sec:ges} we propose a more intrinsic approach to both types of
graphs (directed and undirected) with edge-coloring and introduce structures
with edge similarity. We define there compact Hausdorff spaces with edge
similarity and establish a little bit surprising result (\THM{cdGES}) about them
which shows that in a certain sense such spaces resemble compact metric spaces.
In a short Subsection~\ref{sec:abstr} we show how studies of these compact
(and more general) structures can be interpreted to a purely group-theoretic
language. Section~\ref{sec:cat} is devoted to a categorical approach to
absolute homogeneity. We define there (in \DEF{axioms}) so-called \emph{\CARD}
categories and formulate and prove for such categories counterparts of theorems
on skeletons and skeletoids (see \THM[s]{AB-HO-cat} and
\ref{thm:skton-cat}).\par
In the last part, Section~\ref{sec:fin}, we list open problems related to
the topics discussed in the paper.\par
We believe our work will give a solid background for further studies on
absolutely homogeneous structures.

\subsection{Notation and terminology}
In this paper \(\NNN = \{0,1,2,\ldots\}\) (that is, \(0 \in \NNN\)).\par
From now on, whenever in this paper we speak about homogeneity (or some of its
variants) in metric spaces, we mean metric homogeneity---that is, a property
related to the natural action of the isometry group of the space.\par
Isometric maps need not be surjective and an \emph{isometry} is, by definition,
a bijective isometric map. The isometry group of a metric space \((X,d\)) is
denoted by \(\Iso(X,d)\) or, for brevity, by \(\Iso(X)\). By a \emph{partial}
isometry from \(X\) to \(Y\) we mean any isometric map from a subspace of \(X\)
into \(Y\). When \(Y = X\), we speak about partial isometries of \(X\).\par
When \(G\) is a group of transformations of a set \(X\), the \emph{isotropy}
group of \(G\) at a point \(a \in X\) is the subgroup \(\OPN{stab}_G(a) \df
\{g \in G\dd\ g(a) = a\}.\)\par
The cardinality of a set \(A\) is denoted by \(\card(A)\).

As the Ramsey-Erd\H{o}s-Rado type theorems (\cite{ram}, \cite{e-r}; consult also
Section~6 in Chapter~IX of \cite{k-m}) are our main tools, for the reader's
convenience below we formulate this variant (of a part) of them that will find
applications in the sequel.

\begin{thm}{r-e-r}
Let \(f\dd X \times X \to Y\) be a symmetric function (that is,
\(f(x,y) = f(y,x)\) for all \(x, y \in X\)) where \(X\) is an infinite set.
\begin{itemize}
\item \UP{(Ramsey)} If \(Y\) is finite, then there exists an infinite set
 \(A \subset X\) and a point \(c \in Y\) such that
 \begin{equation}\label{eqn:r-e-r}
 \forall a, b \in A,\ a \neq b\dd\qquad f(a,b) = c.
 \end{equation}
\item \UP{(Erd\H{o}s-Rado)} If \(Y\) is infinite and \(\card(X) >
 2^{\card(Y)}\), then there eixsts a set \(A \subset X\) and a point \(c \in Y\)
 such that \(\card(A) > \card(Y)\) and \eqref{eqn:r-e-r} holds.
\end{itemize}
\end{thm}

\section{Absolute homogeneity in metric spaces}\label{sec:metr}

\begin{dfn}{abshm}
A metric space \((X,d)\) is said to be:
\begin{itemize}
\item \emph{homogeneous} if \(\Iso(X,d)\) acts transitively on \(X\); that is,
 if for any two distinct points \(x\) and \(y\) of \(X\) there is an isometry
 \(u \in \Iso(X,d)\) such that \(u(x) = y\);
\item \emph{\(n\)-point homogeneous} (where \(n\) is a positive integer) if
 every isometric map \(u\dd A \to X\) of any at most \(n\)-point subset \(A\) of
 \(X\) into \(X\) extends to a global isometry \(v \in \Iso(X,d)\);
\item \emph{ultrahomogeneous} or \emph{finitely homogeneous} if it is
 \(n\)-point homogeneous for each \(n\);
\item \emph{\textbf{absolutely homogeneous}} if every isometry between two
 arbitrary subsets of \(X\) is extendable to a global isometry on \(X\).
\end{itemize}
\end{dfn}

It is worth noting that absolute homogeneity is sometimes called
\emph{free-mobility} (see, e.g., \cite{bi1,bi2}).\par
Observe that any group equipped with a left (or right) invariant metric is
homogeneous. This shows that even the realm of compact connected homogeneous
metric spaces contains quite pathological spaces (such as solenoids).
A situation becomes more transparent (for connected locally compact spaces) when 
dealing with higher levels of homogeneity. For example, Wang \cite{wan} gave
a full classification of two-point homogeneous compact connected metric spaces,
whereas Tits \cite{tit} did it for locally compact (connected metric) spaces.
It turned out that all such spaces are Riemannian manifolds. A bit later
Freudenthal \cite{fr1} (cf. also Subsection 2.21 in \cite{fr2}) implicitly
classified all 3-point locally compact connected metric spaces. Modern
discussion on this topic can be found in Section~7 of \cite{pn0}. Probably
the ``easiest'' example of a (compact connected) metric space that is two-point
but not 3-point homogeneous is the two-dimensional real projective space
\(P^2(\RRR)\) (equipped with some natural metric). (Actually, all real
projective spaces of dimension greater than 1 are 2-point but not 3-point
homogeneous. See \cite{pn0} for more details.) Below, in \THM{Freu} we list all
3-point homogeneous locally compact connected metric spaces (having more than
one point). It is extremely interesting that all they are absolutely
homogeneous. Before we introduce these spaces, we give three more (classical)
examples of highly homogeneous metric spaces:
\begin{itemize}
\item One of them is a famous Urysohn space. It is a unique (up to isometry)
 separable complete metric spaces that is both ultrahomogeneous and metrically
 universal for separable metric spaces (that is, it contains an isometric copy
 of each separable metric space). Huhunai\v{s}vili \cite{huh} proved that this
 space is compactly homogeneous (which means that isometries between compact
 subsets extend to global isometries). (According to Melleray's result
 \cite{me1}, compact homogeneity is the highest level of homogeneity the Urysohn
 space can have---in particular, as it is well-known, it is not absolutely
 homogeneous). More information about the Urysohn universal space can be found
 e.g. in \cite{me2}.
\item All homogeneous ultrametric spaces may serve as another example of spaces
 with high level of homogeneity: it is a kind of folklore that each such a space
 is automatically ultrahomogeneous. (Recall that a metric space \((X,d)\) is
 \emph{ultrametric} if \(d(x,z) \leq \max(d(x,y),d(y,z))\) for all \(x, y, z \in
 X\).) A trivial example of such spaces are discrete metric spaces (where all
 points are in distance 1)---all they are finitely homogeneous, but only finite
 are absolutely homogeneous.\par
 All absolutely homogeneous ultrametric spaces will be fully classified in
 \COR{ultrametric} in Section~\ref{sec:classification} below.
\item As it may be simply verified, the space of all rational numbers (with
 natural metric) is absolutely homogeneous. Undoubtedly, it is the simplest
 possible example of such a space that is, in addition, non-complete. It is also
 hereditary disconnected, which may be seen as a typical property of absolutely
 homogeneous metric spaces (see item (D) of \PRO{prop-AHM} below).
\end{itemize}

Now we pass to three classical classes of absolutely homogeneous metric spaces.
In all subsections below, \(n\) denotes a positive integer. Recall that a metric
\(d\) on a set \(X\) is \emph{geodesic} if any two points of the space \(X\) can
be joint by a line segment; that is, if for any two distinct points \(x\) and
\(y\) of \(X\) there is an isometric map \(\gamma\dd [0,r] \to X\) (where \(r =
d(x,y)\)) such that \(\gamma(0) = x\) and \(\gamma(r) = y\).

\subsection{Euclidean spaces (Euclidean geometry)}
Denote by \(d_e\) the Euclidean metric on \(\RRR^n\). That is: \(d_e(x,y) \df
\|x-y\|_2\) where \(\|x\|_2 = \sqrt{\scalar{x}{x}}\) and
\begin{equation}\label{eqn:inner}
\scalar{(x_1,\ldots,x_n)}{(y_1,\ldots,y_n)} = \sum_{k=1}^n x_k y_k.
\end{equation}
It is well-known (and easy to show) that \((\RRR^n,d_e)\) is absolutely
homogeneous. (The main point in the proof is that any partial isometry of this
space that fixes the origin preserves the standard inner product \(\scalarr\)
defined above.) Of course, \(d_e\) is geodesic.\par
Every isometry on \(\RRR^n\) (and thus also every partial isometry) is of
the form \(u(x) = a+V(x)\) where \(V \in O_n\) is the orthogonal matrix (that
is, a linear isometry). The isotropy group of \(\Iso(\RRR^n,d_e)\) at \(0\)
coincides with the orthogonal group \(O_n\) (whereas \(\Iso(\RRR^n,d_e)\) is
the semidirect product \(\RRR^n \rtimes O_n\)).

\subsection{Euclidean spheres (elliptic geometry)}
It follows from the material of the previous subsection that the Euclidean
sphere \(\SSS^{n-1} \df \{x \in \RRR^n\dd\ \|x\|_2 = 1\}\) equipped with \(d_e\)
is absolutely homogeneous as well. However, although \(\SSS^{n-1}\) is connected
for \(n > 1\), the metric \(d_e\) is not geodesic on the sphere. Therefore, in
differential geometry one introduces more appropriate metric, naturally
``induced'' by \(d_e\), which is geodesic (for \(n > 1\))---the so-called
\emph{great-circle distance}:
\begin{equation}\label{eqn:great-circle}
d_s(x,y) \df \frac{1}{\pi} \arccos(\scalar{x}{y}) = \frac{1}{\pi} \arccos
\frac{2-d_e(x,y)^2}{2}.
\end{equation}
(Actually, in the above we scaled the classical great-circle distance so that
the sphere has diameter 1.) Since \(d_s\) is obtained as the composition of
a one-to-one function (namely: \(\alpha(t) = \frac{1}{\pi} \arccos
\frac{2-t^2}{2}\)) with \(d_e\), it follows that also \((\SSS^n,d_s)\) is
absolutely homogeneous.\par
Every isometry on \((\SSS^{n-1},d_s)\) extends (uniquely) to a linear isometry
on \((\RRR^n,d_e)\) and thus \(\Iso(\SSS^{n-1},d_s)\) may \underline{naturally}
be identified with \(O_n\). When \(n > 1\), the isotropy group at
\(e \df (1,0,\ldots,0) \in \SSS^{n-1}\) consists exactly of all the matrices of
the form \(\begin{pmatrix} 1& 0\\0 & U\end{pmatrix}\) where \(U \in O_{n-1}\)
and thus \(\OPN{stab}(e)\) may be identified with \(O_{n-1}\).

\subsection{Hyperbolic spaces (hyperbolic geometry)}
Hyperbolic geometry arose as the example of a non-Euclidean geometry (where
the Euclid's fifth postulate is false). It can be realized in many ways, here
we present a little bit modified Minkowski model (we follow \cite{pn0} where
the reader may find more about this space, with proofs). It is the space
\(H^n(\RRR) \df \RRR^n\) equipped with the hyperbolic metric:
\begin{equation}\label{eqn:hyper}
d_h(x,y) \df \cosh^{-1}([x][y] - \scalar{x}{y}), \qquad \textup{where } [x]
\df \sqrt{1+\|x\|_2}
\end{equation}
(recall that \(\cosh^{-1}(t) = \log(t + \sqrt{t^2-1})\) for \(t \geq 1\)). It is
highly non-trivial that \(d_h\) is indeed a metric that is geodesic. Moreover,
\((H^n(\RRR),d_h)\) is absolutely homogeneous for all \(n\), and
\((H^1(\RRR),d_h)\) is isometric to \((\RRR,d_e)\).\par
Although the isotropy group of \(\Iso(H^n(\RRR),d_h)\) at the origin coincides
with \(O_n\), the whole isometry group is more complicated. Each isometry on
\(H^n(\RRR)\) is of the form:
\begin{equation}\label{eqn:trans-hyper}
\Phi(x) = U(x) + \Bigl([x]+\frac{\scalar{U(x)}{y}}{[y]+1}\Bigr)y
\end{equation}
where \(U \in O_n\) and \(y \in \RRR^n\) are arbitrary. (The pair \((U,y)\) is
uniquely determined by \(\Phi \in \Iso(H^n(\RRR))\).) More on hyperbolic
geometry the reader can find e.g. in the book \cite{ben}.

\subsection{All absolutely homogeneous locally compact connected metric spaces}
The following result, essentially due to Freudenthal \cite{fr1,fr2}, is the main
motivation for our investigations. A detailed sketch of the proof can be found
in \cite{pn0}.

\begin{thm}{Freu}
Every 3-point homogeneous locally compact connected metric space having more
than one point is isometric to a space of the form \((X,\omega \circ \rho)\)
where \((X,\rho)\) is one of \((\RRR^n,d_e)\ (n > 0)\), \((\SSS^n,d_s)\
(n > 0)\), \((H^n(\RRR),d_h)\ (n > 1)\), and \(\omega\) is a continuous
one-to-one map (defined on the range of the metric \(\rho\)) vanishing at 0 such
that \(\omega(x+y) \leq \omega(x)+\omega(y)\) for all (possible) \(x,y \geq
0\).\par
In particular, each such a space is absolutely homogeneous.
\end{thm}

All that we are going to present about absolutely homogeneous metric spaces
can actually be formulated (with no loss on transparency) for more general
structures which we now turn to.

\section{Graphs with edge-coloring}\label{sec:gec}

Recall that an (undirected) \emph{graph} is a pair \(G = (V,E)\) where \(V\) is
any (possible infinite) set (of \emph{vertices}) and \(E\) is a set of (exactly)
two-point subsets of \(V\) (called \emph{edges} of the graph \(G\)).
Equivalently, \(E\) can be considered as a symmetric binary relation on \(V\)
that is disjoint (as a subset of \(V \times V\)) with the diagonal of \(V\).
The graph \(G\) is \emph{full} if \(E\) contains all two-point subsets of \(V\)
(equivalently, if \(E = (V \times V) \setminus \Delta_V\) where \(\Delta_V \df
\{(v,v)\dd\ v \in V\}\)).\par
To speak about absolute homogeneity on graphs in a greater generality, we
introduce

\begin{dfn}{gec}
A \emph{graph with edge-coloring} (abbreviated: a \emph{g.e.c.}) is a pair
\(\gG = (G,\kappa)\) where \(G = (V,E)\) is a graph and \(\kappa\) is any
function defined on \(E\), with values in some (arbitrary) set \(Y\). The set of
all values of \(\kappa\) is called \emph{palette} of the g.e.c. \(\gG\), to be
denoted by \(\COL{\gG}\), and the values of \(\kappa\) are \emph{colors} of
\(\gG\).\par
Let \(\gG_j = (V_j,E_j,\kappa_j)\) for \(j=1,2\) be two g.e.c.'s.
A \emph{partial morphism} between \(\gG_1\) and \(\gG_2\) is any one-to-one
function \(\alpha\dd A \to V_2\) (where \(A \subset V_1\)) such that
\[\{\alpha(p),\alpha(q)\} \in E_2 \iff \{p,q\} \in E_1\] for any distinct \(p, q
\in A\) and moreover \(\kappa_2(\{\alpha(p),\alpha(q)\}) = \kappa_1(\{p,q\})\)
for \(\{p,q\} \in E\). If
\begin{equation}\label{eqn:aux1}
A = V_1,
\end{equation}
the above \(\alpha\) is called a \emph{morphism}. If, in addition to
\eqref{eqn:aux1}, \(\alpha(V_1) = V_2\), \(\alpha\) is an \emph{isomorphism},
and if \(\gG_2 = \gG_1\) (and still \(A = V_1\) and \(\alpha(V_1) = V_1\)), it
is called an \emph{automorphism}. As usual, we call two g.e.c.'s
\emph{isomorphic} if there exists an isomorphism between them. We will write
\(\gG \equiv \hH\) to express that \(\gG\) and \(\hH\) are isomorphic.
The automorphism group of a g.e.c. \(\gG\) will be denoted by \(\AUT{\gG}\).\par
Two g.e.c.'s \(\gG_1 = (G_1,\kappa_1)\) and \(\gG_2 = (G_2,\kappa_2)\) are said
to be \emph{structurally equivalent} if there exists a bijection \(\xi\dd
\COL{\gG_1} \to \COL{\gG_2}\) such that \((G_1,\xi \circ \kappa_1)\) and
\((G_2,\kappa_2)\) are isomorphic g.e.c.'s.
\end{dfn}

\begin{rem}{classic-edge-col}
In a classical graph theory, when speaking about \emph{edge coloring}, one
usually requires that all incident edges (that is, edges emanating from a common
vertex) have different colors (such coloring is sometimes called \emph{proper}).
We wish to underline---at the very beginning of our exposition---that our theory
of graphs with edge-coloring (presented mostly in this and the next two
sections) should not be considered as a part of graph theory, mainly because we
are mostly interested in infinite graphs as well as---as we will see in
a while---in full graphs and hardly ever edge-colorings will be proper. For
example, all graphs with proper edge-coloring that are absolutely homogeneous
can simply be classified (see \THM{skton-grp} below).
\end{rem}

To make study of g.e.c.'s simpler, we make the following basic observation:
introducing a new color, called \emph{null} (or \emph{blank}), that
\underline{never} appears as a ``real'' color of (any) g.e.c., each g.e.c.
\(\gG\) may be considered as a \textbf{full} graph with edge-coloring: it is
enough to paint all edges from outside \(\gG\) null to obtain a full g.e.c.
\(\tilde{\gG}\). Observe that an arbitrary function is a (resp. partial)
morphism between \(\gG_1\) and \(\gG_2\) iff it is a (partial) morphism between
\(\tilde{\gG}_1\) and \(\tilde{\gG}_2\) (however, analogous statement is false
for structural equivalence). That is why we restrict all our further
considerations only to full g.e.c.'s.\par
\textbf{From now on, all g.e.c.'s are full.} In particular, for any g.e.c.
\(\gG = (V,E,\kappa)\) the set \(E\) coincides with \((V \times V) \setminus
\Delta_V\) and thus we will skip it in the notation---we shall write briefly
\((V,\kappa)\) to denote a full g.e.c. To avoid confusions or misunderstandings,
we will use the abbreviation \emph{GEC} in place of ``a full g.e.c.'' and
capital letters \(\Gg, \Hh\) instead of their small versions \(\gG, \hH\)
etc.

\begin{exm}{graph}
Let \(G = (V,E)\) be a graph. Define \(\kappa\dd (V \times V) \setminus \Delta_V
\to \{0,1\}\) by the rule: \(\kappa(a,b) = 1\) iff \(\{a,b\} \in E\). In this
way we obtain a GEC \(\Gg = (V,\kappa)\) \emph{incuced} by \(G\). Make the same
construction for the complement \(H\) of \(G\) (that is, \(H\) has the same set
of vertices as \(G\) and a two-point subset of \(V\) is an edge of \(H\)
precisely when it is not an edge of \(G\)) to obtain a GEC \(\Hh\). Observe that
in general \(\Gg\) and \(\Hh\) are not isomorphic. However, they are always
structurally equivalent (it suffices to switch the colors \(0\leftrightarrow1\)
in \(\Hh\) to get \(\Gg\)).
\end{exm}

When dealing with partial morphisms of a single GEC, we can make the definition
of a GEC more convenient (and probably simpler). Again, adding to its palette
a ``pseudocolor'' (this time, say, \emph{transparent}), we can paint all its
vertices this pseudocolor and then, identifying the set \(V\) of all vertices
with the diagonal, we come to the following convenient definition of a GEC:

\begin{dfn}{GEC}
A \emph{GEC} \(\Gg\) is a set \(V\) equipped with a surjection \(\kappa\dd V
\times V \to Y\) (where \(Y\) is a totally arbitrary set) such that for all \(x,
y, z \in V\):
\begin{enumerate}[(GEC1)]
\item \(\kappa(x,y) = \kappa(z,z) \iff x = y\);
\item \(\kappa(y,x) = \kappa(x,y)\).
\end{enumerate}
The set \(\{\kappa(x,y)\dd\ x \neq y\}\) of colors of edges is called
the \emph{palette} of the GEC. The color of all vertices is called
the \emph{pseudocolor} of the GEC. The \emph{size} of \(\Gg\), denoted by
\(|\Gg|\), is the cardinality of the set \(V\) of vertices.\par
A \emph{sub-GEC} of \(\Gg = (V,\kappa)\) is a GEC of the form \((A,\lambda)\)
where \(A \subset V\) and \(\lambda = \kappa\restriction{A \times A}\). In
that case we shall write, for simplicity, \((A,\kappa)\) instead of
\((A,\lambda)\).\par
A GEC \(\Gg\) is said to be \emph{degenerate} if \(|\Gg| \leq 1\). Otherwise it
is called \emph{non-degenerate}. If \(|\Gg| = 0\), we call \(\Gg\) \emph{empty}.
\end{dfn}

Note that a function \(f\dd A \to W\) (where \(A \subset V\) and \(\Gg =
(V,\kappa)\) and \(\Hh = (W,\lambda)\) are two GEC's) is a partial morphism
(in the sense of \DEF{gec}) iff
\begin{equation}\label{eqn:morphism}
\kappa(a,b) = \lambda(f(a),f(b))
\end{equation}
for any two \underline{distinct} points \(a, b\) of \(A\). If, in addition,
\(\Hh = \Gg\), in the above condition we may omit the assumption that the points
\(a\) and \(b\) be distinct. However, in the context of GEC's defined in
\DEF{GEC}, it is more natural to require that \eqref{eqn:morphism} holds for all
(not only distinct) points \(a\) and \(b\). From now on, we follow this last
approach. Since in GEC's also vertices are painted, definition of structural
equivalence needs fixing: two GEC's \(\Gg_j = (V_j,\kappa_j)\) with pseudocolor
\(p_j\ (j=1,2)\) are structurally equivalent if there exists a bijection
\(\theta\dd \COL{\Gg_1} \cup \{p_1\} \to \COL{\Gg_2} \cup \{p_2\}\) such that
the GEC's \((V_1,\theta \circ \kappa_1)\) and \((V_2,\kappa_2)\) are isomorphic.
In this case we call \(\phi\) a \emph{structural isotopy}.\par
It is now easily seen that each metric space \((X,d)\) is a GEC (with \(\kappa =
d\) and the palette consisting of some positive reals). Partial morphisms,
morphisms and isomorphisms between metric spaces as GEC's coincide with,
respectively, partial isometries, isometric maps and isometries.\par
Now we are ready to define the main notion of this section.

\begin{dfn}{absGEC}
A GEC is said to be \emph{absolutely homogeneous} (abbreviated: \emph{AB-HO}) if
every its partial morphism extends to an automorphism.
\end{dfn}

Let us start from basic facts on AB-HO GEC's.

\begin{pro}{1}
\begin{enumerate}[\upshape(A)]
\item A GEC \(\Gg = (V,\kappa)\) is an AB-HO GEC iff the following condition
 holds:
 \begin{quote}
 Whenever \(\psi\dd A \to V\) is a partial morphism of \(\Gg\) with \(A \neq
 V\), there exists a vertex \(b \in V \setminus A\) and a partial morphism
 \(\phi\dd A \cup \{b\} \to V\) that extends \(\psi\).
 \end{quote}
\item A morphism of an AB-HO GEC is an automorphism.
\end{enumerate}
\end{pro}
\begin{proof}
Item (A) easily follows from Zorn's lemma and the argument presented below in
the proof of (B). To show (B), observe that if \(\phi\dd V \to V\) is a morphism
of a GEC \(\Gg\), then \(\phi^{-1}\dd \phi(V) \to V\) is a partial morphism and
it cannot be (essentially) extended to a partial morphism. Therefore, if \(\Gg\)
is AB-HO, then \(\phi(V) = V\) and we are done.
\end{proof}

\begin{exm}{homo}
Let \(\Gg = (V,\kappa)\) be an AB-HO GEC and \(u, v \in V\) be two arbitrary
vertices. Since all vertices are painted the same color, a unique function from
\(\{u\}\) onto \(\{v\}\) is a partial morphism and thus there exists
an automorphism \(\alpha \in \AUT{\Gg}\) such that \(\alpha(u) = v\). This shows
that AB-HO GEC's are homogeneous in a classical sense.
\end{exm}

We call a set \(A\) of vertices of a GEC \(\Gg = (V,\kappa)\)
\emph{monochromatic} if the edges between distinct points of \(A\) are painted
the same color; that is, if there is \(c\) such that \(\kappa(a,b) = c\) for all
distinct \(a, b \in A\). In the above situation we call \(A\) also
\(c\)-monochromatic.\par
The next result shows that the size of an AB-HO GEC is controlled by the size of
its palette. Part (b) of that result is an immediate consequence of \THM{r-e-r}.

\begin{thm}{card}
Let \(\Gg = (V,\kappa)\) be an AB-HO GEC with palette \(Y\).
\begin{enumerate}[\upshape(a)]
\item For any \(y \in Y\) there exists a natural number \(N = N(y)\) such that
 each \(y\)-monochromatic subset of \(V\) has cardinality not greater than
 \(N\).
\item If \(V\) is infinite, then \(\card(V) \leq 2^{\card(Y)}\). In particular,
 absolutely homogeneous metric spaces are of size at most \(2^{2^{\aleph_0}}\).
\item If \(Y\) is finite, so is \(V\).
\end{enumerate}
\end{thm}
\begin{proof}
To prove (a), we argue by a contradiction. Assume there are finite
\(y\)-monochromatic sets \(A_1,A_2,\ldots\) such that \(\card(A_n) <
\card(A_{n+1})\) for all \(n\). Since \(\Gg\) is AB-HO, there are automorphisms
\(\phi_n\dd V \to V\) such that \(\phi_n(A_n) \subset A_{n+1}\). Then the sets
\(B_n \df (\phi_n \circ \ldots \phi_1)^{-1}(A_{n+1})\) are \(y\)-monochromatic
as well, and \(\card(B_n) = \card(A_{n+1})\). What is more, since
\(\phi_n^{-1}(A_{n+1}) \supset A_n\), we obtain \(B_n \supset \ldots \supset B_1
\supset A_1\). So, the set \(C \df \bigcup_{n=1}^{\infty} B_n\) is infinite and
\(y\)-monochromatic. By Zorn's lemma, there exists a maximal \(y\)-monochromatic
set \(D \subset V\) that contains \(C\). Now let \(\psi\dd D \to D\) be any
one-to-one function that is not surjective. Since it is a partial morphism, it
extends to an automorhism \(\mu \in \AUT{\Gg}\). But then \(\mu^{-1}(D)
\supsetneq D\) is \(y\)-monochromatic, which contradicts maximality of \(D\)
and finishes the proof of (a).\par
Now assume \(V\) is infinite. If \(V\) had size greater than \(2^{\card(Y)}\),
it would follow from \THM{r-e-r} that \(V\) contains an infinite monochromatic
set, which is impossible by (a). So, the whole conclusion of (b) holds. Since
part (c) follows from (b), the proof is finished.
\end{proof}

The above result implies that AB-HO GEC's are finite iff their palettes are
finite. In this paper we are mainly interested in infinite AB-HO GEC's. In
\COR{2toY} below we will show that for arbitrary infinite set \(Y\) there exists
an AB-HO GEC \((V,\kappa)\) whose palette coincides with \(Y\) and moreover
\(\card(V) = 2^{\card(Y)}\). These examples will include AB-HO metric spaces
(for arbitrary bounded infinite \(Y \subset (0,\infty)\)---consult
\THM{exist-metric} below). So, part (b) of the above result cannot be improved
in general.\par
The following result gives us a deeper insight into the structure of AB-HO
GEC's. It will serve as a fundamental tool in our whole theory.

\begin{thm}{embed}
Let \(\Hh = (W,\lambda)\) be an AB-HO GEC. For a GEC \(\Gg = (V,\kappa)\) \tfcae
\begin{enumerate}[\upshape(i)]
\item there exists a morphism from \(\Gg\) into \(\Hh\);
\item for every finite non-empty sub-GEC \(\Ff\) of \(\Gg\) there exists
 a morphism from \(\Ff\) into \(\Hh\).
\end{enumerate}
\end{thm}
\begin{proof}
Of course, we only need to show that (i) follows from (ii). To this end, it is
sufficient to prove, by transfinite induction on \(\mM\), that for every
cardinal \(\mM\):
\begin{itemize}
\item[\((\OPN{Emb})_{\mM}\)] \textit{each sub-GEC \(\Gg'\) of \(\Gg\) with
 \(|\Gg'| \leq \mM\) admits a morphism into \(\Hh\).}
\end{itemize}
So, fix \(\mM\) and assume \((\OPN{Emb})_{\alpha}\) holds for all cardinals
\(\alpha < \mM\). By (ii), we may and do assume that \(\mM\) is infinite.
Take any sub-GEC \((A,\kappa)\) of \(\Gg\) of size not greater than \(\mM\).
If \(\card(A) < \mM\), it suffices to apply the transifinite induction
assumption. So, assume \(\card(A) = \mM\) and consider any initial well-order
\(\ORD\) on \(A\); that is, \(\ORD\) is a well-order on \(A\) such that for any
\(a \in A\) the set \(I_a \df \{x \in A\dd\ x \ORD a\}\) has size less than
\(\mM\) (here we use our assumption that \(\mM\) is infinite). In particular,
thanks to the transfinite induction assumption, for any \(a \in A\) there is
a morphism \(\phi_a\dd I_a \to W\) of \((I_a,\kappa)\) into \(\Hh\). Now, using
again trasfinite induction (on \((A,\ORD)\)), we will construct a family
\(\{\xi_a\dd\ a \in A\}\) of morphisms \(\xi_a\dd I_a \to W\) such that
\begin{equation}\label{eqn:aux2}
\forall c \in I_b\dd \quad \xi_b\restriction{I_c} = \xi_c \qquad (b \in A).
\end{equation}
(Note that having such a collection, we finish the proof by defining a morphism
\(\theta\dd A \to W\) by: \(\theta(a) \df \xi_a(a)\) for \(a \in A\).)\par
When \(a\) is the least element of \(A\) (with respect to \(\ORD\)), we define
\(\xi_a\) as \(\phi_a\). Now assume \(a \in A\) is an element different from
the least one and we have already defined morphisms \(\xi_b\) for all \(b \ORD
a\) distinct from \(a\), such that \eqref{eqn:aux2} holds for all such \(b\).
This enables us to define a morphism \(\mu\dd I_a \setminus \{a\} \to W\) by
the rule: \(\mu(b) \df \xi_b(b)\). Denote for a while \(Z \df \mu(I_a \setminus
\{a\})\) and observe that \(\nu \df \phi_a \circ \mu^{-1}\dd Z \to W\) is
a morphism. Since \(\Hh\) is AB-HO, there is an automorphism \(\eta \in
\AUT{\Hh}\) that extends \(\nu\). Then \(\mu = \eta^{-1} \circ
\phi_a\restriction{I_a \setminus \{a\}}\) and thus \(\xi_a \df \eta^{-1} \circ
\phi_a\dd I_a \to W\) is a morphism extending \(\mu\) (and, consequently, all
\(\xi_b\)).
\end{proof}

The above result motivates us to introduce the following

\begin{dfn}{skelet}
\emph{Skeleton of an AB-HO GEC \(\Gg = (V,\kappa)\)} is the collection
\(\FIN{\Gg}\) of all its finite non-empty sub-GEC's. That is:
\[\FIN{\Gg} \df \{(S,\kappa)\dd\ S \subset V,\ 0 < \card(S) < \aleph_0\}.\]
All classes of the form \(\FIN{\Gg}\) (where \(\Gg\) runs over all AB-HO GEC's)
are called briefly \emph{skeletons}.\par
A class \(\ffF\) of GEC's is said to be a \emph{finitary structure} if it
satisfies all the following properties:
\begin{itemize}
\item \(\ffF\) consists of finite (non-empty) GEC's;
\item \(\ffF\) is \emph{hereditary}; that is, it contains all non-empty
 sub-GEC's of all \(\Gg \in \ffF\);
\item \(\ffF\) is \emph{closed under isomorphic copies}; that is, if \(\Gg \in
 \ffF\) and \(\Gg \equiv \Hh\), then \(\Hh \in \ffF\) as well.
\end{itemize}
\emph{Palette} of \(\ffF\), to be denoted by \(\COL{\ffF}\), consists of all
colors which appear as the color of an edge of a GEC belonging to  \(\ffF\).
The class \(\ffF\) is called \emph{proper} if its palette is a set (not a class)
and all vertices of all GEC's from \(\ffF\) are painted the same pseudocolor.
\end{dfn}

As a rule, we will always consider skeletons as finitary structures.\par
According to \THM{embed}, skeletons of AB-HO GEC's are certain finitary
structures that completely determine the GEC's they come from. We can formally
express this determination as follows.

\begin{cor}{iso-AB-HO}
Two AB-HO GEC's \(\Gg = (V,\kappa)\) and \(\Hh = (W,\lambda)\) are isomorphic
iff \(\FIN{\Gg} = \FIN{\Hh}\).
\end{cor}
\begin{proof}
Assuming the skeletons coincide, it follows from \THM{embed} that there are
morphisms \(\phi\dd V \to W\) and \(\psi\dd W \to V\). Then \(\psi \circ \phi\)
and \(\phi \circ \psi\) are morphisms as well and thus, thanks to \PRO{1}, they
are automorphisms, which implies that both \(\phi\) and \(\psi\) are
isomorphisms.
\end{proof}

It is also an immediate consequence of \THM{embed} that an AB-HO GEC \(\Gg\) is
isomorphic to a sub-GEC of an AB-HO GEC \(\Hh\) iff \(\FIN{\Gg} \subset
\FIN{\Hh}\).\par
Our nearest aim is to give a characterisation of classes \(\ffF\) of finite
GEC's such that
\begin{equation}\label{eqn:fin}
\ffF = \FIN{\Gg}
\end{equation}
for some AB-HO GEC \(\Gg\). When \eqref{eqn:fin} holds, we say that \(\ffF\)
\emph{generates} \(\Gg\). (By \COR{iso-AB-HO} we know that \(\Gg\) is uniquely
determined, up to isomorphism, by the above equation.)\par
We begin with a basic observation that each skeleton is a proper class. To state
some other properties of skeletons, we introduce the following

\begin{dfn}{fin-gen}
Let \(\ffF\) be a finitary structure.
\begin{itemize}
\item By \(\grp{\ffF}\) we denote the class of all GEC's \(\Gg\) such that each
 non-empty finite sub-GEC of \(\Gg\) belongs to \(\ffF\). We call the class
 \(\grp{\ffF}\) \emph{generated} by \(\ffF\).
\item If \(\Gg \in \grp{\ffF}\), we say that \(\Gg\) is \emph{finitely
 represented} in \(\ffF\).
\item We denote by \(\grp{\ffF}_{\omega}\) the class of all \textbf{countable}
 members of \(\grp{\ffF}\).
\item More generally, for any infinite cardinal \(\mM\), \(\grp{\ffF}_{\mM}\)
 stands for the class of all members of \(\grp{\ffF}\) that are of size not
 greater than \(\mM\). (So, \(\grp{\ffF}_{\aleph_0} = \grp{\ffF}_{\omega}\).)
\item \(\ffF\) satisfies \emph{monochromatic boundedness principle}, \MBP\ for
 short, if for any color \(c \in \COL{\ffF}\) there exists a positive integer
 \(\mu = \mu(c)\) such that any \(c\)-monochromatic member of \(\ffF\) has size
 not greater than \(\mu\).
\end{itemize}
\end{dfn}

Note that classes of the form \(\grp{\ffF}\) generated by finitary structures
\(\ffF\) are always hereditary and closed under isomorphic copies. Classes
generated by skeletons may simply be described (provided we know the AB-HO GEC's
generated by them)---see item (A) in the following result.

\begin{pro}{fin-gen-AB-HO}
Let \(\ffF\) be the skeleton of an AB-HO GEC \(\Gg\).
\begin{enumerate}[\upshape(A)]
\item For a GEC \(\Hh\) \tfcae
 \begin{enumerate}[\upshape(i)]
 \item \(\Hh \in \grp{\ffF}\);
 \item \(\Hh\) is isomorphic to a sub-GEC of \(\Gg\).
 \end{enumerate}
\item If \(\Hh \in \grp{\ffF}\) and \(\Gg\) admits a morphism into \(\Hh\), then
 \(\Hh\) is isomorphic to \(\Gg\).
\item \(\Gg\) is a unique (up to isomorphism) member of \(\grp{\ffF}\) that
 contains isomorphic copies of all members of that class.
\item \(\ffF\) is proper and satisfies \MBP.
\end{enumerate}
\end{pro}
\begin{proof}
Item (A) follows from \THM{embed}, (C) is a consequence of the combination of
(A) and (B), the first part of (D) is trivial and the second is a repetition of
item (a) in \THM{card}. Therefore we only discuss how to show (B). The proof
goes more or less the same way as that presented in \COR{iso-AB-HO}: if
\(\phi\dd \Gg \to \Hh\) is a morphism, then take (thanks to (A)) any morphism
\(\psi\dd \Hh \to \Gg\) and observe that \(\psi \circ \phi\dd \Gg \to \Gg\) is
necessarily an automorphism (by \PRO{1}). So, \(\psi\) is surjective and we are
done. 
\end{proof}

As we will see in the proof of \THM{skeleton} below, the following result will
be a crucial tool.

\begin{lem}{bdd}
Let \(\ffF\) be a proper finitary structure. If \(\ffF\) satisfies \MBP, then
there is a cardinal number \(\mM\) such that \(|\Gg| \leq \mM\) for any GEC
\(\Gg\) finitely represented in \(\ffF\).
\end{lem}
\begin{proof}
Denote the cardinality of the palette of \(\ffF\) by \(\nN\). It follows from
\THM{r-e-r} that the assertion holds for \(\mM \df \max(\aleph_0,2^{\nN})\) (cf.
the proof of item (b) of \THM{card}).
\end{proof}

Now we pass to amalgamation---another important for us notion, quite well
studied in category theory. To define this notion for GEC's, we also need to
introduce \emph{incomplete} GEC's which we now turn to.

\begin{dfn}{inco}
An \emph{incomplete} GEC's is a quadruple \((V,a,b,\kappa)\) where \(V\) is
a set, \(a\) and \(b\) are two distinct points of \(V\) and \(\kappa\) is any
function defined on \((V \times V) \setminus \{(a,b),(b,a)\}\) (with values in
an arbitrary set) that satisfies (in its domain) conditions (GEC1)--(GEC2) from
\DEF{GEC}. In other words, an incomplete GEC is a ``GEC with a single unpainted
edge'' (in the above situation the unpainted edge coincides with
\(\{a,b\}\)).\par
Let \(\Gg' = (V,a,b,\kappa')\) be an incomplete GEC. By a \emph{completion} of
\(\Gg'\) we mean any GEC \(\Gg = (V,\kappa)\) such that \(\kappa\) extends
\(\kappa'\). In this situation we say that \(\Gg'\) is \emph{completed} to
\(\Gg\). In other words, to complete \(\Gg'\) we only need to paint the edge
\(\{a,b\}\) any color different from the pseudocolor of \(\Gg'\) (more formally,
we need to assign any, but common, value different from \(\kappa'(a,a)\) to both
\((a,b)\) and \((b,a)\)). This new color may not belong to the palette of
\(\Gg'\).
\end{dfn}

Note that if \(\Gg = (V,a,b,\kappa)\) is an incomplete GEC, then both
\begin{equation}\label{eqn:a-b}
\begin{cases}\Gg_a \df (V \setminus \{a\},\kappa)&\\
\Gg_b \df (V \setminus \{b\},\kappa)&\end{cases}
\end{equation}
are GEC's. This observation enables us to introduce the following two
definitions.

\begin{dfn}{symm}
An incomplete GEC \(\Gg = (V,a,b,\kappa)\) is said to be \emph{affiliated} with
a class \(\ffF\) of GEC's if both the GEC's \(\Gg_a\) and \(\Gg_b\) defined in
\eqref{eqn:a-b} belong to \(\ffF\). \(\Gg\) is said to be \emph{symmetric} if
the function \(s_{\Gg}\dd V \setminus \{a\} \to V \setminus \{b\}\) given by
\begin{equation}\label{eqn:aux3}
s_{\Gg}(x) \df \begin{cases}x & x \neq b\\a & x = b\end{cases}
\end{equation}
is an isomorphism between \(\Gg_a\) and \(\Gg_b\). If \(\Gg\) is not symmetric,
it is called \emph{non-symmetric}. Observe that \(\Gg = (V,a,b,\kappa)\) is
non-symmetric iff there exists a vertex \(c \in V \setminus \{a,b\}\) such that
\(\kappa(a,c) \neq \kappa(b,c)\).
\end{dfn}

\begin{dfn}{amalgam}
Let \(\ffF\) be a class of GEC's that is hereditary and closed under isomorphic
copies. We say \(\ffF\) has the \emph{amalgamation property} (briefly: has \AP)
if every non-symmetric incomplete GEC affiliated with \(\ffF\) may be completed
to a GEC that belongs to \(\ffF\).\par
If \(\ffF\) is a finitary structure, we say that:
\begin{itemize}
\item \(\ffF\) has the \emph{countable amalgamation property} (briefly: has
 \AP[\(\omega\)-]) if the class \(\grp{\ffF}_{\omega}\) has \AP;
\item \(\ffF\) has the \emph{\(\mM\)-amalgamation property} (briefly: has
 \AP[\(\mM\)-]), where \(\mM\) is an infinite cardinal, if the class
 \(\grp{\ffF}_{\mM}\) has \AP;
\item \(\ffF\) has the \emph{absolute amalgamation property} (briefly: has
 \AP[\(*\)-]) if the class \(\grp{\ffF}\) has \AP.
\end{itemize}
A class \(\ffF\) of GEC's is called a \textbf{\emph{skeletoid}} if all
the following conditions are fulfilled:
\begin{enumerate}[(SK1)]
\item \(\ffF\) is a finitary structure;
\item \(\ffF\) is proper;
\item \(\ffF\) satisfies \MBP;
\item \(\ffF\) has \AP.
\end{enumerate}
\emph{Finitary structure of a GEC \(\Gg\)} is, by definition, the class
\(\FIN{\Gg}\) of all finite (non-empty) sub-GEC's of \(\Gg\).
\end{dfn}

\begin{rem}{non-sym}
If we consider, in the condition defining \AP\ for \(\ffF\),
a \textbf{symmetric} incomplete GEC \(\Gg = (V,a,b,\kappa)\) affiliated with
\(\ffF\), it may happen (even for skeletons \(\ffF\)!) that there is no
completion of \(\Gg\) that belongs to \(\ffF\). However, its symmetry implies
that for each vertex \(v\) different from both \(a\) and \(b\) both
the edges---\(\{v,a\}\) and \(\{v,b\}\)---are painted the same color, and thus
we can glue (identify) the points \(a\) and \(b\) to obtain a correctly defined
GEC that coincides with both \(\Gg_a\) and \(\Gg_b\) (cf. \eqref{eqn:a-b}) and,
consequently, belongs to \(\ffF\). Such a construction may also be considered as
a (reduced, but still fully correct) amalgamation.
\end{rem}

We will discuss in more detail skeletoids and give examples illustrating
differences between variants of \AP\ introduced in the last definition after
the proof of the next theorem, which is one of the main results of the paper.
It gives a characterisation of skeletons:

\begin{thm}{skeleton}
A class is a skeleton iff it is a skeletoid with \AP[\(*\)-].
\end{thm}
\begin{proof}
First assume \(\ffF = \FIN{\Hh}\) for some AB-HO GEC \(\Hh = (W,\lambda)\).
To show that \(\ffF\) is a skeletoid with \AP[\(*\)-], we only need to check
that it has \AP[\(*\)-]. To this end, fix a non-symmetric incomplete GEC \(\Gg'
= (V,a,b,\kappa')\) such that both the GEC's \(\Gg_a\) and \(\Gg_b\) (given by
\eqref{eqn:a-b}) belong to \(\grp{\ffF}\). It follows from \PRO{fin-gen-AB-HO}
that there are morphisms \(\phi_a\) and \(\phi_b\) from, respectively, \(\Gg_a\)
and \(\Gg_b\) into \(\Hh\). Denote \(V' \df V \setminus \{a,b\}\), \(Z_a \df
\phi_a(V')\) and \(Z_b \df \phi_b(V')\) and note that \[\rho \df \phi_b \circ
(\phi_a^{-1}\restriction{Z_a})\dd Z_a \to Z_b\] is an isomorphism. So, there is
an automorphism \(\eta \in \AUT{\Hh}\) that extends \(\rho\). Then the morphisms
\(\beta \df \eta \circ \phi_a\) and \(\alpha \df \phi_b\) coincide on \(V'\)
(that is, on the common part of their domains). If \(\alpha(a) = \beta(b)\),
then \(\Gg\) would be symmetric (since then \(s_{\Gg} = \alpha \circ
\beta^{-1}\)). Hence \(\alpha(a) \neq \beta(b)\) and, consequently, \(k \df
\lambda(\alpha(a),\beta(b))\) differs from the pseudocolor. So, we may complete
\(\Gg\) (to a GEC \(\tilde{\Gg}\)) by painting its edge \(\{a,b\}\) color \(k\).
To ensure ourselves that \(\tilde{\Gg}\) belongs to \(\grp{\ffF}\), it is enough
to note that the function \(\gamma\dd V \to W\) defined by:
\[\gamma(x) \df \begin{cases}\alpha(x)=\beta(x) & x \in V'\\\alpha(a) & x = a\\
\beta(b) & x = b\end{cases}\] is a well defined morphism of \(\tilde{\Gg}\) into
\(\Hh\). So, the proof of this (simpler) implication is finished.\par
Now we pass to the second part. Fix a skeletoid \(\ffF\) with \AP[\(*\)-]. It
follows from \LEM{bdd} that \(|\Gg| < \mM\) for some (fixed) cardinal number
\(\mM\) and all GEC's \(\Gg\) from \(\grp{\ffF}\). For further purposes, we fix
a set \(Z\) of cardinality \(\mM\). We begin the proof of this part of
the theorem with the following property (which actually is a classical statement
of the amalgamation property):
\begin{itemize}
\item[\((\star)\)] \textit{If \(\Gg_1 = (V_1,\kappa_1)\) and \(\Gg_2 = (V_2,
 \kappa_2)\) are two GEC's from \(\grp{\ffF}\) and \(\kappa_1\) and \(\kappa_2\)
 coincide on \((V_1 \cap V_2) \times (V_1 \cap V_2)\), then there exists a GEC
 \(\Gg_0 = (W_0,\lambda_0) \in \grp{\ffF}\) and morphisms \(\phi_j\dd V_j \to
 W_0\) from \(\Gg_j\) into \(\Gg_0\) \((j=1,2)\) that coincide on \(V_1 \cap
 V_2\).}
\end{itemize}
To prove \((\star)\), we fix any one-to-one function \(\phi_2\dd V_2 \to Z\).
For simplicity, denote \(Q \df \phi_2(V_2)\), \(V_0 \df V_1 \cap V_2\) and
\(\alpha \df \phi_2\restriction{V_0}\dd V_0 \to Q\). It is clear that there
exists a function \(\rho_Q\) (on \(Q \times Q\)) such that \(\Qq = (Q,\rho_Q)\)
is a GEC from \(\grp{\ffF}\) and \(\phi_2\) is an isomorphism from \(\Gg_2\)
onto \(\Qq\). Since \(\kappa_1\) and \(\kappa_2\) coincide on \(V_0\),
\(\alpha\) is a partial morphism from \(\Gg_1\) into \(\Qq\). It follows from
Zorn's lemma that in the collection of all pairs \((\beta,\Rr)\) where:
\begin{itemize}
\item \(\Rr = (R,\rho_R)\) is a GEC from \(\grp{\ffF}\);
\item \(Q \subset R \subset Z\) and \(\rho_R\) extends \(\rho_Q\);
\item \(\beta\dd W \to R\) with \(V_0 \subset W \subset V_1\) is a partial
 morphism from \(\Gg_1\) into \(\Rr\) that extends \(\alpha\)
\end{itemize}
(with the partial order of inclusion and extension) there exists a maximal
element, say \((\gamma,\Ss)\) where \(\Ss = (S,\mu)\) and \(\gamma\dd W \to S\).
All we need to show is that \(W = V_1\). We argue by a contradiction. Assume
there is \(v \in V\) that does not belong to \(W\) (in particular, \(v \notin
V_2\)). Since \(\Ss \in \grp{\ffF}\), \(\card(S) < \mM = \card(Z)\) and hence
there exists a point \(z \in Z\) that is not in \(S\). Denote \(T = \gamma(W)\)
and consider a (unique) function \(\nu\) on \((T \cup \{z\}) \times (T \cup
\{z\})\) such that \(\Tt \df (T \cup \{z\},\nu)\) is a GEC and the function
\(\tilde{\gamma}\dd W \cup \{v\} \to T \cup \{z\}\) that extends \(\gamma\) and
satisfies \(\tilde{\gamma}(v) = z\) is an isomorphism of \((W \cup \{v\},
\kappa_1)\) onto \(\Tt\). Then \(\Tt \in \grp{\ffF}\). Since \(\gamma\) is
a morphism of \((W,\kappa_1)\) into \(\Ss\), we infer that \(\nu\) and \(\mu\)
coincide on \(T \times T\). It suffices to extend both these functions to
a single function \(\lambda\) on \((S \cup \{z\}) \times (S \cup \{z\})\) in
a way such that \(\tilde{\Ss} \df (S \cup \{z\},\lambda) \in \grp{\ffF}\)
(because then the pair \((\tilde{\gamma},\tilde{\Ss})\) witnesses that the pair
\((\gamma,S)\) is not maximal). To this end, we need to remember one additional
property:
\begin{itemize}
\item[\((\diamondsuit)\)] for any \(s \in S \setminus T\) there is \(t \in T\)
 such that \(\nu(z,t) \neq \mu(s,t)\).
\end{itemize}
Indeed, if \((\diamondsuit)\) were false for \(s \in S \setminus T\), we could
extend (maximal!) \(\gamma\) by sending \(v\) to \(s\) (in this way we would
obtain a morphism, because \(\tilde{\gamma}\) is a morphism into \((T \cup
\{z\},\nu)\) and the equation \(\nu(z,t) = \mu(s,t)\) would hold for all \(t \in
T\)). (It is worth recalling here that all vertices of all GEC's from
\(\grp{\ffF}\) are painted the same color and therefore the equation
\eqref{eqn:morphism} defining a morphism can be verified only for distinct
arguments.)\par
Having \((\diamondsuit)\), we will be able to apply \AP\ in a moment. Similarly
as in the previous step of the proof, it follows from Zorn's lemma that in
the collection of all functions \(\xi\) such that:
\begin{itemize}
\item \(\xi\) is defined on a set of the form \((C \times C)\) where \(T \cup
 \{z\} \subset C \subset S \cup \{z\}\);
\item \((C,\xi)\) (where \(C\) is as specified above) belongs to \(\grp{\ffF}\);
\item \(\xi\) extends \(\nu\) and coincides with \(\mu\) on \((C \cap S) \times
 (C \cap S)\)
\end{itemize}
there exists a maximal element, say \(\lambda\) defined on \(P \times P\). To
finish the proof of \((\star)\), we only need to show that \(P = S \cup \{z\}\).
Assume, again, that this is not the case and take any \(s \in S \setminus P\).
Then \(s \notin T\) and therefore there exists (by \((\diamondsuit)\)) \(t_0 \in
T\) such that
\begin{equation}\label{eqn:aux4}
\nu(z,t_0) \neq \mu(z,t_0).
\end{equation}
Denote by \(U\) the set \(P \cup \{t_0\}\) and by \(\theta\) a unique function
on \((U \times U) \setminus \{(z,t_0),(t_0,z)\}\) that extends \(\lambda\) and
coincides with \(\mu\) on \((U \cap S) \times (U \cap S)\). Observe that \((U,z,
t_0,\theta)\) is an incomplete GEC. It follows from \eqref{eqn:aux4} that it is
non-symmetric. Moreover, it satisfies all the assumptions postulated in
the condition that defines \AP\ (for \(\grp{\ffF}\)). So, it can be completed
to a member of \(\grp{\ffF}\). In other words, we are able to extend \(\lambda\)
and in this way we obtain a contradiction with its maximality, which finally
finishes the proof of \((\star)\).\par
Condition \((\star)\) is simply equivalent to its category-theoretic
counterpart:
\begin{itemize}
\item[\((\star')\)] \textit{If \(\Gg_0,\Gg_1,\Gg_2\) belong to \(\grp{\ffF}\)
 and \(\phi_j\) \UP(for \(j=1,2\)\UP) is a morphism of \(\Gg_0\) into \(\Gg_j\),
 then there exists a GEC \(\Ff \in \grp{\ffF}\) and two morphisms \(\psi_1\) and
 \(\psi_2\) from, respectively, \(\Gg_1\) and \(\Gg_2\) into \(\Ff\) such that
 \(\psi_1 \circ \phi_1 = \psi_2 \circ \phi_2\).}
\end{itemize}
We leave a formal proof of the above condition to the reader.\par
Now we are ready to give a proof of the theorem. For the last time (in this
proof) we make use of the Zorn's lemma to conclude that among all pairs
\((V,\kappa)\) such that:
\begin{itemize}
\item \(V \subset Z\) and \(\kappa\) is a function on \(V \times V\);
\item \((V,\kappa)\) is a GEC belonging to \(\grp{\ffF}\)
\end{itemize}
there exists a maximal element, say \(\Hh = (W,\lambda)\). We will show that it
is an AB-HO GEC such that \(\FIN{\Hh} = \ffF\). Since \(\Hh \in \grp{\ffF}\), we
have \(\FIN{\Hh} \subset \ffF\). First of all, note that \(Z \neq W\) (because
of the choice of \(\mM\)). This relation, combined with the maximality of
\(\Hh\), implies that:
\begin{itemize}
\item[\((*)\)] \textit{Any morphism of \(\Hh\) into a member of \(\grp{\ffF}\)
 is an isomorphism.}
\end{itemize}
Indeed, if \(\Gg = (V,\kappa)\) belongs to \(\grp{\ffF}\) and a morphism
\(\phi\dd W \to V\) from \(\Hh\) into \(\Gg\) were not an isomorphism, then we
could take points \(v \in V \setminus \phi(W)\) and \(z \in Z \setminus W\) and
transport the GEC structure of \((\phi(W) \cup \{v\},\kappa)\) to the set
\(W \cup \{z\}\) via a bijection that extends \(\phi\) and sends \(z\) to \(v\)
to obtain a GEC \((W \cup \{z\},\tilde{\kappa}) \in \grp{\ffF}\) which would
witness that \(\Hh\) was not maximal.\par
Now it follows from the combination of \((\star')\) and \((*)\) that
\begin{itemize}
\item[\((**)\)] \textit{For any \(\Gg \in \grp{\ffF}\), each partial morphism of
 \(\Gg\) into \(\Hh\) extends to a morphism of \(\Gg\) into \(\Hh\).}
\end{itemize}
Indeed, if \(\Gg = (V,\kappa)\) belongs to \(\grp{\ffF}\), \(V_0 \subset V\) and
\(\phi_1\dd V_0 \to W\) is a morphism of \(\Gg_0 \df (V_0,\kappa)\) into
\(\Hh\), then it follows from \((\star')\) (where, in addition to the present
settings, \(\Gg_1 = \Hh\), \(\Gg_2 = \Gg\) and \(\phi_2\) is the inclusion map)
that there is a GEC \(\Ff \in \grp{\ffF}\) and two morphisms \(\psi_1\) and
\(\psi_2\) from (respectively) \(\Hh\) and \(\Gg\) into \(\Ff\) such that
\(\psi_1 \circ \phi_1 = \psi_2\restriction{V_1}\). Then \(\psi_1\) is
an isomorphism (thanks to \((*)\)) and thus \(\psi_1^{-1} \circ \psi_2\) is
the extension of \(\phi_1\) we searched for.\par
It is now immediate that \(\Hh\) is AB-HO, since any its partial morphism
extends to a morphism (by \((**)\)), which actually is an automorphism (thanks
to \((*)\)). Finally, any member \(\Ff\) of \(\ffF\) admits a morphism into
\(\Hh\) (again by \((**)\)---it is sufficient to extend a partial morphism
defined on the empty sub-GEC, or on a one-point sub-GEC of \(\Ff\)) and hence
\(\ffF \subset \FIN{\Hh}\), which finishes the whole proof of the theorem.
\end{proof}

Looking at axioms (SK1)--(SK4) defining skeletoids, we see that \AP\ is the only
condition that can be difficult to check in practice---especially when working
with strictly concrete examples of classes of GEC's (however, it is hard to take
care of \MBP\ during the attempts to construct a new skeletoid and to force
\AP). For a better understanding of this property, we introduce the following
concept.\par
Let \(\ffF\) be a proper finitary structure. Consider a non-symmetric incomplete
GEC \(\Gg = (V,a,b,\kappa)\) affiliated with \(\ffF\) (in particular, \(V\) is
finite). Let \(\Delta_{\ffF}(\Gg)\) denote the \underline{set} of all possible
colors \(c\) such that \(\Gg\) may be completed to a GEC from \(\ffF\) by
painting the edge \(\{a,b\}\) color \(c\) (that is, by defining \(\kappa(a,b) =
\kappa(b,a) \df c\)). (This is a set because \(\ffF\) is proper and necessarily
\(\Delta_{\ffF}(\Gg) \subset \COL{\ffF}\).) Without any other assumptions,
the set \(\Delta_{\ffF}(\Gg)\) may be empty. In fact, non-emptiness of this set
for \underline{each} non-symmetric incomplete GEC affiliated with \(\ffF\) is
equivalent to \AP\ for that class. So, we obtain

\begin{pro}{AP}
A proper finitary structure \(\ffF\) has \AP\ iff
\begin{equation}\label{eqn:AP}
\Delta_{\ffF}(\Gg) \neq \varempty
\end{equation}
for all non-symmetric incomplete GEC's \(\Gg\) affiliated with \(\ffF\).
\end{pro}

Now let \(\ffF\) be a proper finitary structure that has \AP. Further, let \(\Gg
= (V,a,b,\kappa)\) be a non-symmetric incomplete GEC affiliated with
\(\grp{\ffF}\). We infer there exists a vertex \(c \in V\) different from \(a\)
and \(b\) such that \(\kappa(a,c) \neq \kappa(b,c)\). In particular, for any
finite set \(S \subset V\) such that \(a,b,c \in S\) the quadruple
\(\Gg\restriction{S} \df (S,a,b,\kappa)\) is a non-symmetric incomplete GEC
affiliated with \(\ffF\) (not only with \(\grp{\ffF}\)). So, it follows from
\AP\ that \(\Delta_{\ffF}(\Gg\restriction{S}) \neq \varempty\). Moreover,
\(\Gg\) may be completed to a member of \(\grp{\ffF}\) \iaoi
\begin{equation}\label{eqn:centr}
\bigcap_S \Delta_{\ffF}(\Gg\restriction{S}) \neq \varempty
\end{equation}
where \(S\) runs over all finite subsets of \(V\) that contain \(a,b,c\). Notice
also that the family \(\{\Delta_{\ffF}(\Gg\restriction{S})\dd\ a,b,c \in S
\subset V,\ S \UP{ finite}\}\) is centred, simply because
\[\bigcap_{k=1}^n \Delta_{\ffF}(\Gg\restriction{S_k}) \supset
\Delta_{\ffF}\Bigl(\Gg\restriction{\bigcup_{k=1}^n S_k}\Bigr).\]
So, condition \eqref{eqn:centr} has the spirit of compactness. And,
in particular, we have just proved the following result.

\begin{pro}{comp}
Let \(\ffF\) be a skeletoid whose palette is topologised so that it is
a Hausdorff space. If \(\Delta_{\ffF}(\Gg)\) is a compact set in
\(\COL{\ffF}\) for any non-symmetric incomplete GEC \(\Gg\) affiliated with
\(\ffF\), then \(\ffF\) is a skeleton.
\end{pro}

\begin{exm}{metricoid}
For simplicity, we say a skeletoid is \emph{metric} if it consists of metric
spaces. Let \(\Gg = (X,a,b,d)\) be a non-symmetric incomplete GEC affiliated
with a metric skeletoid \(\ffF\). Assume we have extended \(d\) to a metric
(still denoted by \(d\)) on the whole set \(X\). It follows from the triangle
inequality that \(d(a,b) \in [m(\Gg),M(\Gg)]\) where \(m(\Gg) \df
\max\{|d(a,c)-d(b,c)|\dd\ c \in X \setminus \{a,b\}\}\ (> 0)\) and \(M(\Gg) \df
\min\{d(a,c)+d(b,c)\dd\ c \in X \setminus \{a,b\}\}\ (< \infty)\). So,
\[\Delta_{\ffF}(\Gg) \subset [m(\Gg),M(\Gg)]\]
for any non-symmetric incomplete GEC \(\Gg\) affiliated with metric skeletoid
\(\ffF\). Therefore, it follows from \PRO{comp} that if \(\Delta_{\ffF}(\Gg)\)
is closed in \(\RRR\) for all such \(\Gg\), then \(\ffF\) is a skeleton.\par
Although the above example describes a convenient method of constructing
metric skeletons (and consequently absolutely homogeneous metric spaces), it is,
unfortunately, quite hard to find examples of skeletoids with closed sets
\(\Delta_{\ffF}(\Gg)\). More on that topic will be said in Subsection
\ref{sec:metric} below.
\end{exm}

The ideas presented in the paragraph preceding \PRO{comp} enable us to give
characterisations of all variants of amalgamation (introduced in \DEF{amalgam}).
They read as follows:

\begin{pro}{*AP}
Let \(\ffF\) be a skeletoid.
\begin{enumerate}[\upshape(A)]
\item \(\ffF\) has \AP[\(\omega\)-] iff the following condition holds:
 \begin{quote}
 Whenever \(\kappa\) is a function defined on \((\NNN \times \NNN) \setminus
 \{(0,1),(1,0)\}\) so that \(\kappa(0,2) \neq \kappa(1,2)\) and
 \((\{k,2,\ldots,n\},\kappa) \in \ffF\) for any \(0 \leq k \leq 1 < n\), then
 \(\bigcap_{n=2}^{\infty} \Delta_{\ffF}(\{0,1,\ldots,n\},0,1,\kappa) \neq
 \varempty\).
 \end{quote}
\item \(\ffF\) has \AP[\(\alpha\)-], where \(\alpha\) is an infinite cardinal,
 iff the following condition is fulfilled:
 \begin{quote}
 Whenever \(Z\) is an infinite set of cardinality not greater than \(\alpha\),
 \(a,b,c\) are three distinct points of \(Z\) and \(\kappa\) is a function
 defined on \((Z \times Z) \setminus \{(a,b),(b,a)\}\) so that \(\kappa(a,c)
 \neq \kappa(b,c)\) and \((\{z,c\} \cup S,\kappa) \in \ffF\) for any \(z \in
 \{a,b\}\) and each finite subset \(S\) of \(Z\), then
 \[\bigcap_S \Delta_{\ffF}(\{a,b,c\} \cup S,a,b,\kappa) \neq \varempty\]
 where \(S\) runs over all finite subsets of \(Z\).
 \end{quote}
\item If \(\card(\COL{\ffF}) = \mM \geq \aleph_0\), then \(\ffF\) is a skeleton
 iff \(\ffF\) has \AP[\(\mM\)-].
\end{enumerate}
\end{pro}
\begin{proof}
Items (A) and (B) are left to the reader. Here we will show only (C). Thanks to
\THM{skeleton}, we only need to check that \AP[\(*\)-] follows from
\AP[\(\mM\)-]. To this end, fix an arbitrary non-symmetric incomplete GEC \(\Gg
= (V,a,b,\kappa)\) affiliated with \(\grp{\ffF}\) and assume that \(\Gg\) cannot
be completed to a member of \(\grp{\ffF}\). Choose \(c \in V\) different from
\(a\) and \(b\) such that \(\kappa(a,c) \neq \kappa(b,c)\) and denote, for
simplicity, by \(\Sigma\) the family of all finite subsets of \(V\) that contain
all \(a\), \(b\) and \(c\). Our assumption about \(\Gg\) implies that
\[\bigcap_{\sigma\in\Sigma} \Delta_{\ffF}(\sigma,a,b,\kappa) = \varempty.\]
So, for each \(d \in Y \df \COL{\ffF}\) there is \(\sigma_d \in \Sigma\) such
that \(d \notin \Delta_{\ffF}(\sigma_d,a,b,\kappa)\). Now put \(W \df
\bigcup_{d \in Y} \sigma_d\) and observe that \(\card(W) \leq \mM\) and
\(\Delta_{\ffF}(W,a,b,\kappa) \subset \bigcap_{d \in Y}
\Delta_{\ffF}(\sigma_d,a,b,\kappa) = \varempty\), which shows that \(\ffF\) does
not have \AP[\(\mM\)-] and finishes the proof.
\end{proof}

We repeat a special case of the above proposition in a separate result as it is
worth special emphasis:

\begin{cor}{countable}
A skeletoid with countable palette is a skeleton iff it has \AP[\(\omega\)-].
\end{cor}

Observe that a skeletoid with countable palette is a countable collection (after
identifying isomorphic GEC's) with \AP\ and as such it falls within
the framework of Fra\"{\i}ss\'{e} theory \cite{fra}. In particular, even if it
is not a skeleton, it has the so-called \emph{Fra\"{\i}ss\'{e} limit} which in
this context is a countable ultrahomogeneous GEC \(\Gg\) such that
\(\FIN{\Gg}\) coincides with the given skeletoid (and these properties uniquely
determine \(\Gg\) up to isomorphism). It is an interesting question of what are
the ties between the above \(\Gg\) induced by a skeleton and the AB-HO GEC
generated by this skeleton. This issue will be studied by the author in
a subsequent paper. At this moment we only wish to underline that, in general,
these two GEC's do not coincide as there exist uncountable AB-HO GEC's with
countable palettes (see, e.g. \COR{2toY} below).

\begin{rem}{*AP}
Criteria for variants of \AP\ formulated in \PRO{*AP} may be reformulated in
a more intrinsic way (from the point of view of the skeletoid under
consideration), in category-theoretic terms as follows. Define a morphism
between two incomplete GEC's as a one-to-one function that preserves colors of
edges as well as preserves the unpainted edges. Then the \emph{completion} of
a direct system \SYS{\Gg_{\alpha}}{\pi_{\alpha,\beta}}{\alpha\leq\beta} of
non-symmetric incomplete GEC's consists of assigning a single color to all
upainted edges in all the GEC's that appear in that system. Now we have:
\begin{itemize}
\item A skeletoid \(\ffF\) has \AP[\(\omega\)-] iff any direct sequence
 \[\Gg_0 \stackrel{\pi_0}{\longrightarrow} \Gg_1
 \stackrel{\pi_1}{\longrightarrow} \Gg_2 \ldots\]
 of non-symmetric incomplete (necessarily finite) GEC's affiliated with \(\ffF\)
 can be completed to a system consisting of members of \(\ffF\).
\item A skeletoid \(\ffF\) has \AP[\(\mM\)-] (where \(\mM \geq \aleph_0\)) iff
 any direct system
 \[\SYS{\Gg_{\alpha}}{\pi_{\alpha,\beta}}{\alpha\leq\beta}\] consisting of at
 most \(\mM\) non-symmetric incomplete GEC's affiliated with \(\ffF\) can be
 completed to a direct system consisting of members of \(\ffF\).
\end{itemize}
In particular, a skeletoid \(\ffF\) has \AP[\(\omega\)-] iff for any sequence
\[\Gg_1 \subsetneq \Gg_2 \subsetneq \Gg_3 \ldots\] of non-symmetric incomplete
GEC's affiliated with \(\ffF\) (where \(\subsetneq\) in the above notation means
that there exists a morphism that is not an isomorphism) one has:
\[\bigcap_{n=1}^{\infty} \Delta_{\ffF}(\Gg_n) \neq \varempty.\]
Unfortunately, there is no similar way of `forgetting' the morphisms in
uncountable direct systems of incomplete GEC's---so, e.g., the above criterion
for \AP[\(\mM\)-] with \(\mM > \aleph_0\) cannot be simplified in a similar
fashion. This is one of main reasons why we put a special emphasis on
\AP[\(\omega\)-].\par
We leave the details to interested readers.
\end{rem}

\begin{exm}{non-*-AP}
Here we give a simple (in its definition) example of a class of skeletoids
\(\ffF_{\alpha}\) that are not skeletons but satisfy \AP[\(\mM\)-] for any
infinite cardinal \(\mM < \alpha\) where \(\alpha\) is the size of the palette
of \(\ffF_{\alpha}\) (cf. item (C) of \PRO{*AP}). To this end, let \(\alpha\) be
an arbitrary infinite cardinal number and fix a set \(Y_{\alpha}\) of size
\(\alpha\) (to be a palette). (We also fix some pseudocolor not within the set
\(Y_{\alpha}\).) Let \(\ffF_{\alpha}\) be the class of all finite (non-empty)
GEC's that satisfy all the following three conditions:
\begin{enumerate}[(ex1)]
\item all vertices are painted the pseudocolor;
\item all edges have colors from \(Y_{\alpha}\);
\item any two distinct edges having a common vertex have different color.
\end{enumerate}
Observe that this class satisfies (SK1)--(SK3) (in fact, it contains no
monochromatic sub-GEC of size 3). To verify \AP\ and its variants, first notice
that \(\grp{\ffF_{\alpha}}\) consists of all the GEC's that satisfy (ex1)--(ex3)
and then fix a non-symmetric incomplete GEC \(\Gg\) affiliated with
\(\grp{\ffF_{\alpha}}\) whose size is less than \(\alpha\). Then there is at
least one color \(c\) from \(Y_{\alpha}\) that has not been used to paint
the edges of \(\Gg\). If we paint the unpainted edge of \(\Gg\) color \(c\), we
obtain a GEC satisfying (ex1)--(ex3) and hence belonging to
\(\grp{\ffF_{\alpha}}\). This shows that \(\ffF_{\alpha}\) is a skeletoid that
has \AP[\(\mM\)-] for all \(\mM < \alpha\). We will now show that this class is
not a skeleton. To this end, fix three distinct points \(a, b, c\) from \(V \df
Y_{\alpha}\) and a bijection \(\mu\dd V \setminus \{a,b\} \to Y_{\alpha}\). We
will paint all the edges, apart from \(\{a,b\}\), of the full graph \(\Gamma\)
with vertices in \(V\) so that we will obtain a non-symmetric incomplete GEC
satisfying (ex1)--(ex3). We begin the construction with painting the edges
emanating from \(a\) (but different from \(\{a,b\}\)) by declaring that an edge
\(\{a,x\}\) (with \(x \notin \{a,b\}\)) has color \(\mu(x)\). In this way, from
vertex \(a\) emanate edges in all possible colors so that there will be no way
to complete our final incomplete GEC to a member of \(\grp{\ffF}\). Our next
step is to \underline{initially} well-order the set of all the edges of
\(\Gamma\) that do not emanate from \(a\) in a way such that the first edge
(with respect to this order) is \(\{b,c\}\). We paint this first edge any color
from \(Y_{\alpha}\) different from \(\mu(c)\). In this way, our final incomplete
GEC will be non-symmetric. The last step of the construction relies on
transfinite induction: assuming \(\{x,y\}\) with \(x \neq a \neq y\) is an edge
of \(\Gamma\) such that all preceding edges (not emanating from \(a\)) with
respect to our well-order have already been painted, we choose for \(\{x,y\}\)
any color that until this moment has not been used to paint edges not emanating
from \(a\) and differs from \(\mu(x)\) and \(\mu(y)\) (if only \(x\) or,
respectively, \(y\) is distinct from \(b\)). There exists such a color because
the well-order is initial (and there are only \(\alpha\) edges). This method
ensures that (ex3) holds at each step of construction. So, \(\ffF_{\alpha}\)
indeed is not a skeleton. In particular, \AP\ does not imply \AP[\(\omega\)-]
and properties \AP[\(\alpha\)-] (for various \(\alpha \geq \aleph_0\)) are
pairwise non-equivalent.
\end{exm}

The above example deals with skeletoids consisting of certain GEC's in which all
the edges emanating from a single vertex have different colors. Below we give
a full characterisation (as well as a full classification) of all AB-HO GEC's
with the same property. To this end, we recall that a group \((G,\cdot)\) is
\emph{Boolean} if \(a^2 = e\) (where \(e\) is the neutral element of \(G\)) for
all \(a \in G\). Each Boolean group is Abelian and actually a full
classification of such groups is well-known: since all they are vector spaces
over the field \(\FFF_2 \df \ZZZ / 2\ZZZ\), each Boolean group is isomorphic to
the direct product of a certain (finite or not) number of copies of this field.
Below, for a Boolean group \(G\) with neutral element \(e\), the palette of
\(G\) coincides with \(G_* \df G \setminus \{e\}\) and \(e\) is reserved to be
the pseudocolor (of vertices of \(G\)).

\begin{thm}{skton-grp}
\begin{enumerate}[\upshape(A)]
\item For any Boolean group \((G,+)\), the pair \((G,+)\) (where the group
 operation \(+\) is naturally treated as a function of \(G \times G\) into
 \(G\)) is an AB-HO GEC with the following property:
 \begin{equation}\label{eqn:diff}
 \textit{All the edges emanating from a single vertex have different colors.}
 \end{equation}
 Moreover, \(\AUT{G,+}\) consists precisely of all translations in this group.
 In particular, each its partial morphism defined on a non-empty sub-GEC of
 \((G,+)\) admits a unique extension to an automorphism of this GEC.
\item Two Boolean groups are structurally equivalent as GEC's iff they are
 isomorphic as groups. Moreover, structural isotopies between them are precisely
 group isomorphisms.
\item Every non-empty GEC with the property that each its partial morphism
 (defined on a non-empty sub-GEC) admits a unique extension to an automorphism
 is isomorphic to a Boolean group.
\item Every non-empty AB-HO GEC with property \eqref{eqn:diff} is isomorphic to
 a Boolean group.
\end{enumerate}
\end{thm}
\begin{proof}
(A): Since \(G\) is Boolean, it is clear that \((G,+)\) is a GEC in which each
translation is an automorphism. Now if \(u\dd A \to G\) is a morphism where
\(A\) is a non-empty subset of \(G\), then, after fixing \(a \in A\), we have
for all \(x \in A\), \(u(x)+u(a) = x+a\) and therefore \(u(x) = x+b\ (x \in A)\)
where \(b \df u(a)+a\). This shows that \(G\) is AB-HO and \(u\) has a unique
extension to an automorphism.\par
(B): To avoid repetitions, we postpone the proof of this part and refer
the reader to part (g) of \PRO{dgroup} in Section~\ref{sec:directed} where we
will show a counterpart of (B) for all groups.\par
Now fix a GEC \(\Gg = (V,\kappa)\) with the properties specified in (C) and fix
its vertex \(e\). Observe that then \(\Gg\) is automatically AB-HO. Put \(G \df
\AUT{\Gg}\) and consider a function
\[\Phi\dd G \ni \phi \mapsto \phi(e) \in V.\]
Since \(\Gg\) is homogeneous, \(\Phi\) is surjective. Moreover, the property
formulated in (C) implies that \(\Phi\) is also one-one-to. So, it is
a bijection. So, we can transport the group operation via \(\Phi\) from \(G\)
to \(V\) as follows:
\[v * w \df \Phi(\Phi^{-1}(v) \circ \Phi^{-1}(w)).\]
In this way \((V,*)\) is a group (with neutral element \(e\)) and \(\Phi\) is
a group isomorphism. We claim that \(\Gg\) has property \eqref{eqn:diff}.
Indeed, if \(x \neq z \neq y\) are vertices of \(\Gg\) and satisfy \(\kappa(x,z)
= \kappa(y,z)\), then (since \(\Gg\) is AB-HO) there exists \(\psi \in
\AUT{\Gg}\) such that \(\psi(x) = \psi(y)\) and \(\psi(z)\). In particular,
\(\psi\) extends the identity morphism on \(\{z\}\) and the property stated in
(C) implies that \(\psi\) is the identity map. So, \(x = y\) and condition
\eqref{eqn:diff} is fulfilled.\par
Now fix for a moment arbitrary \(a \in V\) different from \(e\) and look at
\(\phi \df \Phi^{-1}(a)\): \(\kappa(a,e) = \kappa(\phi(a),\phi(e)) =
\kappa(\phi(a),a)\). We infer from \eqref{eqn:diff} that \(\phi(a) = e\).
Consequently, \(\Phi(\phi \circ \phi) = e\) and hence \(\phi \circ \phi\) is
the identity map. This implies that both the groups \(\AUT{\Gg}\) and \((V,*)\)
are Boolean. Further, observe that, for any \(x,y,z \in V\):
\begin{equation}\label{eqn:xyz}
z = x * y \iff \kappa(z,e) = \kappa(x,y).
\end{equation}
To show the above equivalence, denote \(\phi_w \df \Phi^{-1}(w) \in \AUT{\Gg}\)
for \(w \in \{x,y,z\}\) and notice that:
\[z = x*y \iff y = x*z \iff \phi_y = \phi_x \circ \phi_z \iff y = \phi_x(z)\]
(the last equivalence follows from the fact that \(\Phi\) is one-to-one) and:
\[\kappa(e,z) = \kappa(x,y) \iff \kappa(\phi_x(e),\phi_x(z)) = \kappa(x,y) \iff
y = \phi_x(z)\]
where the last equivalence follows from \eqref{eqn:diff} and the equation
\(x = \phi_x(e)\). So, \eqref{eqn:xyz} is fulfilled. Now we pass to the last
step of the proof. Denote by \(E\) the pseudocolor of \(\Gg\) and put \(G \df
\COL{\Gg} \cup \{E\}\). We conclude from the homogeneity of \(\Gg\) and
\eqref{eqn:diff} that the function \(\Psi\dd V \ni x \mapsto \kappa(x,e) \in G\)
is a bijection. Transporting the group structure of \(V\) to \(G\) via \(\Psi\)
we make \((G,+)\) a Boolean group such that \(\Psi\) is a group isomorphism. In
particular, thanks to \eqref{eqn:xyz}, for any \(x, y \in V\):
\[\kappa(x,y) = \kappa(x*y,e) = \Psi(x*y) = \Psi(x)+\Psi(y),\]
which means that \(\Psi\) is also an isomorphism of GEC's and finishes the proof
of (C).\par
(D): It is sufficient to show that every AB-HO GEC satisfying \eqref{eqn:diff}
has the property specified in (C), which is left to the reader as a simple
exercise.
\end{proof}

The above result, among other things, shows that each Boolean group possesses
a \underline{natural} structure of an AB-HO GEC. In Section \ref{sec:directed}
we will see that, in a quite analogous manner, all groups have (also natural)
structures of the so-called AB-HO directed GEC's. And with respect to these
structures precisely Boolean groups become (undirected) GEC's.\par
Below we present a simple way of constructing new AB-HO GEC's from other AB-HO
GEC's. This result is inspired by a (quite simple) observation that spheres in
AB-HO metric spaces are also AB-HO.

\begin{pro}{sphere}
Let \(\Gg = (V,\kappa)\) be an AB-HO GEC and \(W \subset V\) be an arbitrary
non-empty set. Finally, let \(r\dd W \to \COL{\Gg}\) be any function and
\[S(W,r) \df \{x \in V|\quad \forall w \in W\dd\ \kappa(w,x) = r(w)\}.\]
Then \((S(W,r),\kappa)\) is a (possibly degenerate) AB-HO GEC.
\end{pro}

(Note that, in case of metric spaces the above set \(S(W,r)\) is just
the intersection of a non-empty collection of spheres.)

\begin{proof}[Proof of \PRO{sphere}]
Take any \(A \subset S \df S(W,r)\) and a morphism \(\phi\dd A \to S\). Then
it can be `canonically' extended to \(\psi\dd A \cup W \to V\) by declaring that
\(\psi(w) = w\) for any \(w \in W\). It follows from the definition of \(S\)
that \(\psi\) is a partial morphism of \(\Gg\) and as such extends to
an automorphism \(\Phi \in \AUT{\Gg}\). Since \(\Phi\) fixes all the points of
\(W\), we get \(\Phi(S) = S\) and therefore \(\Phi\restriction{S} \in
\AUT{S,\kappa}\) is the extension of \(\phi\) we searched for.
\end{proof}

\subsection{Almost absolute homogeneity and other high levels of homogeneity}
For an infinite cardinal number \(\alpha\), let us call a GEC \(\Gg =
(V,\kappa)\) \emph{\(\alpha\)-homogeneous} if all its partial morphism defined
on sub=GEC's of size less than \(\alpha\) extend to automorphisms of \(\Gg\).
\(\Gg\) is said to be \emph{almost AB-HO} if it is \(\mM\)-homogeneous with
\(\mM = \card(V)\). In particular, \(\Gg\) is \(\aleph_0\)-homogeneous iff it is
finitely homogeneous (cf. \DEF{abshm}).\par
Inspections of the proof of \THM{embed} as well as of the first part of
the proof of \THM{skeleton} show that the following two properties can be proved
only with small changes in arguments presented therein (below \(\alpha\) is
an infinite cardinal number):
\begin{itemize}
\item If a GEC \(\Gg\) is \(\alpha\)-homogeneous, then a GEC \(\Hh\) with
 \(|\Hh| \leq \alpha\) admits a morphism into \(\Gg\) iff  it is finitely
 represented in \(\FIN{\Gg}\).
\item In particular, if \(\Gg\) is almost AB-HO, then a GEC \(\Hh\) admits
 a morphism into \(\Gg\) iff it is finitely represented in \(\FIN{\Gg}\) and
 \(|\Hh| \leq |\Gg|\).
\item If \(G\) is \(\alpha\)-homogeneous, then the class \(\FIN{\Gg}\) has \AP\
 and \AP[\(\mM\)-] for any infinite \(\mM < \alpha\).
\end{itemize}
The last property naturally corresponds to \AP[\(*\)-]. However, there is a one
main difference between skeletons and finitary structures of, for example,
almost AB-HO GEC's: namely, in the former class \MBP\ is automatically satisfied
and can fail in the latter. Below we give four classical examples of finitary
structures each of which has \AP[\(*\)-], but fails to satisfy \MBP. We also
discuss aspects related to the existence of highly homogeneous GEC's whose
finitary structures coincide with the given.

\begin{exm}{metric-Hilbert}
\begin{enumerate}[\upshape(A)]
\item We begin with a finitary structure \(\mmM\) of all finite (non-empty)
 metric spaces. Of course, it does not satisfy \MBP. But, as is well-known (and
 actually may easily be proved), \(\mmM\) has \AP[\(*\)-], which follows from
 the property that \(m(\Gg) \leq M(\Gg)\) (cf. \EXM{metricoid}) for any
 non-symmetric incomplete GEC \(\Gg = (X,a,b,d)\) affiliated with \(\mmM\) where
 \(m(\Gg) \df \sup\{|d(x,a)-d(x,b)|\dd\ x \in X \setminus \{a,b\}\}\) and
 \(M(\Gg) \df \inf\{d(x,a)+d(x,b)\dd\ x \in X \setminus \{a,b\}\}\). As we
 mentioned in Introduction, the Urysohn universal metric space \(\UUU\) is
 an \(\aleph_0\)-homogeneous separable complete metric such that \(\FIN{\UUU} =
 \mmM\). Because of its separability, it is close to be almost AB-HO. Since it
 contains an isometric copy of the countably infinite discrete metric space, it
 cannot have the property of extending partial isometries defined on
 \underline{closed} countable subsets (to global). Up to now, \(\UUU\) is
 a unique known example in ZFC of a metric space containing isometric copies of
 all finite spaces and with the property that its cardinal level of homogeneity
 equals the topological weight of this space.
\item Now let \(\eeE \subset \mmM\) consist of all those (finite non-empty)
 spaces that are isometrically embeddable into a certain Euclidean space (that
 is, into \((\RRR^n,d_e)\) for some \(n > 0\) where \(n\) depends on the space
 under consideration). Again, this finitary structure does not satisfy \MBP.
 However, this class also has \AP[\(*\)-]. What is more, for any infinite
 cardinal \(\alpha\) satisfying
 \begin{equation}\label{eqn:alpha}
 \alpha = \alpha^{\aleph_0}
 \end{equation}
 (e.g., for \(\alpha = 2^{\mM}\) where \(\mM\) is an arbitrary infinite
 cardinal) there exists a metric space \(H_{\alpha}\) that is almost AB-HO and
 satisfies \(\FIN{H_{\alpha}} = \eeE\).\par
 To show the properties of \(\eeE\) postulated above, let us consider
 an arbitrary infinite-dimensional real Hilbert space \(H_{\alpha}\) of Hilbert
 space dimension \(\alpha\) (that is, \(\alpha\) coincides with the topological
 weight of \(H_{\alpha}\)). Note that \(\card(H_{\alpha}) = \alpha\) if
 \eqref{eqn:alpha} holds. It is well-known (and is not difficult to prove) that
 \(H_{\alpha}\) is \(\alpha\)-homogeneous. (In particular, it is almost AB-HO if
 \eqref{eqn:alpha} is fulfilled.) It follows that \(\FIN{H_{\alpha}}\) has
 \AP[\(\mM\)-] for all infinite \(\mM < \alpha\). But, of course,
 \(\FIN{H_{\alpha}} = \eeE\), which shows that \(\eeE\) has \AP[\(*\)-].\par
 It is worth mentioning here that an ``intrinsic'' characterisation of members
 of \(\eeE\) is known and due to Menger \cite{mg1,mg2}. We refer interested
 readers to a modern exposition \cite{b-b} of this topic, where a criterion is
 given in a readable form in terms of certain determinants related with
 the metric of a (finite) space.
\item Repeating the ideas of (B), now consider a finitary structure \(\ssS\)
 consisting of all (finite non-empty) spaces that are isometrically embeddable
 into a certain Euclidean sphere equipped with the great-circle distance as
 a metric (cf. \eqref{eqn:great-circle})---again, the dimension of that sphere
 depends on the metric space under consideration. As in (B), this class has
 \AP[\(*\)-], but does not satisfy \MBP. As we may guess, the first formula in
 \eqref{eqn:great-circle} may serve as the definition of a new (but compatible
 with the topology) metric on the unit sphere \(S_{\alpha}\) (centred at \(0\))
 in the Hilbert space \(H_{\alpha}\) introduced in the previous example. With
 this metric \(S_{\alpha}\) becomes geodesic, is \(\alpha\)-homogeneous and
 satisfies \(\FIN{S_{\alpha}} = \ssS\). Again, if \eqref{eqn:alpha} is
 fulfilled, this space is almost AB-HO.
\item Finally, consider a finitary structure \(\hhH\) of all (finite non-empty)
 metric spaces that admit an isometric embedding to a certain hyperbolic space
 \(H^n(\RRR)\) (where \(n\) runs over all positive integers). Again, this class
 has \AP[\(*\)-], but fails to satisfy \MBP. However, this time to prove
 directly the lack of \MBP\ is a much harder task than in our three previous
 examples (we will show it below as a consequence of \LEM{bdd}).\par
 Using \eqref{eqn:hyper}, we may introduce on Hilbert spaces \(H_{\alpha}\) new
 metrics (we will call them hyperbolic and denote by \(d_h\)). Substituting
 the identity map on \(H_{\alpha}\) for \(U\) in \eqref{eqn:trans-hyper}, we
 obtain a function on \(H_{\alpha}\) that is an isometry in the hyperbolic
 metric (which one shows similarly as in the finite-dimensional case). So, this
 space is homogeneous. What is more, each (linear) unitary operator on
 \(H_{\alpha}\) is an isometry with respect to \(d_h\). And, moreover, all
 partial isometries that fix the origin extend to linear partial isometries and,
 as such, extend to global isometries, provided the topological weight of their
 domains is less than \(\alpha\). So, hyperbolic Hilbert space \(H_{\alpha}\) is
 \(\alpha\)-homogeneous and, moreover, it is almost AB-HO provided
 \eqref{eqn:alpha} holds. As it is easily seen, \(\FIN{H_{\alpha},d_h} = \hhH\).
 Therefore, \(\eeE\) has \AP[\(*\)-] and fails to satisfy \MBP\ (thanks to
 \LEM{bdd}).
\end{enumerate}
The authors do not know whether the almost AB-HO spaces discussed in the three
example above can be characterised among finitely represented metric spaces in
their finitary structures in terms of homogeneity and related notions.
\end{exm}

\subsection{Metric skeletoids}\label{sec:metric}
In \EXM{metricoid} we have defined metric skeletoids. We call a metric skeletoid
\(\ffF\) \emph{edge-closed} if \(\Delta_{\ffF}(\Gg)\) is a closed subset of
\(\RRR\) for any non-symmetric incomplete GEC \(\Gg\) affiliated with \(\ffF\).
In the aforementioned example we have also given the proof of the following

\begin{pro}{edge-closed}
Each edge-closed metric skeletoid is a skeleton.
\end{pro}

For further purposes, let us call a metric skeletoid \(\ffF\)
\emph{layer-closed} if for any \(n > 1\) the set \(\DdD_{\ffF}(I_n)\) of all
metrics \(d\) on \(I_n = \{1,\ldots,n\}\) with \((I_n,d) \in \ffF\) is closed
(in the pointwise convergence topology) in the set \(\DdD(I_n)\) of all metrics
on \(I_n\).\par
Let \(\ffF\) be a metric skeleton. It generates a unique AB-HO GEC (which is
a metric space), say \(F\). Because of its uniqueness, we may assign any
topological property of \(F\) to the skeleton \(\ffF\), e.g.: \(\ffF\) is said
to be \emph{complete} (respectively: \emph{separable}; \emph{compact};
\emph{locally compact}; \emph{Heine-Borel}; \emph{connected}; \emph{arcwise
connected}; \emph{locally connected}; and so on) if so is \(F\). (Recall that
a metric space is \emph{Heine-Borel} if all its closed balls are compact.)

\begin{pro}{prop-AHM}
Let \((X,\rho)\) ba an AB-HO metric space and \(\xxX\) be its skeleton.
\begin{enumerate}[\upshape(A)]
\item Every layer-closed metric skeletoid is an edge-closed skeleton.
\item A Heine-Borel metric skeleton is layer-closed. More generally, if for some
 \(r > 0\) all closed balls in \(X\) of radius \(r\) are compact, then for any
 \(n > 0\) the set \(\DdD_{\xxX}^{(r)}(I_n)\) of all metrics from
 \(\DdD_{\xxX}(I_n)\) that are upper bounded by \(r\) is closed in the set
 \(\DdD(I_n)\) of all metrics on \(I_n\).
\item A layer-closed metric skeleton is complete. More generally, if
 \(\DdD_{\xxX}^{(r)}(I_n)\) is closed in \(\DdD(I_n)\) for any \(n > 0\) and
 some fixed \(r > 0\), then \((X,\rho\)) is a complete metric space.
\item \(X\) is locally connected or locally hereditary disconnected. If \(X\) is
 disconnected, all its connected components are bounded.
\item All connected components as well as arcwise-connected components of \(X\)
 are AB-HO.
\end{enumerate}
\end{pro}

Recall that a topological space is hereditary disconnected if all its connected
components are singletons. It is locally hereditary disconnected if it may be
covered by open sets each of which is hereditary disconnected. As we will see in
the proof, part (D) of the above result is actually a property of all two-point
homogeneous metric spaces \(X\).

\begin{proof}[Proof of \PRO{prop-AHM}]
(A): Let \(\ffF\) be a layer-closed metric skeletoid. According to
\PRO{edge-closed}, we only need to check that \(\ffF\) is edge-closed. To this
end, we consider a non-symmetric incomplete GEC \(\Gg = (I_k,1,2,d')\) (where
\(k > 2\)) affiliated with \(\ffF\) and observe that if the numbers \(t_n \in
\Delta_{\ffF}(\Gg)\) tend to \(t_{\infty} \in \RRR\), then the corresponding
metrics \(d_n \in \DdD_{\ffF}(I_k)\) (which are obtained from \(d'\) by
declaring that \(d' \subset d_n\) and \(d_n(1,2) = t_n\)) converge pointwise to
a respective metric \(d_{\infty}\) on \(I_k\) that corresponds (in a similar
way) to \(t_{\infty}\). So, it follows from the assumption in (A) that
\(d_{\infty} \in \DdD_{\ffF}(I_k)\) and hence \(t_{\infty} \in
\Delta_{\ffF}(\Gg)\).\par
(B): It is sufficient to show only the second statement of that part. So, assume
all closed balls in \(X\) of radius \(r\) are compact (this is equivalent to
saying that a single ball among them is compact---thanks to the homogeneity of
\(X\)) and consider a sequence of metrics \(d_n \in \DdD_{\xxX}^{(r)}(I_k)\)
(for some \(k > 1\)) that converge pointwise to a metric \(d\) on \(I_k\).
Fixing a point \(a \in X\), there are isometric maps \(\phi_n\dd (I_k,d_n) \to
(X,\rho)\) that send 1 onto \(a\). Then all the points \(\phi_n(j)\) where \(n >
0\) and \(j \in I_k\) lie in a closed ball around \(a\) of radius \(r\). So, it
follows from our assumption that, after passing to a subsequence, we may (and
do) assume that \(\phi_n(j) \to z_j \in X\) as \(n \to \infty\) for all \(j \in
I_k\). But then the assignment \(j \mapsto z_j\) defines an isometric map of
\((I_k,d)\) into \((X,\rho)\) and therefore \(d \in \DdD_{\xxX}(I_k)\).\par
To prove (C), it is enough to show the second statement. So, we fix \(r > 0\)
with the property specified therein and consider a one-to-one Cauchy sequence
\((z_n)_{n=1}^{\infty}\) in \(X\). We may and do assume that \(\rho(z_j,z_k)
\leq r\) for all \(j\) and \(k\) and that none of \(z_k\) is the limit of this
sequence. Denote by \(D = \{z_n\dd\ n > 0\} \cup \{g\}\) the completion of
\(\{z_n\dd\ n > 0\}\) (so, \(z_n \to g\) in \(D\) as \(n\to\infty\)). In order
to verify that this Cauchy sequence is convergent in \(X\), it is sufficient to
prove that \(D\) isometrically embeds into \(X\) (because the absolute
homogeneity of \(X\) guarantees that if one isometric copy of
\((z_n)_{n=1}^{\infty}\) converges in \(X\), then all do so). To this end, we
involve \THM{embed}. So, let \(S\) be any finite non-empty subset of \(D\). If
\(g \notin S\), then \(S \subset X\) and we have nothing to do. Hence we assume
\(g \in S\). We may also assume that \(S\) has more than one point. Write \(S
\setminus \{g\}\) as \(\{u_1,\ldots,u_q\}\) where all these points are
different. Let \(N > 0\) be such an index that \(z_n \notin S\) for all \(n >
N\). Fix \(n > N\) and consider a metric \(d_n\) on \(I_{q+1}\) that is obtained
from the metric of \(D\) by identifying the numbers \(1,\ldots,q\) with
the points \(u_1,\ldots,u_q\) (respectively) and \(q+1\) with \(z_n\). Since
\(D \setminus \{g\} \subset X\), we see that \(d_n \in
\DdD_{\xxX}^{(r)}(I_{q+1})\). Moreover, since \(\lim_{n\to\infty} z_n = g\) (in
\(D\)), these metrics converge pointwise to the metric \(d_{\infty}\) on
\(I_{q+1}\) that is obtained from the metric of \(D\) by identifying the numbers
\(1,\ldots,q\) with the points \(u_1,\ldots,u_q\) (respectively) and \(q+1\)
with \(g\). It follows from our assumption in (C) that \(d_{\infty} \in
\DdD_{\xxX}(I_{q+1})\) and hence \(S\), being isometric to
\((I_{q+1},d_{\infty})\), belongs to \(\xxX\). In other words, \(S\)
isometrically embeds in \(X\) and we are done.\par
(D): Assume \(X\) is not locally connected. This means that for some point \(a\)
and radius \(r > 0\) the connected component \(S\) (containing \(a\)) of
the open ball \(B(a,r)\) around \(a\) of radius \(r\) is not a neighbourhood of
\(a\). We claim that \(S = \{a\}\). To this aim, assume, on the contrary, that
there is \(b \in S\) different from \(a\). Denote \(s = \rho(a,b) > 0\). We will
show that the closed ball \(\bar{B}(a,s)\) around \(a\) of radius \(s\) is
contained in \(S\), which will contradict our assumption that \(S\) is not
a neighbourhood of \(a\).\par
It follows from the connectedness of \(S\) that the set \(J \df \{\rho(a,x)\dd\
x \in S\}\) is connected as well and thus \((0,s] \subset J\). Fix any point \(z
\in X\) whose distance from \(a\) lies in \(J\). Choose any point \(c \in S\)
such that \(\rho(a,c) = \rho(a,z)\). We infer from the two-point homogeneity of
\(X\) that there is an isometry \(u\) of \(X\) such that \(u(a) = a\) and \(u(c)
= z\). It follows from the former equality that \(u(B(a,r)) = B(a,r)\) and,
consequently, \(u(S) = S\). So, \(z \in S\), as we claimed.\par
Knowing that \(S = \{a\}\) it is now easy to check that all open balls of radius
\(r/2\) are hereditary disconnected: first of all, we conclude from
the homogeneity of \(X\) that the connected component (containing \(x\)) of
the open ball \(B(x,r)\) around (arbitrarily chosen) \(x \in X\) of radius \(r\)
coincides with \(\{x\}\). So, the same property holds for all neighbourhoods of
\(x\) contained in the aforementioned ball. Therefore, if \(x \in W \df
B(w,r/2)\), then \(W \subset B(x,r)\) and the connected component of \(W\)
that contains \(x\) is a singleton, which finishes the proof of the first part
of (D). The second follows by the same argument as presented above (if some
connected component of \(X\) is unbounded, it contains closed balls of arbitrary
large radii, centred at a fixed point). Part (E) is left to the reader.
\end{proof}

For a locally compact metric space \(X\), hereditary disconnectedness of \(X\)
is equivalent to the so-called strong zero-dimensionality of \(X\) (that is, to
the equality \(\dim X = 0\) where `\(\dim\)' is the covering dimension).
The latter property is local and hence (for metrizable spaces) \(X\) is locally
hereditary disconnected iff it itself is hereditary disconnected. All these
remarks combined with properties (D)--(E) of \PRO{prop-AHM}, and with \THM{Freu}
and \THM{card} yield

\begin{cor}{loc-comp}
A locally compact AB-HO metric space is homeomorphic to exactly one of
the following spaces:
\begin{enumerate}[\upshape(a)]
\item a discrete topological space \(D\) of arbitrary size not greater than
 \(2^{2^{\aleph_0}}\);
\item \(\RRR^n \times D\) where \(n > 0\) is finite and \(D \neq \varempty\) is
 as specified in \UP{(a)};
\item \(\SSS^n \times D\) where \(n > 0\) is finite and \(D \neq \varempty\) is
 as specified in \UP{(a)};
\item the Cantor set \(\CcC\);
\item \(\CcC \times D\) where \(\CcC\) is the Cantor set and \(D\) is infinite
 and as specified in \UP{(a)}.
\end{enumerate}
\end{cor}
\begin{proof}
All that has not been said yet is a classification of topologically homogeneous
hereditary disconnected non-discrete locally compact metrizable spaces which
says that each such a space is homeomorphic to a space of the form specified in
(d) or (e) (with no limitations on the size of the space \(D\)). The details are
left to the reader.
\end{proof}

In \REM{exm-loc-comp} (see the next section) we comment on the existence of
compatible metrics that make AB-HO the spaces listed in the above result.\par
Our next main aim is to characterise Heine-Borel skeletons. It will be
an immediate consequence of the following lemma and our previous results.

\begin{lem}{comp-ball}
Let \(\ffF\) be a layer-closed metric skeletoid that generates an AB-HO metric
space \(X\) and let \(r\) be a positive real number.
\begin{enumerate}[\upshape(A)]
\item If closed balls of radius \(r\) are compact in \(F\), then for any \(\epsi
 \in (0,r)\) there exists a positive integer \(\mu = \mu(\epsi,r)\) such that
 each member of \(\ffF\) whose all positive distances lie in \([\epsi,r]\) has
 size not greater than \(\mu\).
\item Conversely, if for any \(\epsi \in (0,r)\) there exists a positive integer
 \(\mu\) with the property specified in \UP{(A)}, then closed balls of radius
 \(r/2\) are compact in \(F\).
\end{enumerate}
\end{lem}
\begin{proof}
First of all, note that a layer-closed metric skeletoid is a complete skeleton,
which follows from parts (A) and (C) of \PRO{prop-AHM}. Now both the parts
follow from a classical theorem characterising compactness in complete metric
spaces by means of finite \(\epsi\)-nets (an \(\epsi\)-net is any subset \(A\)
of a metric space such that each point of this space is at distance less than
\(\epsi\) from at least one point of \(A\)). Again, we leave the details to
the reader.
\end{proof}

As a corollary, we obtain the announced result (we skip its proof).

\begin{thm}{H-B-AB-HO}
For a metric skeletoid \(\ffF\) \tfcae
\begin{enumerate}[\upshape(i)]
\item \(\ffF\) is the skeleton of a Heine-Borel AB-HO metric space;
\item \(\ffF\) is layer-closed and for any integer \(N > 1\) there exists
 a positive integer \(\mu_N\) such that each metric space belonging to \(\ffF\)
 whose all positive distances lie in \([1/N,N]\) has size not greater than
 \(\mu_N\).
\end{enumerate}
\(\ffF\) is compact iff \UP{(ii)} holds and all metric spaces from \(\ffF\) have
uniformly bounded diameters.
\end{thm}

\begin{rem}{H-B}
Since Heine-Borel metric spaces are locally compact and separable,
\COR{loc-comp} (combined with part (E) of \PRO{prop-AHM}) implies that each
AB-HO Heine-Borel metric space is homeomorphic to exactly one of the following
spaces:
\begin{enumerate}[\upshape(a)]
\item at most countable discrete topological space \(D\);
\item \(\RRR^n\) where \(n > 0\) is finite;
\item \(\SSS^n \times D\) where \(n > 0\) is finite and \(D \neq \varempty\) is
 as specified in \UP{(a)};
\item the Cantor set \(\CcC\);
\item \(\CcC \times D\) where \(\CcC\) is the Cantor set and \(D\) is
 an infinite countable discrete topological space.
\end{enumerate}
Our second remark on AB-HO Heine-Borel metric spaces is related to
\THM{H-B-AB-HO}. For a finite metric space \((X,d)\) having at least two points,
denote by \(R(X)\) the maximum of two real numbers: the diameter of \((X,d)\)
and the reciprocal of \[\min\{d(x,y)\dd\ x \neq y\}.\] Then the last mentioned
theorem says that a metric skeletoid \(\ffF\) is Heine-Borel iff it is
layer-closed and satisfies the following condition for any sequence of spaces
\((X_n,d_n) \in \ffF\):
\[\lim_{n\to\infty} \card(X_n) = \infty \implies
\lim_{n\to\infty} R(X_n) = \infty.\]
\end{rem}

\section{Products of GEC's}\label{sec:products}

So far we have seen very few examples of AB-HO GEC's. In this section we
introduce a method of constructing such new objects from pre-existing. To this
end, we introduce

\begin{dfn}{disjoint}
Let \(\FFf = \{\Gg_s\}_{s \in S}\) be a totally arbitrary collection of GEC's.
\(\FFf\) is said to be \emph{productable} if the following three conditions are
fulfilled:
\begin{itemize}
\item each of the GEC's \(\Gg_s\) is non-empty and the set \(S\) of indices is
 non-empty as well;
\item the pseudocolors of all the GEC's \(\Gg_s\) coincide;
\item the palettes of these GEC's are pairwise disjoint; that is,
 \[\forall s_1,s_2 \in S,\ s_1 \neq s_2\dd\quad \COL{\Gg_{s_1}} \cap
 \COL{\Gg_{s_2}} = \varempty.\]
\end{itemize}
\end{dfn}

For any family \(\{X_t\}_{t \in T}\) and two elements \(\uU = (u_t)_{t \in T}\)
and \(\vV = (v_t)_{t \in T}\) of the product \(\prod_{t \in T} X_t\), we will
denote by \(\DIF{\uU}{\vV}\) the set of all indices \(t \in T\) such that \(u_t
\neq v_t\). (So, \(\DIF{\uU}{\vV} = \varempty\) iff \(\uU = \vV\).) If \(\uU
\neq \vV\) and, in addition, the set \(\DIF{\uU}{\vV}\) is well-ordered with
respect to a certain total order \(\ORD\) on \(T\), we will use symbol
\(\LAB{\uU}{\vV}\) to denote the minimal (with respect to this order) element of
\(\DIF{\uU}{\vV}\). Notice that if \(\uU\), \(\vV\) and \(\wW\) are three
distinct elements of \(\prod_{t \in T} X_t\) and two out of the three sets
\(\DIF{\uU}{\vV}\), \(\DIF{\uU}{\wW}\), \(\DIF{\vV}{\wW}\) are well-ordered with
respect to a given total order on \(T\), then the third one is well-ordered as
well (because each of these sets is contained in the union of the other
two).\par
We are now ready to define the main notion of this section.

\begin{dfn}{product}
Let \(\FFf = \{\Gg_s\}_{s \in S}\) be a productable collection of GEC's with
\(\Gg_s = (V_s,\kappa_s)\) and the pseudocolor \(e\). For a fixed element \(\aA
= (a_s)_{s \in S}\) of the product \(\prod_{s \in S} V_s\) and a total order
\(\ORD\) on the set \(S\) of indices, the \emph{product of \(\ffF\) with respect
to the pair \((\aA,\ORD)\)} is a GEC, denoted by \(\PROd{s \in S}{\ORD}{\aA}
\Gg_s = \PROd{s \in S}{\ORD}{\aA} (V_s,\kappa_s)\), of the form \((W,\lambda)\)
where \[W \df \Bigl\{\xX = (x_s)_{s \in S} \in \prod_{s \in S} V_s\dd
\DIF{\xX}{\aA} \UP{ is well-ordered w.r.t.}\ORD\Bigr\}\]
and for two elements \(\xX = (x_s)_{s \in S}\) and \(\yY = (y_s)_{s \in S}\) of
\(W\):
\[\lambda(\xX,\yY) \df \begin{cases}e & \xX = \yY\\
\kappa_{\ell}(x_{\ell},y_{\ell}) & \xX \neq \yY\end{cases}\]
where \(\ell = \LAB{\xX}{\yY}\) for \(\xX \neq \yY\). Note that
\(\COL{W,\lambda} = \bigcup_{s \in S} \COL{\Gg_s}\). Observe also that if
\(\ORD\) is a well-order on \(T\), then the set \(W\) coincides with the full
set-theoretic product \(\prod_{s \in S} V_s\) (however, the edge-coloring of
\(W\) does depend on this order). Remember also that the formula for \(\lambda\)
does not depend on the choice of \(\aA\).\par
If, in the above settings, the set \(S = \{k_1,\ldots,k_r\}\) of indices (where
\(k_1 < \ldots < k_r\)) is finite, consists of non-negative integers and is
ordered by the natural order \(\leq\) on \(\ZZZ\), we shall write \(\Gg_{k_1}
\times^* \ldots \times^* \Gg_{k_r}\) instead of \(\PROd{s \in S}{\leq}{\aA}
\Gg_s\). (We leave the asterisk to underline that, in general, this product is
non-symmetric; that is, it may happen that \(\Gg_1 \times^* \Gg_2 \not\equiv
\Gg_2 \times^* \Gg_1\).) Analogously, if \(S\) consists of all positive integers
and the order on \(S\) coincides with the natural one, we will use a suggestive
notation \(\Gg_1 \times^* \Gg_2 \times ^* \ldots\)
\end{dfn}

\begin{rem}{UZU-rem-prod-ultram}
The above construction of products with rescaled discrete metric spaces (as
factors) and the reverse order to the natural one on the set of positive reals
(as the set of indices) is known in the geometry of ultrametric spaces and is
used to classify all the so-called \emph{spherically complete} homogeneous
ultrametric spaces. We have learnt about this constructrion from a presentation
found in the Internet after finishing our research on AB-HO GEC's.
\end{rem}

Since we are mainly interested in products of AB-HO GEC's (and we wish to
consider them up to isomorphism), let us make the following simple observation.

\begin{pro}{iso-products}
Let \(\{\Gg_s\}_{s \in S}\) be a productable collection of GEC's where \(\Gg_s =
(V_s,\kappa_s)\), let \(\ORD\) be a total order on \(S\), \(\aA =
(a_s)_{s \in S}\) be a fixed element of \(\prod_{s \in S} V_s\) and for each
\(s \in S\) let \(\phi_s\) be an arbitrary automorphism of \(\Gg_s\). Then, for
\(\bB = (\phi_s(a_s))_{s \in S}\) the GEC's \(\PROd{s \in S}{\ORD}{\aA} \Gg_s\)
and \(\PROd{s \in S}{\ORD}{\bB} \Gg_s\) are isomorphic.\par
In particular, if all the GEC's \(\Gg_s\) are homogeneous, all the GEC's
\(\PROd{s \in S}{\ORD}{\aA} \Gg_s\) where \(\aA\) runs over all elements of
\(\prod_{s \in S} V_s\) are pairwise isomorphic.
\end{pro}
\begin{proof}
It is enough to show only the first statement. To this end, observe that
the formula \(\Phi((\xX_s)_{s \in S}) \df (\phi_s(\xX_s))_{s \in S}\) correctly
defines an isomorphism from \(\PROd{s \in S}{\ORD}{\aA} \Gg_s\) onto
\(\PROd{s \in S}{\ORD}{\bB} \Gg_s\).
\end{proof}

The above result enables us to simplify our notation of products: if all
the GEC's \(\Gg_s\) are homogeneous, we will write \(\PROD{s \in S} \Gg_s\)
instead of \(\PROd{s \in S}{\ORD}{\aA}\) (as the structure of the latter GEC is
independent of the choice of \(\aA\)). We leave in the notation the asterisk to
underline that, in general, the set of vertices of the product of GEC's does not
coincide with the full set-theoretic product of the resepctive sets of
vertices.\par
Now we will state and prove the main result of the section:

\begin{thm}{prod-AB-HO}
For a productable collection \(\{\Gg_s\}_{s \in S}\) of AB-HO GEC's and
an arbitrary total order \(\preceq\) on \(S\), the GEC
\(\PROd{s \in S}{\preceq}{*} \Gg_s\) is AB-HO as well.
\end{thm}
\begin{proof}
We apply part (A) of \PRO{1}. Write \(\Gg_s = (V_s,\kappa_s)\ (s \in S)\) and
\(\Hh = (W,\lambda)\) where \(\Hh \df \PROd{s \in S}{\preceq}{\aA} \Gg_s\) (for
some fixed \(\aA \in \prod_{s \in S} V_s\)), and fix a partial morphism
\(\phi\dd X \to W\) of \(\Hh\) (with \(X \subsetneq W\)) as well as a vertex
\(\bB \in W \setminus X\). We only need to show that there exists a vertex \(\cC
\in W \setminus \phi(X)\) such that
\begin{equation}\label{eqn:aux5}
\lambda(\phi(\xX),\cC) = \lambda(\xX,\bB)
\end{equation}
holds for all \(\xX \in X\). For any \(t \in S\) we denote by \(\pi_t\dd W \to
V_t\) the natural projection onto the \(t\)-th coordinate; that is,
\(\pi_t((x_s)_{s \in S}) = x_t\). To simplify notation, we will write
\(\ell(\xX,\yY)\) in place of \(\LAB[\preceq]{\xX}{\yY}\) (for distinct \(\xX\)
and \(\yY\)).\par
Since \(\phi\) is a morphism, it follows from the property that palettes are
pairwise disjoint (cf. \DEF{disjoint}) that
\begin{equation}\label{eqn:lab}
\ell(\phi(\xX),\phi(\yY)) = \ell(\xX,\yY) \qquad
(\xX, \yY \in X,\ \xX \neq \yY).
\end{equation}
For any \(s \in S\), denote by \(Z_s\) the set of all \(\xX \in X\) such that
\(s \precneqq \ell(\xX,\bB)\), by \(I_0\) the set of all \(s \in S\) for which
\(Z_s \neq \varempty\), by \(I_1\) the set of all \(s \in S \setminus I_0\) such
that \(s = \ell(\xX,\bB)\) for some \(\xX \in X\) and finally let \(I_2 \df S
\setminus (I_0 \cup I_1)\). It follows from the very definitions of both
the sets \(I_0\) and \(I_1\) that \(I_1\) consists of maximal elements of
the set \(\{\ell(\xX,\bB)\dd\ \xX \in X\}\) and therefore \(\card(I_1) \leq 1\).
We shall now show that:
\begin{enumerate}[({a}ux1)]
\item for any three distinct points \(\uU, \vV, \wW \in W\) with \(\ell(\uU,\vV)
 \preceq \ell(\uU,\wW) \preceq \ell(\vV,\wW)\) one has \(\ell(\uU,\vV) =
 \ell(\uU,\wW)\);
\item \(\pi_s(\xX) = \pi_s(\yY)\) and \(\pi_s(\phi(\xX)) = \pi_s(\phi(\yY))\)
 for all \(s \in S\) and distinct \(\xX, \yY \in Z_s\);
\item if \(I_1 = \{t\}\) and \(\xX, \yY \in Y\) where \(Y\) consists of all
 \(\vV \in X\) that satisfy \(\ell(\vV,\bB) = t\), then:
 \begin{itemize}
 \item \(\pi_t(\xX) = \pi_t(\yY) \iff \pi_t(\phi(\xX)) = \pi_t(\phi(\yY))\); and
 \item if \(\pi_t(\xX) \neq \pi_t(\yY)\), then
  \(\kappa_t(\pi_t(\phi(\xX)),\pi_t(\phi(\yY))) =
  \kappa_t(\pi_t(\xX),\pi_t(\yY))\).
 \end{itemize}
\end{enumerate}
To see (aux1), observe that it follows from the definition of \(\ell\) that
for any index \(q \precneqq \ell(\uU,\vV)\) one has \(\pi_q(\uU) = \pi_q(\vV)\).
So, if \(q \df \ell(\uU,\vV)\) were different from \(\ell(\uU,\wW)\), then
\(\pi_q(\uU) = \pi_q(\wW)\) and, consequently, \(\pi_q(\vV) \neq \pi_q(\wW)\),
which would imply that \(\ell(\vV,\wW) \preceq q\), contradictory to our
supposition.\par
To prove (aux2), we may (and do) assume that \(\ell(\xX,\bB) \preceq
\ell(\yY,\bB)\). It follows from (aux1) that \(\ell(\xX,\bB) \preceq
\ell(\xX,\yY)\) and hence \(\pi_s(\xX) = \pi_s(\yY)\). So, we infer from
\eqref{eqn:lab} that also \(\ell(\xX,\bB) \preceq \ell(\phi(\xX),\pi(\yY))\).
Consequently, \(\pi_s(\phi(\xX)) = \pi_s(\phi(\yY))\) and we are done.\par
Finally, to verify (aux3), assume (with no loss on generality) that \(\xX \neq
\yY\) and conclude from (aux1) that \(t \preceq \ell(\xX,\yY)\). So,
\[\pi_t(\xX) = \pi_t(\yY) \iff t \neq \ell(\xX,\yY) =
\ell(\phi(\xX),\phi(\yY)) \iff \pi_t(\phi(\xX)) = \pi_t(\phi(\yY))\]
(by \eqref{eqn:lab}). So, if \(\pi_t(\xX) \neq \pi_t(\yY)\), then
\(\ell(\xX,\yY) = \ell(\phi(\xX),\phi(\yY)) = t\) and
\[\kappa_t(\pi_t(\xX),\pi_t(\yY)) = \lambda(\xX,\yY) =
\lambda(\phi(\xX),\phi(\yY)) = \kappa_t(\pi_t(\phi(\xX)),\pi_t(\phi(\yY))).\]
If \(I_1\) is non-empty (and \(I_1 = \{t\}\)), (aux3) enables us to define
a partial morphism \(\psi\dd \pi_t(Y) \to V_t\) of \(\Gg_t\) by the rule:
\[\psi(\pi_t(\vV)) = \pi_t(\phi(\vV)) \qquad (\vV \in Y).\]
In that case we take any automorphism \(\tau\) of \(\Gg_t\) that extends
the above \(\psi\).\par
Now for any \(s \in S\) we define \(c_s \in V_s\) as follows:
\begin{itemize}
\item if \(s \in I_0\), let \(c_s \df \pi_s(\phi(\xX))\) for arbitrarily chosen
 \(\xX \in Z_s\) ((aux2) shows that this definition does not depend on
 the choice of \(\xX\));
\item if \(s \in I_1\) (and thus \(s = t\)), then \(c_s \df \tau(\pi_s(\bB))\);
\item if \(s \in I_2\), then \(c_s \df \pi_s(\aA)\).
\end{itemize}
Notice that \(\cC \df (c_s)_{s \in S}\) is a well defined element of
\(\prod_{s \in S} V_s\). We claim that \(\cC \in W \setminus \phi(X)\) and that
the assignment \(\phi(\bB) \df \cC\) extends (previous) \(\phi\) to a partial
morphism of \(\Hh\). To this end, observe that \(\DIF{\cC}{\aA} \subset I_0 \cup
I_1\) and thus (since \(I_1\) is finite), it is sufficient to show that for any
\(q \in I_0\) the set \(J_q \df \{s \in \DIF{\cC}{\aA}\dd\ s \preceq q\}\) is
well-ordered by \(\preceq\). But for \(q \in I_0\) we can take any \(\xX \in
Z_q\). Then \(\xX \in Z_s\) for any \(s \in J_q\) and, consequently,
\(\pi_s(\cC) = \pi_s(\xX)\) for \(s \in J_q\). Hence \(J_q \subset
\DIF{\xX}{\aA}\) and therefore \(\cC \in W\). Now fix any \(\xX \in X\) and
denote \(s \df \ell(\xX,\bB)\). We have two possibilities:
\begin{itemize}
\item \(s \in I_0\): then there exists \(\zZ \in X\) such that \(s \precneqq
 \ell(\zZ,\bB)\) (in particular, \(\pi_s(\zZ) = \pi_s(\bB)\)); we conclude from
 (aux1) and \eqref{eqn:lab} that (\(\zZ \neq \xX\) and) \(s = \ell(\xX,\zZ) =
 \ell(\phi(\xX),\phi(\zZ))\); observe  that \(\pi_q(\cC) = \pi_q(\phi(\zZ))\)
 for any index \(q \precneqq \ell(\zZ,\bB)\) (as \(\zZ \in Z_q\) for all such
 \(q\)), which implies that \(\cC \neq \phi(\xX)\), \(\ell(\cC,\phi(\xX)) = s\)
 (again by (aux1)) and
 \begin{multline*}
 \lambda(\cC,\phi(\xX)) = \kappa_s(\pi_s(\cC),\pi_s(\phi(\xX))) =
 \kappa_s(\pi_s(\phi(\zZ)),\pi_s(\phi(\xX))) \\= \lambda(\phi(\zZ),\phi(\xX)) =
 \lambda(\zZ,\xX) = \kappa_s(\pi_s(\zZ),\pi_s(\xX)) =
 \kappa_s(\pi_s(\bB),\pi_s(\xX)) = \lambda(\bB,\xX);
 \end{multline*}
 or else
\item \(s \in I_1\): then \(s = t\), \(\xX \in Y\), \(\pi_t(\xX) \neq
 \pi_t(\bB)\) and \(\pi_t(\cC) = \tau(\pi_t(\bB)) \neq \tau(\pi_t(\xX)) =
 \psi(\pi_t(\xX)) = \pi_t(\phi(\xX))\) and hence \(\cC \neq \phi(\xX)\);
 moreover, \(\xX \in Z_q\) for any \(q \precneqq s\) and therefore \(\pi_q(\cC)
 = \pi_q(\phi(\xX))\) for all such \(q\), so \(\ell(\cC,\phi(\xX)) = t\) and
 \begin{multline*}
 \lambda(\cC,\phi(\xX)) = \kappa_t(\pi_t(\cC),\pi_t(\phi(\xX))) \\=
 \kappa_t(\tau(\pi_t(\bB)),\tau(\pi_t(\xX))) = \kappa_t(\pi_t(\bB),\pi_t(\xX)) =
 \lambda(\bB,\xX).
 \end{multline*}
\end{itemize}
Both the above cases prove that our extension is a partial morphism and we are
done.
\end{proof}

\begin{cor}{2toY}
For any infinite set \(Y\) there exists an AB-HO GEC \(\Hh = (V,\kappa)\) such
that \(\COL{\Hh} = Y\) and \(\card(V) = 2^{\card(Y)}\).
\end{cor}
\begin{proof}
Fix a pseudocolor \(e \notin Y\) and a well-order \(\ORD\) on \(Y\). For any \(c
\in Y\) let \(\Gg_c\) be a \(c\)-monochromatic GEC with two vertices and
pseudocolor \(e\). Then \(\{\Gg_c\}_{c \in Y}\) is a productable collection and
\(\Hh = \PROD{c \in Y} \Gg_c\) is the GEC we searched for (the set of vertices
of this product coincides with the full set-theoretic product of the respective
sets of vertices, because \(Y\) is well-ordered).
\end{proof}

A situation becomes more complex in case of metric spaces, as the product of
AB-HO metric spaces (considered as GEC's) is, in general, not a metric
space---that is, the triangle inequality may fail to hold in products. We
discuss it in greater detail in next results. Observe that a collection
of metric spaces is productable (in the sense of \DEF{disjoint}) iff any two
members of this family with different indices admit different positive
distances (that is, if their sets of all attainable positive distances are
distjoint).\par
For any productable collection \(\{\Gg_s = (V_s,\kappa_s)\}_{s \in S}\) of GEC's
and a fixed element \(\aA = (a_s)_{s \in S}\) of \(\prod_{s \in S} V_s\),
denoting the product \(\PROd{s \in S}{\ORD}{\aA}\) by \(\Hh = (Z,\rho)\), we may
consider \emph{canonical} morphisms \(j_t\dd V_t \to Z\) (from \(\Gg_t\) into
\(\Hh\)) defined for all \(t \in S\) as follows: for \(v \in V_t\), \(j(v) =
(v_s)_{s \in S}\) satisfies \(v_s = a_s\) for \(s \neq t\) and \(v_t = v\). It
may easily checked that \(j_t\) is indeed a (well defined) morphism.\par
In the sequel we shall use the following intuitive result, whose proof is left
to the reader.

\begin{pro}{prod-prod}
Let \(\{\Gg_s = (V_s,\kappa_s)\}_{s \in S}\), \(a_s\) and \(\ORD\) be,
respectively, a productable collection of GEC's, any element from \(V_s\) and
a total order on \(S\). Assume \(S = \bigcup_{t \in T} S_t\) where:
\begin{itemize}
\item \((T,\preceq)\) is a totally ordered non-empty set; and
\item the sets \(S_t\ (t \in T)\) are pairwise disjoint and non-empty; and
\item for any two distinct elements \(t_1,t_2 \in T\) such that \(t_1 \preceq
 t_2\) all the elements of \(S_{t_1}\) precede all the elements of \(S_{t_2}\);
 that is:
 \[\forall s_1 \in S_{t_1},\ s_2 \in S_{t_2}\dd\ s_1 \ORD s_2.\]
\end{itemize}
Then for each \(t \in T\) the collection \(\{\Gg_s\}_{s \in S_t}\) is
productable and the assignment
\[\PROd{s \in S}{\ORD}{\aA} \ni (z_s)_{s \in S} \mapsto
((z_s)_{s \in S_t})_{t \in T} \in \PROd{t \in T}{\preceq}{\aA'}
(\PROd{s \in S_t}{\ORD}{\aA_t} \Gg_s)\]
correctly defines an isomorphism of GEC's (in particular, the collection of all
factors of the right-hand side product is productable) where \(\aA =
(a_s)_{s \in S}\), \(\aA_t = (a_s)_{s \in S_t}\) and \(\aA' =
(\aA_t)_{t \in T}\).
\end{pro}

To simplify further statements, for any productable collection
\(\FFf = \{\Gg_s\}_{s \in S}\) of GEC's we denote by \(\SUPP{\FFf}\) the set
of all indices \(s \in S\) for which \(\Gg_s\) is non-degenerate and call it
the \emph{support} of the collection \(\FFf\). Observe that:
\begin{itemize}
\item if \(\SUPP{\FFf} = \varempty\), the product (with respect to any order on
 the set of indices) is a degenerate GEC;
\item if \(\SUPP{\FFf} = \{t\}\), then the canonical morhism \(j_t\) is
 an isomorphism (again, for any total order on the set of indices).
\end{itemize}
That is why we assume in part (B) below that the support of the collection
contains more than one index.

\begin{pro}{prod-metric}
\begin{enumerate}[\upshape(A)]
\item Let \(\{\Gg_s\}_{s \in S}\) be a productable collection of AB-HO metric
 spaces and let \(\ORD\) be a total order on \(S\). Then the product
 \(\PROD{s \in S} \Gg_s\) is a metric space iff the following condition is
 fulfilled:
 \begin{quote}
 If \(s_1\) and \(s_2\) are distinct indices from \(S\) with \(s_1 \ORD s_2\)
 and \(a\) and \(b\) are two positive real numbers that are attainable as
 distance in, respectively, \(\Gg_{s_1}\) and \(\Gg_{s_2}\), then \(b \leq 2a\).
 \end{quote}
\item Let \(\FFf = \{\Gg_s = (X_s,d_s)\}_{s \in S}\) and \(\ORD\) be
 a productable collection of AB-HO metric spaces and a total order on \(S\) such
 that \(\card(\SUPP{\FFf}) > 1\) and the GEC \(\Hh = (Z,\rho) \df
 \PROd{s \in S}{\ORD}{\aA} \Gg_s\) is a metric space (where \(\aA =
 (a_s)_{s \in S}\) is an arbitrarily fixed element of \(\prod_{s \in S} X_s\)).
 Then:
 \begin{enumerate}[\upshape(a)]
 \item The topology of \(\Hh\) is finer than the topology inherited from
  the product topology of \(\prod_{s \in S} X_s\).
 \item For any \(s \in S\), the canonical isometric map \(j_s\dd X_s \to Z\) has
  closed image.
 \item For any \(s \in S\) that is not the greatest element of \(\SUPP{\FFf}\)
  the space \((X_s,d_s)\) is a complete metric space with discrete topology.
 \item For any \(s \in S\) that is not the least element of \(\SUPP{\FFf}\)
  the metric \(d_s\) is bounded.
 \item If \(t\) is the greatest element of \(\SUPP{\FFf}\) and \(\aA' =
  (a_s)_{s \in S \setminus \{t\}}\), then:
  \begin{itemize}
  \item the metric space \(\Hh' = (Z',\rho') \df
   \PROd{s \in S \setminus \{t\}}{\ORD}{\aA'} \Gg_s\) is complete and has
   discrete topology;
  \item \(\Hh\) is naturally isometric to \(\Hh' \times^* \Gg_t\);
  \item the topology of \(\Hh' \times^* \Gg_t\) coincides with the product
   topology of \(Z' \times X_t\).
  \end{itemize}
  In particular, \(\Hh\) is a complete metric space iff so is \(\Gg_t\).
 \item If \(\SUPP{\FFf}\) does not have the greatest element and every its
  countable subset is upper bounded in \(\SUPP{\FFf}\), then \(\Hh\) is
  a complete metric space and has discrete topology. The same conclusion holds
 if \(d_s(X_s \times X_s) \subset \{0\} \cup [\epsi,\infty)\) for some fixed
 \(\epsi > 0\) and all \(s \in S\).
 \item If \(\SUPP{\FFf}\) does not have the greatest element, then \(\Hh\) is
  topologically non-discrete iff there exists a strictly increasing (with
  respect to \(\ORD\)) sequence \((t_n)_{n=1}^{\infty} \subset \SUPP{\FFf}\)
  that is not upper bounded in \(\SUPP{\FFf}\) and moreover among positive
  reals attainable as a value of any of \(d_{t_n}\) there are numbers
  arbitrarily close to 0. In that case let:
  \begin{itemize}
  \item \(S_1 \df \{s \in S\dd\ s \ORD t_1\}\); and
  \item \(S_n \df \{s \in S\dd\ t_{n-1} \ORD s \ORD t_n\} \setminus
   \{t_{n-1}\}\) for \(n > 1\); and
  \item \(\aA_n \in \prod_{s \in S_n} X_s\) for \(n > 0\) coincide with \(\aA\)
   on coordinates from \(S_n\);
  \item \(\Ww_n = (W_n,\mu_n) \df \PROd{s \in S_n}{\ORD}{\aA_n}\) for \(n > 0\).
  \end{itemize}
  Then:
  \begin{itemize}
  \item each of the spaces \(\Ww_n\) is complete and has discrete topology;
  \item \(\Hh\) is naturally isometric to \(\Ww_1 \times^* \Ww_2 \times^*
   \ldots\);
  \item the topology of \(\Ww_1 \times^* \Ww_2 \times^* \ldots\) coincides with
   the product topology of \(\prod_{n=1}^{\infty} W_n\);
  \item \(\Hh\) is a complete metric space.
  \end{itemize}
 \end{enumerate}
 In particular, \(\Hh\) is a complete metric space iff so are all \(\Gg_s\).
\item For any finite productable collection \(\{\Gg_k = (X_k,d_k)\}_{k=1}^r\) of
 AB-HO metric spaces whose product is metric as well, the topology of \(\Gg_1
 \times^* \ldots \times^* \Gg_r\) coincides with the product topology of \(X_1
 \times \ldots \times X_r\).
\end{enumerate}
\end{pro}
\begin{proof}
To see (A), write \(\Gg_s = (X_s,d_s)\) and fix \(\zZ = (z_s)_{s \in S} \in
\prod_{s \in S} X_s\) so that \(\Hh \df \PROD{s \in S} \Gg_s =
\PROd{s \in S}{\ORD}{\zZ} \Gg_s\). Write \(\Hh = (Z,\rho)\), assume it is
a metric space and fix \(s_1,s_2\) and \(a,b\) as specified in (A). Take two
pairs \((u,v) \in X_{s_1}\) and \((p,q) \in X_{s_2}\) such that \(d_{s_1}(u,v) =
a\) and \(d_{s_2}(p,q) = b\). Now define three elements \(\wW = (w_s)_{s \in S},
\xX = (x_s)_{s \in S}, \yY = (y_s)_{s \in S}\) of \(Z\) as follows. For any \(s
\in S \setminus \{s_1,s_2\}\) we put \(w_s = x_s = y_s \df z_s\) and for \(s \in
\{s_1,s_2\}\):
\[w_s \df \begin{cases}u & s = s_1\\p & s = s_2\end{cases},\qquad
x_s \df \begin{cases}u & s = s_1\\q & s = s_2\end{cases},\qquad
y_s \df \begin{cases}v & s = s_1\\p & s = s_2\end{cases}.\]
Notice that then \(\rho(\wW,\xX) = d_{s_2}(p,q) = b\), \(\rho(\wW,\yY) =
d_{s_1}(u,v) = a\) and \(\rho(\xX,\yY) = d_{s_1}(u,v) = a\) and therefore
(thanks to the triangle inequality) \(b \leq 2a\).\par
Now assume that the condition formulated in (A) is fulfilled. We will check that
\(\rho\) satisfies the triangle inequality. To this end, fix three distinct
points \(\wW = (w_s)_{s \in S}\), \(\xX = (x_s)_{s \in S}\) and \(\yY =
(y_s)_{s \in S}\) of \(Z\). With no loss on generality, we may (and do) assume
that \(s \df \LAB{\wW}{\xX} \ORD \LAB{\wW}{\yY} \ORD t \df \LAB{\xX}{\yY}\). It
follows from (aux1) (see the proof of \THM{prod-AB-HO}) that \(\LAB{\wW}{\yY} =
s\). So, \(a \df \rho(\wW,\xX) = d_s(w_s,x_s)\), \(b \df \rho(\wW,\yY) =
d_s(w_s,y_s)\) and \(c \df \rho(\xX,\yY) = d_t(x_t,y_t)\). If \(t = s\),
the triangle inequality (for all arrangements of \(\wW,\xX,\yY\)) simply follows
since \(d_s\) is a metric. So, we now assume that \(t \neq s\). Then (since
\(\LAB{\xX}{\yY} = t\)) \(x_s = y_s\) and hence \(a = b\). We infer from
the condition formulated in (A) that \(c \leq 2a\), which is equivalent to
the triangle inequality for all arrangements of \(\wW,\xX,\yY\).\par
Wa pass to (B). For simplicity, we denote \(K \df \SUPP{\FFf}\). To get (a), it
suffices to show that the natural projections \(\pi_t\dd Z \to X_t\) (introduced
at the beginning of the proof of \THM{prod-AB-HO}) for all \(t \in S\) satisfy
Lipschitz condition with constant 2. But this easily follows from (A): if
\(\xX\) and \(\yY\) are two elements of \(Z\) such that \(\pi_t(\xX) \neq
\pi_t(\yY)\), then \(\LAB{\xX}{\yY} \ORD t\) and thus:
\begin{itemize}
\item either \(\LAB{\xX}{\yY} = t\) and then \(\rho(\xX,\yY) = d_t(\pi_t(\xX),
 \pi_t(\yY))\); or else
\item \(\LAB{\xX}{\yY} = s \neq t\) and then \(d_t(\pi_t(\xX),\pi_t(\yY))
 \stackrel{\UP{(A)}}{\leq} 2 d_s(\pi_s(\xX),\pi_s(\yY)) = 2\rho(\xX,\yY)\).
\end{itemize}
Further, it follows from the condition formulated in (A) that if for \(s \in S\)
there exists \(t \in K \setminus \{s\}\) such that \(s \ORD t\), then
\(d_s(X_s \times X_s) \subset \{0\} \cup [\epsi,\infty)\) for some \(\epsi > 0\)
(it is enough to define \(\epsi\) as the half of a positive distance attainable
in \(X_t\)). So, for such \(s\) the metric \(d_s\) is complete and the topology
of \(X_s\) is discrete. Since all indices \(s \notin K\) have the same two last
mentioned properties, we see that part (c) holds and (b) holds for all \(t \in
S\) that differ from the greatest element of \(K\) (which may or may not exist).
So, it remains to check that \(j_t\) has closed image if \(t\) is the greatest
element of \(K\). But if \(\xX = (x_s)_{s \in S} \in Z\) is not in this image,
then necessarily the least index \(s \in S\) such that \(x_s \neq a_s\) differs
from \(t\) and satisies \(s \ORD t\), and, consequently, \(\rho(j_t(v),\xX) =
d_s(a_s,x_s) > 0\) for any \(v \in X_t\), which finishes the proof of (b).\par
Part (d) follows from the property established in (b): if \(s \in K\) is not
the least element of this set, there is \(q \in K \setminus \{s\}\) such that
\(q \ORD s\) and then, by (A), \(d_s(X_s \times X_s) \subset [0,2p]\) where
\(p\) is any positive distance attainable in \(X_q\). As all the spaces \(X_s\)
with \(s \notin K\) are one-point, the proof of (d) is finished.\par
The first conclusion of (e) follows, again, from (A): since for all indices from
\(S\) that are greater than \(t\) the space \(X_s\) is one-point, the metric of
\(Z'\) satisfies the condition: \(\rho'(Z' \times Z') \subset \{0\} \cup [p/2,
\infty)\) where \(p\) is any positive real number attainable as a distance in
\(X_t\). The same arguments show the remaining parts of (e), cf. \PRO{prod-prod}
(we skip the details).\par
To prove the first statement of (f), it is enough to show that \(\inf(Y) > 0\)
where \(Y \df \rho(Z \times Z) \setminus \{0\}\) (so, \(Y\) is the palette of
\(\Hh\)). We argue by a contradiction. Assume there are real numbers \(c_1 > c_2
> \ldots\) from \(Y\) that converge to 0. It follows from the formula defining
\(\rho\) (see \DEF{product}) that then \(c_n\) is a distance attainable in
\(X_{s_n}\) for some \(s_n \in S\). In particular, \(s_n \in \SUPP{\FFf}\). Now
it follows from the assumptions of (f) that there is \(t \in \SUPP{\FFf}
\setminus \{s_n\dd\ n > 0\}\) such that \(s_n \ORD t\) for all \(n\). So, we
infer from (A) that \(t \leq 2s_n\), which is impossible. The second statement
of (f) is trivial, whereas the first statement of part (g) follows from (f) and
the above argument. To prove the remaining claims of (g), we conclude firstly
from (A) that the metric spaces \(\Ww_n\) are complete and have discrete
topologies, and secondly from \PRO{prod-prod} that \(\Hh\) is isometric to
their product. So, it remains to show that the metric of such a product is
complete and compatible with the product topology (because the last statement of
(B) is a summary of (e), (f) and (g)). To this end, first observe that the set
of vertices of the last mentioned product coincides with the full set-theoretic
product of all \(W_n\). Next, if the points \(x_n = (x_n^{(1)},x_n^{(2)},
x_n^{(3)},\ldots) \in \prod_{k=1}^{\infty} W_k\) form a Cauchy sequence, then
also \((x_n^{(k)})_{n=1}^{\infty}\) is a Cauchy sequence in \(W_k\) for all \(k
> 0\) (since the natural projections are Lipschitz) and thus it is convergent,
say, to \(g_k \in W_k\). Then \(g \df (g_k)_{k=1}^{\infty}\) belongs to
\(\prod_{n=1}^{\infty} W_n\) and the points \(x_n\) converge to \(g\) in
the product topology. So, it remains to show that both the topologies coincide.
Since all they are metrizable, we only need to check (thanks to (a)) that,
maintaining the above notation, the points \(x_n\) converge to \(g\) in
the topology of the GEC product. To this end, we denote by \(\mu\) the metric on
\(\prod_{n=1}^{\infty} W_n\) obtained from \DEF{product}, fix \(\epsi > 0\) and
assume (without loss on generality) that \(x_n \neq g\) for all \(n\).
The properties of the sequence \(t_1,t_2,\ldots\) mentioned in the statement of
(g) imply that there is \(n_0\) and two distinct points \(u, v \in W_{n_0}\)
such that \(\mu_{n_0}(u,v) < \epsi/2\). Since all the spaces \(W_n\) are
topologically discrete, the convergence of \(x_1,x_2,\ldots\) to \(g\) yields
that there exists \(N > n_0\) such that \(x_n^{(k)} = g_k\) for all \(n \geq N\)
and \(k=1,\ldots,n_0\). Then for any \(n \geq N\) the index \(\ell \df
\LAB[\leq]{x_n}{g}\) computed for the collection of all \(\Ww_n\) is greater
than \(n_0\) and hence, by (A), \(\mu_{\ell}(x_n^{(\ell)},g_{\ell}) \leq 2
\mu_{n_0}(u,v) < \epsi\). But \(\mu(x_n,g) = \mu_{\ell}(x_n^{(\ell)},g_{\ell})\)
and we are done.\par
Finally, item (C) is an immediate consequence of parts (c) and (e) of (B).
\end{proof}

\begin{cor}{prod-loc-comp}
Let \(\FFf = \{\Gg_s = (X_s,d_s)\}_{s \in S}\) and \(\ORD\) be a productable
collection of AB-HO metric spaces and a total order on \(S\) such that the GEC
\(\Hh \df \PROD{s \in S} \Gg_s\) is a metric space. Then \(\Hh\) is locally
compact iff one of the following conditions holds:
\begin{itemize}
\item \(\SUPP{\FFf} = \varempty\); or
\item \(\SUPP{\FFf}\) has the greatest element, say \(t\), and \(\Gg_t\) is
 locally compact; or
\item there exists \(\epsi > 0\) such that \(d_s(X_s \times X_s) \subset \{0\}
 \cup [\epsi,\infty)\) for all \(s \in S\); or
\item there exists \(t \in \SUPP{\FFf}\) such that the set \(I \df \{s \in
 \SUPP{\FFf}\dd\ t \ORD s\}\) is order-isomorphic to \(\NNN\) and the spaces
 \(X_s\) with \(s \in I\) are all finite and their diameters tend to 0 when
 \(s \in I\) increases.
\end{itemize}
Moreover, if \(\Hh\) is locally compact, so are all \(\Gg_s\).
\end{cor}
\begin{proof}
The whole assertion follows from parts (e)--(g) of item (B) of the previous
result. Here we only focus on the hardest part (last item on the above list).
If \(\SUPP{\FFf}\) is non-empty, does not have the greatest element and
the metrics of all \(\Gg_s\) can attain arbitrarily small positive values, then
it follows from the proofs of parts (f) and (g) of the invoked proposition that
\(\SUPP{\FFf}\) contains an upper unbounded (in \(\SUPP{\ffF}\)) strictly
increasing sequence \(t_1 \ORD t_2 \ORD \ldots\) such that the diameters of
\(X_{s_n}\) tend to 0. Then, if \(S_n\) for \(n > 0\) are as specified in
the proof of (g) therein, we know that \(\Hh\) is homeomorphic to
\(\prod_{n=1}^{\infty} W_n\) where \(W_n\) is a topologically discrete space of
vertices of \(\PROD{s \in S_n}\). So, \(\Hh\) is locally compact iff all but
a finite number of the spaces \(W_n\) are finite. But \(W_n\) is finite iff
\(S_n\) is finite and all the spaces \(X_s\) with \(s \in S_n\) are finite. This
gives us the necessity of the last condition on the above list (in the discussed
situation). It is also sufficient, because of the above arguments (then all but
a finite number of \(W_n\) will be finite).
\end{proof}

\begin{thm}{exist-metric}
\begin{enumerate}[\upshape(A)]
\item For an arbitrary infinite set \(Y\) of positive reals let \(\mM(Y)\) be
 the l.u.b. (in the class of cardinals) of all cardinals of the form
 \(2^{\card(Y \cap (0,N))}\) where \(N = 1,2,\ldots\) Then:
 \begin{itemize}
 \item \(\card(X) \leq \mM(Y)\) for any AB-HO metric space \((X,d)\) with \(d(X
  \times X) \subset Y \cup \{0\}\); and
 \item there exists an AB-HO complete metric space \((X,d)\) such that \(d(X
  \times X) = Y \cup \{0\}\) and \(\card(X) = \mM(Y)\).
 \end{itemize}
\item If \(Y \subset (0,\infty)\) is an infinite bounded set, then there exists
 an AB-HO complete metric space \((X,d)\) such that \(d(X \times X) = Y \cup
 \{0\}\) and \(\card(X) = 2^{\card(Y)}\). More generally, the same conclusion
 holds if \(Y\) is infinite and there is \(c > 0\) such that \(\card(Y \cap
 (0,c)) = \card(Y)\) (e.g. if \(\card(Y) = 2^{\aleph_0}\)).
\item Let \(Y \subset (0,\infty)\) be an infinite set that is discrete in
 \(\RRR\) (that is, \(Y = \{a_n\dd\ n > 0\}\) where \(a_n > 0\) and
 \(\lim_{n\to\infty} a_n = \infty\)). Any AB-HO metric space \((X,d)\) with
 \(d(X \times X) \subset Y \cup \{0\}\) is Heine-Borel and at most countable.
\end{enumerate}
\end{thm}
\begin{proof}
We begin with the first part of (A). So, let \((X,d)\) be a non-empty AB-HO
metric space with \(d(X \times X) \subset Y \cup \{0\}\). We fix an element \(b
\in X\) and a positive integer \(N\). It follows from \PRO{sphere} that
the spheres \(S(a,r) \df \{x \in X\dd\ d(a,x) = r\}\) are AB-HO as well. But
the metric \(d\) on \(S(a,r)\) for \(r \in (0,N]\) has all positive values in
\(Y_N \df Y \cap (0,2N]\). So, we infer from \THM{card} that either \(S(a,r)\)
is finite (if so is \(Y_N\)) or else \(\card(S(a,r)) \leq 2^{\card(Y_N)} \leq
\mM(Y)\). Since: \(\mM(Y)\) is infinite, \(\bar{B}(a,N) \df \{x \in X\dd\ d(a,x)
\leq N\}\) is contained in \(\{a\} \cup \bigcup_{r \in Y_N} S(a,r)\) and
\(\card(Y_N) \leq \mM(Y)\), we conclude that \(\card(\bar{B}(a,N)) \leq
\mM(Y)\). Consequently, \(\card(X) \leq \mM(Y)\) as well.\par
To show the second part of (A), we apply part (A) of the previus result. We fix
an infinite set \(Y\) and for any \(c \in Y\) we fix a two-point metric space
\(\Gg_c \df (\{0,1\},\rho_c)\) with a unique positive distance \(c\). For
simplicity, let \(\zZ \in \{0,1\}^Y\ (= \prod_{y \in Y} \{0,1\})\) have all
coordinates equal to 0. Note that the collection \(\{\Gg_y\}_{y \in Y}\) is
productable. We take any total order \(\ORD\) on \(Y\) with the following two
properties:
\begin{enumerate}[(ord1)]
\item if \(a\) and \(b\) are two numbers from \(Y\) such that \(a \ORD
 b\), then \(b \leq 2a\);
\item for any \(c > 0\) the set \(Y \cap (0,c)\) is well-ordered by
 \(\ORD\).
\end{enumerate}
To construct such \(\ORD\), it is enough:
\begin{itemize}
\item to declare that \(x \ORD y\) for all \(x\) and \(y\) from \(Y\) such that
 \(x \in [2^m,2^{m+1})\), \(y \in [2^k,2^{k+1})\) where \(k, m \in \ZZZ\)
 satisfy \(k < m\);
\item to well-order (by \(\ORD\)) each of the sets \(Y \cap [2^k,2^{k+1})\)
 where \(k\) runs over all the integers.
\end{itemize}
Having such an order, we apply part (A) of \PRO{prod-metric} (and
\THM{prod-AB-HO}) to conclude that \((X,d) \df \PROd{y \in Y}{\ORD}{\zZ} \Gg_y\)
is an AB-HO metric space (thanks to (ord1)). Moreover, it follows from part (B)
of the proposition just cited that the metric is complete and from (ord2) that
for any \(c > 0\), the set \(\{0,1\}^{Y \cap (0,c)} \times \{0\}^{Y \cap
[c,\infty)} \subset \{0,1\}^Y\) is contained in \(X\) and hence \(\card(X) \geq
2^{\card(Y \cap (0,c))}\), which finishes the proof of (A). Since (B) is
a special case of (A), it remains to prove (C). But (C) follows from the first
part of the proof of (A), where it was shown that in case of such \(Y\) all
closed balls are finite.
\end{proof}

The case of AB-HO ultrametric spaces will be fully discussed in the next
section. Now we pass to AB-HO Heine-Borel metric spaces. For the purposes of
the next result, for a finite metric space \(X\) having more than one point, we
call the least positive distance in \(X\) \emph{co-diameter}.

\begin{thm}{exm-AB-HO}
\begin{enumerate}[\upshape(A)]
\item Let \(\{\Gg_s = (X_s,d_s)\}_{s \in S}\) be a productable collection of
 AB-HO metric spaces and let \(\ORD\) be a total order on \(S\) such that
 \(\card(\SUPP{\FFf}) > 1\) and the product \(\Hh = (Z,\rho) \df \PROD{s \in S}
 \Gg_s\) is a metric space. Then \(\Hh\) is Heine-Borel iff one of the following
 conditions is fulfilled:
 \begin{itemize}
 \item \(\SUPP{\FFf}\) is finite, with the least and the greatest elements, say,
  \(p\) and \(q\) (respectively), the space \(\Gg_q\) is compact, \(\Gg_p\) is
  Heine-Borel and topologically discrete, and all other spaces \(X_s\) are
  finite; or
 \item \(\SUPP{\FFf}\) has the least element, say \(a\), is order-isomorphic to
  \(\NNN\), \(\Gg_a\) is Heine-Borel and topologically discrete, and all other
  spaces \(X_s\) are finite and their diameters tend to 0 as \(s \in
  \SUPP{\FFf}\) increases; or
 \item \(\SUPP{\FFf}\) has the greatest element, say \(b\), is order-isomorphic
  to the set \(\ZZZ_-\) of all negative integers, \(\Gg_b\) is compact and all
  other spaces \(X_s\) are finite and their co-diameters tend to \(\infty\) as
  \(s \in \SUPP{\FFf}\) decreases.
 \item \(\SUPP{\FFf}\) is order isomorhic to \(\ZZZ\), all spaces \(X_s\) are
  finite, their diameters tend to 0 as \(s \in \SUPP{\FFf}\) increases and
  their co-diameters tend to \(\infty\) as \(s \in \SUPP{\FFf}\) decreases.
 \end{itemize}
 In particular, the product \(\Gg_1 \times^* \ldots \times^* \Gg_r\) of a finite
 productable collection \(\{\Gg_1,\ldots,\Gg_r\}\) of AB-HO Heine-Borel metric
 spaces is Heine-Borel as well, provided it is a metric space.
\item Any space listed in items \UP{(a)--(e)} of \REM{H-B} admits a compatible
 metric that makes this space both AB-HO and Heine-Borel.
\end{enumerate}
\end{thm}
\begin{proof}
We begin with (A). First we assume \(\Hh\) is Heine-Borel. Then it follows from
part (B) of \PRO{prod-metric} that:
\begin{itemize}
\item each of the spaces \(\Gg_s\ (s \in S)\) is Heine-Borel as well (by item
 (b) of the result cited above); and hence:
\item if \(s \in S\) is neither the greatest nor the least element of
 \(\SUPP{\FFf}\), then \(\Gg_s\) is finite (by items (c)--(d) therein); and
\item if \(t\) is the greatest element of \(\SUPP{\FFf}\), then \(\Gg_t\) is
 compact (by item (d) therein); and
\item if \(t\) is the least element of \(\SUPP{\FFf}\), then \(\Gg_t\) is
 topologically discrete (by item (c) therein).
\end{itemize}
Recall also that a topologically discrete Heine-Borel space is at most countable
and the image of its metric is discrete in \(\RRR\). Now we will show that:
\begin{itemize}
\item[\((\spadesuit)\)] For any two elements \(p, q \in \SUPP{\FFf}\), the set
 \(K_{p,q} \df \{s \in \SUPP{\FFf} \setminus \{p,q\}\dd\ p \ORD s \ORD q\}\) is
 finite.
\end{itemize}
To show this, we fix two positive reals \(r\) and \(R\) that are attainable as
distances in, respectively, \(\Gg_p\) and \(\Gg_q\). Then it follows from part
(A) of \PRO{prod-metric} that \(d_s(X_s \times X_s) \subset \{0\} \cup
[R/2,2r]\) for all \(s \in K_{p,q}\) and, consequently, also the metric of
\(\Kk_{p,q} \df \PROD{s \in K_{p,q}} \Gg_s\) takes all its values in this set of
reals. But \(\Kk_{p,q}\) is isometrically embeddable into \(\Hh\) (by
\PRO{prod-prod}) and thus it is finite. But this is possible only if \(K_{p,q}\)
is finite, as all the spaces \(X_s\) for \(s \in K_{p,q}\) have more than one
point. So, \((\spadesuit)\) holds, as we claimed. Having this, it now easily
follows that \(\SUPP{\FFf}\) is either finite or order-isomorphic to one of
\(\NNN\), \(\ZZZ_-\) or \(\ZZZ\). It remains to justify behavior of
the diameters and the co-diameters formulated in part (A). The conclusion for
the diameters follows from item (g) of part (A) of \PRO{prod-metric} combined
with part (A) therein. In a similar manner we can prove that the co-diameters
tend to \(\infty\), provided \(\SUPP{\FFf}\) does not have the least element: in
that case if there existed a one-to-one sequence \(s_1 \succcurlyeq s_2
\succcurlyeq \ldots\) of indices from \(\SUPP{\FFf}\) such that the co-diameters
of \(\Gg_{s_n}\) were all upper bounded by a real number, say \(M\),
then---again by part (A) of \PRO{prod-metric}---the diameter of each of these
spaces (\(\Gg_{s_n}\)) would be bounded by \(2M\), and these two properties
would imply that the closed ball in \(\Hh\) of radius \(2M\) would be
non-compact. All these remarks prove the first implication in (A). To show
the reverse one, we use part (C) or item (g) of part (B) (both) of
\PRO{prod-metric} to conclude that for any \(t \in \SUPP{\FFf}\) that is not
the least element of this set the product \(\PROD{s \succcurlyeq t} \Gg_s\) is
compact. And if \(\SUPP{\FFf}\) does not have the least element, the condition
on the divergence of the co-diameters of \(\Gg_s\) (combined with finiteness of
suitable \(\Gg_s\)) easily guarantees that for any \(t \in \SUPP{\FFf}\) that is
not the greatest element of this set the product \(\PROD{s \ORD t}\) has all
closed balls finite (and hence is Heine-Borel). These two remarks reduce
the problem to the GEC product of two Heine-Borel spaces \(X \times^* Y\) with
topologically discrete \(X\) and compact \(Y\). A verification that such a space
(provided it is metric) is indeed Heine-Borel is left to the reader.\par
We pass to (B). We will use (A) to construct AB-HO Heine-Borel metric spaces
with a prescribed set of positive distances.\par
First consider a sequence of spaces \(\Gg_n = (\{0,1\},d_n)\) (with \(n > 0\))
where \(d_n\) has a unique positive value equal to \(2^{-n}\). This sequence
forms a productable family and \(\Cc \df \Gg_1 \times^* \Gg_2 \times^* \ldots\)
is a metric space (thanks to part (A) of \PRO{prod-metric}). It follows that
\(\Cc\) is a compact space homeomorphic to the Cantor set (cf. part (g) of item
(B) of the last cited result).\par
Now for \(n > 0\) let \(\Zz_{-n}\) denote a metric space \((\{0,1\},\rho_n)\)
where a unique positive value of \(\rho_n\) equals \(n\). It follows from (A)
that the GEC \(\Dd \df \PROd{k \in \ZZZ_-}{\leq}{*} \Zz_k\) is Heine-Borel (it
is metric again thanks to part (A) of \PRO{prod-metric}). This is a countable
topologically discrete AB-HO Heine-Borel space whose metric is integer-valued
(and all positive integers are attainable as distances). We could also take
the space of integers as \(\Dd\).\par
For any \(n \in \NNN\) we denote by \(\Ff_n\) the discrete metric space of size
\(n\). (Of course, it is AB-HO and Heine-Borel). Finally, it suffices to
consider the GEC products \(\Ff_n \times^* \Cc\), \(\Dd \times^* \Cc\), \(\Ff_n
\times^* \SSS^d\) and \(\Dd \times^* \SSS^d\) (where, in the last two cases,
the metric of \(\SSS^d\) is rescaled so that the diameter of the sphere is less
than 1) to obtain AB-HO Heine-Borel spaces homeomorphic to all remaining spaces
that are listed in \REM{H-B}.
\end{proof}

\begin{rem}{exm-loc-comp}
In a quite similar manner as presented in the last part of the above proof, we
may show that each locally compact space among those listed in items (b)--(e) of
\COR{loc-comp} admits a compatible metric that makes this space AB-HO, provided
so do discrete topological spaces from item (a) of that result. (Indeed, recall
that the space \((\RRR^n,\frac{d_e}{1+d_e})\) is AB-HO. So, the Euclidean
spaces---considered as topological spaces---may be realised with bounded metrics
as AB-HO.) However, we do not know whether any infinite discrete topological
space of size not greater than \(2^{2^{\aleph_0}}\) admits such a metric.
The results of this section give an affirmative answer for cardinals
\(\aleph_0\) and \(2^{2^{\aleph_0}}\). Here we present an example of
topologically discrete AB-HO metric spaces of arbitrary size not greater than
\(2^{\aleph_0}\). (So, under GCH our problem concerning the existence of
suitable metrics on sets of size less than \(2^{2^{\aleph_0}}\) has affirmative
answer.) To this end, we fix an infinite cardinal \(\alpha \leq 2^{\aleph_0}\)
and consider any set \(S \subset [1,2]\) whose size is precisely \(\alpha\).
Finally, we denote by \(\ORD\) the reverse order to any well-order on \(S\) and
define \(\Hh\) as the GEC product \(\PROD{s \in S} \Gg_s\) where \(\Gg_s\) is
a two-point metric space with a unique positive distance equal to \(s\). Then
\(\Hh\) is an AB-HO metric space (as its colors are contained in \([1,2]\)) and
\(\card(\Hh) = \alpha\), because each subset of \(S\) that is well-ordered by
\(\ORD\) is finite.\par
The above problem (that is, of searching for compatible metrics on topologically
discrete spaces that make these spaces AB-HO) has actually a set-theoretic
spirit, not metric. In fact, if an AB-HO GEC \(\Gg = (V,\kappa)\) has palette of
size not greater than \(2^{\aleph_0}\), then it is structurally equivalent to
an AB-HO metric space with discrete topology. Indeed, it is sufficient to take
any one-to-one function \(\theta\dd \COL{\Gg} \to [1,2]\) in order to define
a metric \(d\) on \(V\) as \(d \df \theta \circ \kappa\) (under the convention
that \(\theta\) sends the pseudocolor of \(\Gg\) onto 0) which makes \(V\)
a topologically discrete AB-HO metric space.\par
The construction presented in the example given above (of an AB-HO metric space
of size \(\alpha\)) generalises to arbitrary infinite cardinals (one only needs
to skip the restriction that \(S \subset [1,2]\)). More precisely, for any
infinite cardinal number \(\alpha\) there exists an AB-HO GEC whose palette and
set of all vertices have size equal to \(\alpha\). Recall that for infinite
palettes, the size \(\beta\) of the set of all vertices of an AB-HO GEC is
always bounded by inequalities \(\alpha \leq \beta \leq 2^{\alpha}\) where
\(\alpha\) is the size of the palette. According to the above remark and
\COR{2toY}, each of these two inequalities is best possible.
\end{rem}

Taking into account the above remarks, we pose the following

\begin{prb}{discrete}
Is it true in ZFC that for any two infinite cardinals \(\alpha\) and \(\beta\)
such that \(\alpha < \beta < 2^{\alpha}\) there exists an AB-HO GEC whose
palette and set of vertices have size, respectively, \(\alpha\) and \(\beta\)?
\end{prb}

\section{Classification of AB-HO GEC's with all triangles isosceles}
\label{sec:classification}

\begin{dfn}{isosceles}
We call a GEC a \emph{triangle} if its size equals 3. By a \emph{triangle in
a GEC \(\Gg\)} or a \emph{triangle of \(\Gg\)} we mean any sub-GEC of \(\Gg\)
which itself is a triangle. A triangle is said to be \emph{isosceles} if its
palette has at most two colors (that is, if at least two of the edges of that
triangle have the same color). We say thay \(\Gg\) has \emph{\UP{(IT)} property}
(for short: \(\Gg\) has \ITP) if all its triangles are isosceles.
\end{dfn}

In this section we will use the machinery of the two previous parts to give
a full classification of all AB-HO GEC's with \ITP. Roughly speaking, all they
are isomorphic to GEC products of ultrahomogeneous (ordinary, finite) graphs
(that satisfy a certain additional, but quite natural and easy to check,
condition). And the factors of the product as well as their order are uniquely
determined by the GEC under consideration. So, we may speak about
the \emph{primary} decomposition of such GEC's. In particular, we will obtain
a full classification (up to isometry) of all AB-HO ultrametric spaces.\par
\ITP\ is a very useful propery, as shown by

\begin{pro}{skelet-ITP}
A skeletoid \(\ffF\) in which all trangles are isosceles is a skeleton.
\end{pro}
\begin{proof}
Fix a finite non-symmetric incomplete GEC \(\Gg = (V,a,b,\kappa)\) affiliated
with \(\ffF\) and take a vertex \(c \in V \setminus \{a,b\}\) such that
\(\kappa(a,c) \neq \kappa(b,c)\). Since all triangles in \(\ffF\) are isosceles,
in any completion of \(\Gg\) to a member of \(\ffF\) the triangle with vertices
\(a,b,c\) must be such, which implies that \(\Delta_{\Gg}(\ffF) \subset
\{\kappa(a,c),\kappa(b,c)\}\). Now the conclusion follows from \PRO{comp}.
\end{proof}

To formulate the main result of the section, we need to introduce a few
additional notions which we now turn to.

\begin{dfn}{2-conn}
Let \(\Gg = (V,\kappa)\) be a GEC and \(c\) be a color from its palette.
A \emph{\(c\)-part} of \(\Gg\) is a graph \(\Gamma_{\Gg}(c)\) on the same set of
vertices as \(\Gg\) (that is, on \(V\)) such that two distinct vertices \(a\)
and \(b\) from \(V\) are adjacent in \(\Gamma_{\Gg}(c)\) iff \(\kappa(a,b) =
c\). \(\Gg\) is said to be:
\begin{itemize}
\item \emph{\(c\)-connected} if \(\Gamma_{\Gg}(c)\) is a connected graph;
\item \emph{\(*\)-connected} if it is \(p\)-connected for any color \(p\) from
 its palette;
\item \emph{bi-connected} if it is \(*\)-connected and its palette consists of
 exactly two colors.
\end{itemize}
Observe that each monochromatic GEC is \(*\)-connected.\par
Let \(G = (V,E)\) be an ordinary finite graph. It induces an edge-coloring
\(\eta_G\dd V \times V\to \{-1,0,1\}\) (with 0 as a pseudocolor) given by
the rules:
\begin{itemize}
\item \(\eta_G(v,v) = 0\) for all \(v \in V\);
\item \(\eta_G(v,w) = 1\) if \(\{v,w\} \in E\);
\item \(\eta_G(v,w) = -1\) otherwise.
\end{itemize}
We call \(G\) \emph{bi-connected} if \((V,\eta_G)\) is a bi-connected GEC. So,
\(G\) is bi-connected if \(\card(V) > 1\) and both \(G\) and the complement
graph of \(G\) are connected graphs. (Then, in fact, \(\card(V) > 3\).)
\end{dfn}

The functions \(\eta_G\) defined above will be used in our main result of this
section.\par
Recall that an ordinary finite graph is \emph{ultrahomogeneous} if any
isomorphism between its two subgraphs extends to an automorphism of that graph
(cf. \DEF{abshm}).

\begin{dfn}{primary}
A \emph{primary collection} is a (countable) set \(\PPp\) with all the following
properties:
\begin{enumerate}[(PC1)]
\item \(\PPp\) consists of finite ultrahomogeneous graphs each of which has at
 least two vertices;
\item each graph from \(\PPp\) is either full or bi-connected;
\item for any \(n > 1\), \(\PPp\) contains a unique full graph on \(n\)
 vertices;
\item for any finite bi-connected ultrahomogeneous graph \(\Gamma\), \(\PPp\)
 contains a unique graph that is isomorphic to \(\Gamma\) or to its complement
 graph.
\end{enumerate}
In other words, a primary collection selects a unique ``model'' for each
(finite) full non-degenerate graph and a single ``representative'' for each pair
of the form \(\{\Gamma,\Gamma^c\}\) where \(\Gamma\) is a (finite) bi-connected
ultrahomogeneous graph and \(\Gamma^c\) denotes its complement graph.\par
A primary collection \(\PPp\) naturally divides into two parts: \(\PPp_1\) and
\(\PPp_2\) consisting of all, respectively, full and bi-connected graphs from
\(\PPp\). The part \(\PPp_2\) divides into two important parts: \(\PPp_2^s\)
consists of all graphs from \(\PPp_2\) that are isomorphic to their complement
graphs, and \(\PPp_2^a\) consisting of all other graphs from \(\PPp_2\). 
\end{dfn}

Now we are able to formulate the main result of this section (below we use
the symbol \(\sqcup\) to denote the disjoint union).

\begin{thm}[Classification of AB-HO GEC's with \ITP]{classification}
Let \(\PPp\) be a primary collection.
\begin{enumerate}[\upshape(A)]
\item Let \((S,\ORD)\) be a totally ordered non-empty set and
 \[\tau\dd S \ni s \mapsto \Gamma_s = (V_s,E_s) \in \PPp\]
 be an arbitrary function. For \(j=1,2\) set \(S^{\tau}_j \df
 \tau^{-1}(\PPp_j)\). Additionally, set \(S^{\tau}_{2s} \df
 \tau^{-1}(\PPp_2^s)\) and \(S^{\tau}_{2a} \df \tau^{-1}(\PPp_2^a)\). Let
 \(\omega\) be any one-to-one function defined on the set \(S^{\tau}_1 \sqcup
 (S^{\tau}_2 \times \{-1,1\}) \sqcup \{0\}\) (with values in an arbitrary set).
 (In the above settings, we call \((S,\ORD,\tau,\omega)\) a \emph{modelling
 quadruple}.) Further, for any \(s \in S\) let \(\kappa_s\) be a function
 defined on \(V_s \times V_s\) as follows: \(\kappa_s(v,v) = \omega(0)\) for any
 \(v \in V_s\) and for distinct \(v\) and \(w\) from \(V_s\):
 \begin{itemize}
 \item \(\kappa_s(v,w) = \omega(s)\) if \(s \in S^{\tau}_1\); or
 \item \(\kappa_s(v,w) = \omega(s,\eta_{\Gamma_s}(v,w))\) if \(s \in
  S^{\tau}_2\).
 \end{itemize}
 Finally, set \(\Gg_s \df (V_s,\kappa_s)\) for \(s \in S\). Then:
 \begin{itemize}
 \item the collection \(\{\Gg_s\}_{s \in S}\) consists of GEC's and is
  productable; and
 \item the GEC \(\Hh = \Hh(S,\ORD,\tau,\omega) \df \PROD{s \in S} \Gg_s\) is
  non-degenerate, AB-HO and has \ITP.
 \end{itemize}
\item If \((S,\ORD,\tau,\omega)\) and \((Z,\preceq,\sigma,\upsilon)\) are two
 modelling quadruples, then the GEC's \(\Hh(S,\ORD,\tau,\omega)\) and
 \(\Hh(Z,\preceq,\sigma,\upsilon)\) are isomorphic iff \(\omega(0) =
 \upsilon(0)\) and there exists an order-isomorphism \(\xi\dd (S,\ORD) \to
 (Z,\preceq)\) such that:
 \begin{enumerate}[\upshape({i}so1)]
 \item \(\tau = \sigma \circ \xi\) (in particular, \(\xi(S^{\tau}_1) =
  Z^{\sigma}_1\), \(\xi(S^{\tau}_{2s}) = Z^{\sigma}_{2s}\) and
  \(\xi(S^{\tau}_{2a}) = Z^{\sigma}_{2a}\));
 \item \(\omega(s) = \upsilon(\xi(s))\) for all \(s \in S^{\tau}_1\);
 \item \(\omega(s,\epsi) = \upsilon(\xi(s),\epsi)\) for all \(s \in
  S^{\tau}_{2a}\) and \(\epsi \in \{-1,1\}\);
 \item \(\{\omega(s,-1),\omega(s,1)\} =
  \{\upsilon(\xi(s),-1),\upsilon(\xi(s),1)\}\) for all \(s \in S^{\tau}_{2s}\).
 \end{enumerate}
 The above GEC's are structurally equivalent iff there exists
 an order-isomor\-phism \(\tau\dd (S,\ORD) \to (Z,\prec)\) that satisfies
 condition \UP{(iso1)} above.
\item For any non-degenerate AB-HO GEC \(\Gg\) with \ITP\ there exists
 a modelling quadruple \((S,\ORD,\tau,\omega)\) such that \(\Gg\) is isomorphic
 to \(\Hh(S,\ORD,\tau,\omega)\).
\item For a non-degenerate AB-HO GEC \(\Gg\) with \ITP\ \tfcae
 \begin{enumerate}[\upshape(i)]
 \item \(\Gg\) is \emph{irreducible}; that is, \(\Gg\) is isomorphic to no GEC
  of the form \(\Gg_1 \times^* \Gg_2\) where \(\Gg_1\) and \(\Gg_2\) are two
  non-degenerate GEC's (and form a productable collection);
 \item there exists a graph \(G = (V,E) \in \PPp\) such that \(\Gg\) is
  structurally equivalent to \((V,\eta_G)\).
 \end{enumerate}
\end{enumerate}
\end{thm}

We divide the proof of the above theorem into separate lemmas and eventually
give its final step. Along the way, we will describe skeletons of AB-HO GEC's
with \ITP\ (see \LEM{4} below).\par
We begin with a simple result from graph theory. Its conclusion is surely known
(but is probably formulated in other terms than GEC products), but we could not
find it in the literature. For the reader's convenience, be give its simple
proof.

\begin{lem}{1}
Let \(G = (V,E)\) be a finite ultrahomogeneous graph such that neither \(G\)
nor its complement graph is a full graph. Then \tfcae
\begin{enumerate}[\upshape(i)]
\item \(G\) is not bi-connected;
\item the GEC \((V,\eta_G)\) is isomorphic to \(\Gg_1 \times^* \Gg_2\) where
 \(\Gg_1\) and \(\Gg_2\) are two non-degenerate monochromatic GEC's (forming
 a productable collection);
\item the GEC \((V,\eta_G)\) is irreducible;
\item \((V,\eta_G)\) contains no triangles whose edges have colors \(-1,-1,1\)
 or else no triangles whose edges have colors \(-1,1,1\).
\end{enumerate}
\end{lem}

Recall that irreducible GEC's have been defined in item (i) of part (D) of
\THM{classification}. As we will see in the proof, instead of ultrahomogeneity
of the graph under consideration it is sufficient to assume only its two-point
homogeneity.

\begin{proof}[Proof of \LEM{1}]
Set \(\Gg = (V,\eta_G)\) and notice that the palette of \(\Gg\) has exactly two
colors (which follows from the preliminary assumptions on \(G\)). We infer from
this property that (ii) and (iii) are equivalent (since members of a productable
collection have disjoint palettes and the palette of a non-degenerate GEC is
non-empty). It may also be easily verified that the GEC of the form \(\Gg_1
\times^* \Gg_2\) as specified in (ii) is not bi-connected and satisfies
appropriate condition (iv) (we skip the proofs) and therefore we only focus on
showing that (ii) follows from (i) as well as from (iv). To this end, assume (i)
holds and denote by \(c \in \{-1,1\}\) a color of \(\Gg\) such that its
\(c\)-part is disconnected. Let \(\Lambda\) be any graph-component of
\(\Gamma_{\Gg}(c)\). We claim it is a full graph on \(m > 1\) vertices. Indeed,
\(m > 1\) because \(c\) is a color of \(\Gg\) and all the graph-components of
\(\Gamma_{\Gg}(c)\) are isomorphic (since \(\Gg\) is homogeneous). And if
certain two distinct vertices \(u\) and \(v\) of \(\Lambda\) were not adjacent,
we could choose any vertex \(w\) of \(\Gg\) that lies on some other
graph-component of \(\Gamma_{\Gg}(c)\) and try to extend a partial morphism of
\(\Gg\) that fixes \(u\) and sends \(v\) onto \(w\), which would be impossible,
because \(v\) lies in the same graph-component of the \(c\)-part of \(\Gg\) as
\(u\) and \(w\) does not so. So, all graph-components of \(\Gg\) are full graphs
on \(m\) vertices. Denoting by \(n\) the number of all graph-components of
\(\Gamma_{\Gg}(c)\), it is now easy to see that \(n > 1\) and that \(\Gg \equiv
\Gg_1 \times^* \Gg_2\) where \(\Gg_1\) and \(\Gg_2\) are monochromatic graphs on
\(m\) and \(n\) vertices (respectively).\par
It remains to show that (iv) is followed by (ii). Since (ii) is equivalent to
(i), it is enough to show that if \(\Gg\) is bi-connected, then it contains
triangles with both configurations of colors specified in (iv), which we will
now do. Fix any color \(c \in \{-1,1\}\) and consider a \(c\)-monochromatic
sub-GEC \(\Gg'\) of \(\Gg\) of maximal size. Then \(\Gg' \neq \Gg\) (as \(\Gg\)
is not monochromatic) and it follows from the graph-connectedness of
the \(c\)-part of \(\Gg\) that there is a vertex \(w\) not within the set of
vertices of \(\Gg'\) that is joint with some vertex \(v\) of \(\Gg'\) by an edge
in color \(c\). On the other hand, it follows from the maximality of \(\Gg'\)
that \(\Gg\) has a vertex, say \(u\), that is joint with \(w\) by an edge in
a color different from \(c\). In this way we have obtained a triangle (whose
vertices are \(u,v,w\)) in \(\Gg\) whose edges have colors \(c,c,-c\), as we
wished.
\end{proof}

\begin{lem}{0}
The product \(\Hh = (W,\mu) = \PROD{s \in S} \Gg_s\) of a productable collection
\(\{\Gg_s = (V_s,\kappa_s)\}_{s \in S}\) of GEC's with \ITP\ has \ITP\ as well.
Morevoer, for any two (distinct or not) colors \(a\) and \(b\) from the palette
of \(\Hh\) \tfcae
\begin{enumerate}[\upshape(i)]
\item \(\Hh\) contains a triangle whose edges have colors \(a,a,b\);
\item either:
 \begin{itemize}
 \item for some \(s \in S\), \(\Gg_s\) contains a triangle whose edges have
  colors \(a,a,b\); or
 \item there are two distinct indices \(s, t \in S\) such that \(s \ORD t\) and
  \(a \in \COL{\Gg_s}\) and \(b \in \COL{\Gg_t}\).
 \end{itemize}
\end{enumerate}
\end{lem}
\begin{proof}
Fix three distinct vertices \(\uU = (u_s)_{s \in S}\), \(\vV = (v_s)_{s \in S}\)
and \(\wW = (w_s)_{s \in S}\) from \(W\). With no loss on generality, we may
(and do) assume that \(s \df \LAB{\uU}{\vV} \ORD \LAB{\uU}{\wW} \ORD t \df
\LAB{\vV}{\wW}\). It then follows from property (aux1) established in the proof
of \THM{prod-AB-HO} that \(\LAB{\uU}{\wW} = s\). So, if \(t = s\), then
\(\mu(\uU,\vV) = \kappa_s(u_s,v_s)\), \(\mu(\uU,\wW) = \kappa_s(u_s,w_s)\) and
\(\mu(\vV,\wW) = \kappa_s(v_s,w_s)\) and \ITP\ of \(\Gg_s\) yields that
the triangle on the vertices \(\uU\), \(\vV\) and \(\wW\) is isosceles. And
otherwise (that is, if \(t \neq s\)) we get that \(v_s = w_s\) and therefore
\(\mu(\uU,\vV) = \kappa_s(u_s,v_s) = \kappa_s(u_s,w_s) = \mu(\uU,\wW)\), which
finishes the proof of \ITP\ of \(\Hh\). Notice that the above argument shows
also that (ii) is implied by (i). The reverse implication is left to the reader
(recall that the palettes of the GEC's \(\Gg_s\) are pairwise disjoint).
\end{proof}

From now on to the end of the proof of \THM{classification}, \(\Gg =
(V,\kappa)\) is a fixed non-degenerate AB-HO GEC with \ITP, and \(\ffF\) and
\(Y\) stand for, respectively, its skeleton and palette. In all our examples of
(finite) GEC's as well as incomplete GEC's (used in the proofs that will follow)
all vertices are painted the pseudocolor of \(\Gg\) and, of course,
the edge-coloring is a symmetric function on pairs of vertices. Therefore,
defining edge-colorings, we will skip all formulas that follow directly from
these two rules.

\begin{dfn}{ord}
For two colors \(a, b \in Y\) of \(\Gg\), we will write \(a \prec_{\Gg} b\) or
simply \(a \prec b\) if \(a \neq b\) and there is a triangle in \(\Gg\) whose
two edges have color \(a\) and the third one is painted \(b\). More precisely,
if \(a \neq b\), then \(a \prec b\) if there are distinct vertices \(x, y, z \in
V\) such that \(\kappa(x,z) = \kappa(y,z) = a\) and \(\kappa(x,y) = b\) (see
Figure~\ref{fig:tri}a).\par
For any color \(a \in Y\), we set
\[I_{\Gg}(a) \df \{b \in Y\dd\ b \prec_{\Gg} a \prec_{\Gg} b\}.\]
\end{dfn}

\begin{lem}{2}
For any \(a, b, c \in Y\):
\begin{enumerate}[\upshape(\(\prec\)1)]
\item \(a = b\ \lor\ a \prec b\ \lor\ b \prec a\);
\item \(a \prec b\ \land\ b \prec c\ \land\ c \neq a \implies a \prec c\);
\item \(\card(I_{\Gg}(a)) \leq 1\).
\end{enumerate}
\end{lem}
\begin{proof}
All the GEC's considered in this proof (apart from \(\Ss_0\) which is very
simple) are illustrated in Figure~\ref{fig:Kabc}. In case of incomplete GEC's,
the vertices of the unpainted edge are painted black (on the illustration).\par
Part (\(\prec\)1) simply follows from \AP\ and \ITP\ of \(\ffF\): if \(a \neq
b\), then an incomplete GEC \(\Ss_0 \df (\{0,1,2\},1,2,\mu)\) with \(\mu(0,1) =
a\) and \(\mu(0,2) = b\) is non-symmetric and affiliated with \(\ffF\) and
therefore it can be completed to a GEC from \(\ffF\) but \ITP\ forces that we
can perform it only with the usage of \(a\) or \(b\) as a color for
the unpainted edge. So, \(a \prec b\) or \(b \prec a\).\par
To see (\(\prec\)2), we consider an incomplete GEC on four vertices (see
Figure~\ref{fig:Kabc}a): \[\Ss_1 \df (\{0,1,2,3\},1,3,\mu)\] where \(\mu(0,1) =
a\), \(\mu(0,2) = \mu(1,2) = b\) and \(\mu(0,3) = \mu(2,3) = c\). It may easily
be verified that this is a non-symmetric GEC affiliated with \(\ffF\) and
that \(c\) is the only possible color that can be used to paint the unpainted
edge (again, remember that \(\ffF\) has \AP\ and \ITP). Then the triangle on
the vertices \(0,1,3\) of the completed GEC witnesses that \(a \prec c\).\par
\begin{figure}
\DIAG{0.25}{\Vertices[NoLabel=false]{circle}{1,2,0}\WE[NoLabel=false](1){3}%
\AddVertexColor{black}{1,3}\EDGE{a}{0}{1}\EDGE{b}{0}{2}\EDGE{b}{1}{2}%
\EDGE{c}{0}{3}\EDGE{c}{2}{3}\VERTCAP{1}{\UP{a. }\Ss_1}}
\DIAGP{b}{c}{a}{b}{b}{c}{a}{a}{\UP{b. }\Ss_2}
\DIAGP{a}{c}{a}{b}{b}{a}{a}{c}{\UP{c. }\Ss_3}
\DIAGF{\EDGE{a}{X}{Y}\EDGE{b}{X}{Z}\EDGE{c}{X}{W}\EDGE{b}{Y}{Z}\EDGE{c}{Y}{W}%
\EDGE{c}{Z}{W}}{Z}{\UP{d. }\Kk(a,b,c)} \DIAGK{\pmb{a}}{\UP{e. }\Kk_1(a,b)}
\DIAGK{\pmb{b}}{\UP{f. }\Kk_2(a,b)}
\caption{GEC's from the proof of \LEM{2}: \(\Ss_1\), \(\Ss_2\), \(\Ss_3\) (all
three incomplete), \(\Kk(a,b,c)\), \(\Kk_1(a,b)\) and \(\Kk_2(a,b)\).}
\label{fig:Kabc}
\end{figure}
We turn to (\(\prec\)3). It follows from the previous paragraph that:
\begin{equation}\label{eqn:abc}
a \prec b\ \land\ b \prec c\ \land\ a \neq c \implies \Kk(a,b,c) \in \ffF
\end{equation}
where \(\Kk(a,b,c)\) is a GEC shown in Figure~\ref{fig:Kabc}d. Further, it
follows from both \AP\ and \ITP\ that:
\begin{equation}\label{eqn:ab}
b \in I_{\Gg}(a) \implies \Kk_1(a,b) \in \ffF\ \lor\ \Kk_2(a,b) \in \ffF
\end{equation}
where \(\Kk_1(a,b)\) and \(\Kk_2(a,b)\) are GEC's shown in
Figures~\ref{fig:Kabc}e and \ref{fig:Kabc}f. Indeed, \eqref{eqn:ab} follows,
since a non-symmetric incomplete GEC obtained by erasing from any of these two
GEC's the edge whose color is printed in bold (on the illustration) is
affiliated with \(\ffF\) (and \(\ffF\) has \ITP).\par
Now we claim that:
\begin{equation}\label{eqn:k2}
b \neq a\ \land\ \Kk_2(a,b) \in \ffF \implies I_{\Gg}(a) = \{b\}.
\end{equation}
Indeed, if \eqref{eqn:k2} holds, then \(b \in I_{\Gg}\) (since \(\ffF\) is
hereditary). Assume, on the contrary, that there is \(c \in I_{\Gg}(a)\) that
differs from \(b\). Then \(c \neq a\) and \(b \prec c \prec a\), thanks to
(\(\prec\)2). So, we infer from \eqref{eqn:abc} that \(\Kk(b,c,a) \in \ffF\).
But then a non-symmetric incomplete GEC \(\Ss_2\) shown in
Figure~\ref{fig:Kabc}b is affiliated with \(\ffF\) (after erasing the upper or
lower vertex filled in black we obtain a GEC isomorphic to, respectively,
\(\Kk(b,c,a)\) or \(\Kk_2(a,b)\)), but cannot be completed to a GEC that has
\ITP. This proves \eqref{eqn:k2}.\par
We are now ready to give a proof of (\(\prec\)3). Assume, on the contrary, that
there are two distinct colors \(b\) and \(c\) in \(I_{\Gg}(a)\). Then both they
are different from \(a\) and we infer from \eqref{eqn:k2} that neither
\(\Kk_2(a,b)\) nor \(\Kk_2(a,c)\) belongs to \(\ffF\). So, it follows from
\eqref{eqn:ab} that both \(\Kk_1(a,b)\) and \(\Kk_1(a,c)\) are in \(\ffF\). But
then a non-symmetric incomplete GEC \(\Ss_3\) shown in Figure~\ref{fig:Kabc}c is
affiliated with \(\ffF\) (without the upper or lower vertex painted in black we
obtain a GEC isomorphic to, respectively, \(\Kk_1(a,c)\) or \(\Kk_1(a,b)\)) and
cannot be completed to a GEC with \ITP, which finishes the proof.
\end{proof}

\begin{lem}{3}
Let \(a\) be a color of \(\Gg\) and \(\sigma = \sigma(a) \df I_{\Gg}(a) \cup
\{a\}\). There exists a unique (up to isomorphism) finite AB-HO GEC
\(\Pp_{\sigma}\) with at least two vertices and \ITP\ such that
\(\COL{\Pp_{\sigma}} = \sigma\) and the following condition is fulfilled:
\begin{itemize}
\item[\((\clubsuit)\)] For any finite (non-empty) GEC \(\Ff\), \(\Ff \in
 \FIN{\Pp_{\sigma}}\) iff \(\Ff \in \ffF\) and \(\COL{\Ff} \subset \sigma\).
\end{itemize}
Moreover, if \(\card(\sigma) = 1\), it is monochromatic, and otherwise it is
bi-connected.
\end{lem}
\begin{proof}
Uniqueness of \(\Pp_{\omega}\) and its non-degeneracy follow, e.g., from
(respectively) \THM{embed} and the fact that \(a\) belongs to the palette of
\(\ffF\). Further, if \(\card(\sigma) = 1\) (that is, if \(\sigma = \{a\}\)),
the conclusion trivially follows from \MBP\ and in that case we skip the proof.
Below we assume that \(\card(\sigma) \neq 1\). This means that \(\sigma =
\{a,b\}\) for a unique \(b \in I_{\Gg}(a)\) (thanks to \LEM{2}).\par
Let \(\ffF_0\) consist of all GEC's from \(\ffF\) whose palettes are contained
in \(\sigma\). We claim that \(\ffF_0\) is a skeleton. Indeed, axioms
(SK1)--(SK3) simply follow from these properties of \(\ffF\). It is also clear
that \(\ffF_0\) has \ITP. Hence, we conclude from \PRO{skelet-ITP} that
\(\ffF_0\) is a skeleton if it has \AP. To check the last mentioned property,
fix a non-symmetric incomplete GEC \(\Hh' = (W,p,q,\mu)\) affiliated with
\(\ffF_0\) and take a vertex \(v \in W \setminus \{p,q\}\) such that \(\mu(p,v)
\neq \mu(q,v)\). Then \(\Hh'\) is also affiliated with \(\ffF\) which means that
it can be completed to a GEC \(\Hh\) from this class. But \ITP\ (of \(\ffF\))
implies that the color (in \(\Hh\)) of the unpainted edge belongs to the set
\(\{\mu(p,v),\mu(q,v)\}\) which is contained in \(\sigma\) and, consequently,
\(\Hh \in \ffF_0\).\par
Now let \(\Pp_{\omega}\) denote the AB-HO GEC generated by \(\ffF_0\). Since
\(a \prec b \prec a\), we infer that \(\ffF_0\) contains two triangles shown
in Figure~\ref{fig:tri}. In particular, \(\COL{\Pp_{\sigma}} = \sigma\) and
moreover \(\Pp_{\sigma}\) is finite. It follows from \LEM{1} that
\(\Pp_{\sigma}\) is bi-connected. So, it remains to check \((\clubsuit\)), which
is immediate.
\begin{figure}
\DIAG{0.25}{\Vertices{circle}{X,Y,Z}\EDGE{a}{X}{Y}\EDGE{a}{X}{Z}\EDGE{b}{Y}{Z}%
\VERTCAP[,Ldist=0mm]{X}{\UP{a. }a \prec b}}
\DIAG{0.25}{\Vertices{circle}{X,Y,Z}\EDGE{b}{X}{Y}\EDGE{b}{X}{Z}\EDGE{a}{Y}{Z}%
\VERTCAP[,Ldist=0mm]{X}{\UP{b. }b \prec a}}
\caption{Triangles witnessing that \(a \prec b \prec a\) for \(a \neq b\) (see
\LEM{4} and the proof of \LEM{3})}\label{fig:tri}
\end{figure}
\end{proof}

\begin{lem}{4}
Let \(\Omega = \Omega(\Gg)\) be a collection of all sets of the form
\(\sigma(a)\) (specified in \LEM{3}). For each \(\sigma \in \Omega\) let
\(\Pp_{\sigma} = \Pp_{\sigma}(\Gg)\) be a unique AB-HO GEC described in \LEM{3}.
Then \(\Omega\) is a partition of \(Y\) and for any finite (non-empty) GEC \(\Hh
= (W,\mu)\) \tfcae
\begin{enumerate}[\upshape(i)]
\item \(\Hh \in \ffF\);
\item \(\COL{\Hh} \subset Y\) and \(\Hh\) has \ITP\ and the following two
 properties:
 \begin{itemize}
 \item if \(a\) and \(b\) belong to two distinct members of \(\Omega\) and
  the triangle shown in Figure~\ref{fig:tri}a belongs to \(\FIN{\Hh}\), then
  \(a \prec b\);
 \item for any set \(\sigma \in \Omega\), each non-empty sub-GEC \(\Hh'\) of
  \(\Hh\) with \(\COL{\Hh'} \subset \sigma\) belongs to \(\FIN{\Pp_{\sigma}}\).
 \end{itemize}
\end{enumerate}
\end{lem}
\begin{proof}
It is immediate that \(\Omega\) is a partition of \(Y\) (that is, it consists of
non-empty pairwise disjoint sets that cover \(Y\)) and that (ii) follows from
(i). Below we will proof that the reverse implication also holds.\par
Let \(\Hh\) be a finite non-empty GEC with all properties desribed in (ii).
First observe that:
\begin{itemize}
\item[\((\dag)\)] If \(a\) and \(b\) are two distinct colors such that
 the triangle shown in Figure~\ref{fig:tri}a belongs to \(\FIN{\Hh}\), then
 \(a \prec b\).
\end{itemize}
(Because either \(a\) and \(b\) are in a common member of \(\Omega\)---and then
the last relation is automatic---or else this relation is assumed in (ii).) Note
also that if \(\card(\COL{\Hh}) \leq 1\), then (i) holds (thanks to
\((\clubsuit)\) and the last assumption in (ii)). So, we assume further that
\(\card(\COL{\Hh}) > 1\). We prove (i) by induction on the number \(n\) of
vertices of \(\Hh\). The cases \(n=1,2\) are covered by the last observation.
Thus, we assume that \(n > 2\) and that each proper non-empty sub-GEC of \(\Hh\)
belongs to \(\ffF\) (since all the properties listed in (ii) are
hereditary).\par
First assume that \(\Hh\) contains a triangle \(\Tt\) shown in
Figure~\ref{fig:tri}a with distinct colors \(a\) and \(b\) such that
the relation \(b \prec a\) does not hold. Denote by \(u\) and \(v\) the two
vertices of any of the edges of \(\Tt\) that is painted \(a\). It then follows
from the inductive hypothesis that an incomplete GEC \(\Hh'\) obtained from
\(\Hh\) by declaring the edge \(\{u,v\}\) unpainted is affiliated with \(\ffF\).
It is also non-symmetric, which the third vertex of \(\Tt\) witnesses to. So, we
may complete \(\Hh'\) to a member of \(\ffF\). But \ITP\ combined with the form
of \(\Tt\) implies that we can do it only with the usage of \(a\) or \(b\) (as
a color for the unpainted edge). However, \(b\) is not allowed since it is not
true that \(b \prec a\). So, the only possible completion of \(\Hh'\) in
\(\ffF\) is \(\Hh\) and hence \(\Hh \in \ffF\).\par
Finally, we assume that \(\Hh\) has at least two colors and satisfies
the following condition:
\begin{itemize}
\item[\((\ddag)\)] If a triangle shown in Figure~\ref{fig:tri}a with distinct
 \(a\) and \(b\) belongs to \(\FIN{\Hh}\), then \(b \prec a\).
\end{itemize}
Observe that if in (any) GEC all the triangles are monochromatic, then this GEC
itself is monochromatic. So, we conclude that \(\Hh\) contains a triangle
\(\Tt\) that is not monochromatic. Since \(\Hh\) has \ITP, this triangle is
two-color. We denote the colors of its edges by \(a\) and \(b\) in a way such
that two edges are painted \(a\). We infer from \((\dag)\) and \((\ddag)\) that
\(a \prec b\) and \(b \prec a\). Hence \(\sigma \df \{a,b\}\) is a member of
\(\Omega\) (see the last part of \LEM{2}). Thanks to the last assumption in
(ii), to conclude that \(\Hh \in \ffF\), it is sufficient to show that
\(\COL{\Hh} \subset \sigma\), which we now turn to.\par
For simplicity, we denote the vertices of \(\Tt\) by \(u\), \(v\) and \(w\) in
a way such that \(\mu(u,w) = \mu(v,w) = a\) and \(\mu(u,v) = b\). Let \(x\) be
an arbitrary vertex of \(\Hh\) distinct from \(u\). If \(x \in \{v,w\}\), then
\(\mu(x,u) \in \sigma\). And if \(x\) differs from both \(v\) and \(w\), then
thanks to \ITP\ of \(\Hh\) there are two possibilities:
\begin{itemize}
\item \(\mu(x,u) = \mu(u,w)\), then automatically \(\mu(x,u) \in \sigma\); or
 else
\item \(c \df \mu(x,u)\) differs from \(\mu(u,w) = a\) and then \(\mu(x,v) \in
 \{a,c\}\), which implies that \(a \prec c\) as well as \(c \prec a\) (by
 \((\dag)\) and \((\ddag)\)); consequently, \(c \in \sigma(a) = \sigma\).
\end{itemize}
In this way we have shown that \(\mu(x,u) \in \sigma\) for any \(x \in W
\setminus \{u\}\). Now assume \(p\) and \(q\) are arbitrary two distinct points
of \(W\). If \(u \in \{p,q\}\), then the property just established implies that
\(\mu(p,q) \in \sigma\). And if \(u \notin \{p,q\}\), there are two
possibilities:
\begin{itemize}
\item \(\mu(p,q) = \mu(p,u)\), then \(\mu(p,q) \in \sigma\) (by the previous
 case); or else
\item \(c \df \mu(p,q)\) differs from \(d \df \mu(p,u) \in \sigma\) and then
 \(\mu(q,u) \in \{c,d\}\), which implies (similarly as before) that \(d \prec
 c\) as well as \(c \prec d\); consequently, \(c \in \sigma(d) = \sigma\),
\end{itemize}
which finishes the proof.
\end{proof}

Now we will apply the ideas presented in preceding lemmas to a more specific
AB-HO GEC's:

\begin{lem}{5}
Let \(\Ww = (W,\mu)\) denote the product \(\PROd{s \in S}{\preceq}{*} \Hh_s\) of
a productable collection \(\{\Hh_s = (V_s,\rho_s)\}\) of non-degenerate AB-HO
GEC's each of which is either bi-connected or monochromatic. For any color \(c
\in \COL{\Ww}\) let \(t(c)\) stand for a unique index \(s \in S\) such that \(c
\in \COL{\Hh_s}\). Then:
\begin{enumerate}[\upshape(a)]
\item for any two distinct colors \(a\) and \(b\) from \(\COL{\Ww}\) one has:
 \[a \prec_{\Ww} b \iff t(a) \preceq t(b);\]
\item \(\Omega(\Ww) = \{\COL{\Hh_s}\dd\ s \in S\}\);
\item \(\Pp_{\sigma}(\Ww) \equiv \Hh_s\) for any \(s \in S\) and \(\sigma =
 \COL{\Hh_s}\).
\end{enumerate}
\end{lem}
\begin{proof}
We start from a remark that \(\Ww\) has \ITP\ (by \LEM{0}; notice that each of
\(\Hh_s\) also has \ITP, having at most two colors in its palette) and therefore
we can speak about \(\prec_{\Ww}\) and other notions appearing in the statement
of the lemma. Now part (a) follows from the equivalence formulated in \LEM{0}
(involved here when \(t(a) \neq t(b)\)) combined with the equivalence of
conditions (i) and (iv) from \LEM{1} (involved in the other case). Item (b) is
a swift consequence of (a) and therefore here we focus only on part (c). To this
end, fix \(q \in S\), set \(\sigma \df \COL{\Hh_q}\) and notice that simply
\(\Hh_q \in \FIN{\Pp_{\sigma}(\Ww)}\). To show the reverse relation, that is,
that \(\Tt \df \Pp_{\sigma}(\Ww) \in \FIN{\Hh_q}\), consider \(\Tt\) as
a sub-GEC of \(\Ww\) and observe that for any two distinct vertices \(\xX =
(x_s)_{s \in S}\) and \(\yY = (y_s)_{s \in S}\) of \(\Tt\), \(\mu(\xX,\yY) \in
\sigma\), which implies that \(\LAB[\preceq]{\xX}{\yY} = q\) (as the palettes of
the factors of the product of GEC's are pairwise disjoint) and, consequently,
\(\mu(\xX,\yY) = \rho_q(x_q,y_q)\). So, the assignment \(\xX = (x_s)_{s \in S}
\mapsto x_q\) correctly defines a morphism of \(\Tt\) into \(\Hh_q\). Since both
\(\Tt\) and \(\Hh_q\) are finite GEC's, the conclusion follows.
\end{proof}

\begin{lem}{6}
Maintaining the notation introduced in \LEM{4}, let \(\ORD_{\Gg}\) be a binary
relation on \(\Omega = \Omega(\Gg)\) defined as follows:
\begin{equation}\label{eqn:ord0}
\sigma_1 \ORD_{\Gg} \sigma_2 \stackrel{\UP{def}}{\iff} \sigma_1 = \sigma_2\
\lor\ \forall a \in \sigma_1,\ b \in \sigma_2\dd\ a \prec_{\Gg} b \qquad
(\sigma_1, \sigma_2 \in \Omega).
\end{equation}
Then \(\ORD_{\Gg}\) is a total order on \(\Omega\), the collection
\(\{\Pp_{\sigma}(\Gg)\}_{\sigma\in\Omega}\) is productable and the GEC \(\Ww =
(W,\mu) \df \PROd{\sigma\in\Omega}{\ORD_{\Gg}}{*} \Pp_{\sigma}(\Gg)\) is AB-HO,
has \ITP\ and all the following properties:
\begin{enumerate}[\upshape(pr1)]\addtocounter{enumi}{-1}
\item \(\COL{\Ww} = \COL{\Gg}\) and the pseudocolors of \(\Ww\) and \(\Gg\)
 coincide;
\item \(\prec_{\Ww} = \prec_{\Gg}\) (that is, for any \(a, b \in Y\), \(a
 \prec_{\Ww} b\) is equivalent to \(a \prec_{\Gg} b\)); in particular,
 \(\Omega(\Ww) = \Omega(\Gg)\);
\item \(\Pp_{\sigma}(\Ww) \equiv \Pp_{\sigma}(\Gg)\) for all \(\sigma \in
 \Omega(\Gg)\).
\end{enumerate}
\end{lem}
\begin{proof}
It follows from \LEM{2} and the definition of \(\Omega\) that \(\ORD\) is
a total order, whereas from \THM{prod-AB-HO} and \LEM{0} that \(\Ww\) is
an AB-HO GEC with \ITP\ (it is clear that the respective collection is
productable). Since (pr0) is trivial, we pass to (pr1). First assume \(a
\prec_{\Gg} b\). Then:
\begin{itemize}
\item if \(b \prec_{\Gg} a\) as well, then \(\sigma \df \{a,b\}\) belongs to
 \(\Omega\) and \(\Pp_{\sigma}(\Gg)\) contains a triangle whose edges are
 painted \(a,a,b\), which implies that \(a \prec_{\Ww} b\);
\item otherwise, \(a \in \sigma_a\) and \(b \in \sigma_b\) for two distinct
 members of \(\Omega\) such that \(\sigma_a \ORD \sigma_b\), and \(x \in
 \COL{\Pp_{\sigma_x}(\Gg)}\) for \(x \in \{a,b\}\)---in that case the relation
 \(a \prec_{\Ww} b\) follows from \LEM{0}.
\end{itemize}
The reverse implication (that is, that \(a \prec_{\Gg} b\) follows from \(a
\prec_{\Ww} b\)) is an immediate consequence of part (a) of \LEM{5}, whereas
(pr2) follows from part (c) therein.
\end{proof}

We can now formulate the main two parts of \THM{classification} in separate
results.

\begin{thm}[Primary decomposition]{model}
Every non-degenerate AB-HO GEC \(\Gg\) with \ITP\ is isomorphic to \(\Ww \df
\PROd{\sigma\in\Omega(\Gg)}{\ORD_{\Gg}}{*} \Pp_{\sigma}(\Gg)\).
\end{thm}

The above product is called \emph{primary decomposition} of \(\Gg\).

\begin{proof}[Proof of \THM{model}]
Since both the GEC's \(\Gg\) and \(\Ww\) are AB-HO and have \ITP\ (by
\THM{prod-AB-HO} and \LEM{0}), \COR{iso-AB-HO} and \LEM{4} apply. But the latter
result combined with \LEM{6} show that skeletons of these two GEC's coincide and
the proof is finished by an application of the corollary just cited.
\end{proof}

\begin{thm}{model-iso}
Let \(\{\Tt_s\}_{s \in S}\) and \(\{\Ss_t\}_{t \in T}\) be two productable
collections consisting of non-degenerate AB-HO GEC's each of which has \ITP\ and
is either bi-connected or monochromatic. Let, in addition, \(\ORD\) and
\(\preceq\) be two total orders on \(S\) and \(T\), respectively. Then \tfcae
\begin{enumerate}[\upshape(i)]
\item the GEC's \(\PROD{s \in S} \Tt_s\) and \(\PROd{t \in T}{\preceq}{*}
 \Ss_t\) are isomorphic;
\item there exists an order-isomorphism \(\phi\dd (S,\ORD) \to (T,\preceq)\)
 such that \(\Ss_{\phi(s)} \equiv \Tt_s\) for all \(s \in S\).
\end{enumerate}
\end{thm}
\begin{proof}
The implication ``(ii)\(\implies\)(i)'' (which holds for totally arbitrary
productable collections of GEC's) is left to the reader. Here we focus only on
the reverse implication. To prove it, we denote by \(\Hh\) and \(\Ww\),
respectively, the products \(\PROD{s \in S} \Tt_s\) and
\(\PROd{t \in T}{\preceq}{*} \Ss_t\) and take any isomorphism \(\Phi\dd \Hh \to
\Ww\). Since \(\Phi\) transforms any sub-GEC \(\Hh'\) of \(\Hh\) onto a sub-GEC
of \(\Ww\) isomorphic to \(\Hh'\), it follows that (in order) \(X \df \COL{\Hh}
= \COL{\Ww}\), \(\prec_{\Hh} = \prec_{\Ww}\), \(\Sigma \df \Omega(\Hh) =
\Omega(\Ww)\), \(\ORD_{\Hh} = \ORD_{\Ww}\) (cf. \eqref{eqn:ord0}) and
\(\Pp_{\sigma}(\Hh) \equiv \Pp_{\sigma}(\Ww)\) for any \(\sigma \in \Sigma\).
We set \(\nu_{\Hh}\dd S \ni s \mapsto \COL{\Tt_s} \in \Sigma\) and
\(\nu_{\Ww}\dd T \ni t \mapsto \COL{\Ss_t} \in \Sigma\). We infer from part (a)
of \LEM{5} and \eqref{eqn:ord0} that these two functions are order-isomorphisms
(from, respectively, \((S,\ORD)\) and \((T,\preceq)\) onto
\((\Sigma,\ORD_{\Hh})\)). So, \(\phi \df \nu_{\Ww}^{-1} \circ \nu_{\Hh}\dd
(S,\ORD) \to (T,\preceq)\) is a well defined order-isomorphism as well. Now for
any \(s \in S\) and \(\sigma \df \nu_{\Hh}(s)\), we get \(\nu_{\Ww}(\phi(s)) =
\sigma\) and hence an application of part (c) of \LEM{5} yields \(\Tt_{\phi(s)}
\equiv \Pp_{\sigma}(\Ww) \equiv \Pp_{\sigma}(\Hh) \equiv \Ss_s\), which finishes
the proof.
\end{proof}

Finally, we are ready to give a proof of the main result of this section.

\begin{proof}[Proof of \THM{classification}]
All that is eseentially missing is a construction of a modelling quadruple for
the product of the form specified, e.g., in \LEM{5}. To describe it, we begin
with a crucial observation:
\begin{quote}\itshape
If \(\Ss\) is a non-degenerate AB-HO GEC that is either bi-connected or
monochromatic, then there exists a \underline{unique} graph \(\Gamma = (V,E)\)
from \(\PPp\), to be denoted by \(\Gamma_{\Ss} = (V_{\Ss},E_{\Ss})\), such that
\(\Ss\) is structurally equivalent to the GEC \((V,\eta_{\Gamma})\).
\end{quote}
(Recall that structural equivalence consists of repainting GEC's.) The above
property follows from axioms (PC3)--(PC4) defining a primary collection.
(Indeed, observe that for any graph \(G = (V,E)\), the GEC's \((V,\eta_G)\) and
\((V,\eta_{G^c})\) are structurally equivalent where \(G^c\) is the complement
graph of \(G\).)\par
Now we fix a productable collection \(\{\Tt_s\}_{s \in S}\) of non-degenerate
AB-HO GEC's each of which has \ITP\ and is either bi-connected or monochromatic
as well as any total order \(\ORD\) on \(S\). Let \(e\) stand for a common
pseudocolor of these GEC's. We define \(\tau\dd S \to \PPp\) by \(\tau(s) \df
\Gamma_{\Tt_s}\) and respective sets \(S^{\tau}_1\) and \(S^{\tau}_2\),
introduced in \THM{classification}. Next, for any \(s \in S^{\tau}_2\) we fix
a structural isotopy \(\beta_s\dd \{-1,0,1\} \to \COL{\Tt_s} \cup \{e\}\)
between \((V_{\Tt_s},\eta_{\Gamma_{\Tt_s}})\) and \(\Tt_s\). Finally, we define
the last ingredient of the modelling quadruple---the function \(\omega\)---by
the rules:
\begin{itemize}
\item \(\omega(0) = e\);
\item for \(s \in S^{\tau}_1\), \(\omega(s)\) is the unique color in the palette
 of \(\Tt_s\);
\item for \(s \in S^{\tau}_2\) and \(\epsi \in \{-1,1\}\), \(\omega(s,\epsi) \df
 \beta_s(\epsi)\).
\end{itemize}
We leave this as an exercise that \((S,\ORD,\tau,\omega)\) is a modelling
quadruple such that \(\Hh(S,\ORD,\tau,\omega)\) is isomorphic to
\(\PROD{s \in S} \Tt_s\). Now we discuss the proofs of all parts of
the theorem.\par
The first conclusion of (A) is readily seen, whereas the second one follows from
\LEM{0} (cf. axioms (PC1)--(PC4) defining primary collections). Part (B) is
a consequence of \THM{model-iso}. Item (C) is implied by \THM{model} and
previous paragraphs of the present proof. Finally, part (D) follows from
a combination of (C), \PRO{prod-prod} and \LEM{1}. The details are left to
the reader.
\end{proof}

For the purposes of the next result, let us set \(I_n \df \{1,\ldots,n\}\) for
\(n > 0\) and denote by \(\delta_n\) the discrete metric on \(I_n\). Below
\(\geq\) denotes the reverse order to the natural one on \((0,\infty)\).

\begin{cor}[Classification of AB-HO ultrametric spaces]{ultrametric}
There is a `canonical' one-to-one correspondence between functions from
\((0,\infty)\) to the set of all positive integers \(\NNN_1 \df \NNN \setminus
\{0\}\) and all non-empty AB-HO ultrametric spaces (considered up to isometry).
Precisely, each such a space \((Z,\rho)\) is isometric to a metric space of
the form
\begin{equation}\label{eqn:ultra-prod}
\PROd{t>0}{\geq}{*} (I_{f(t)},t \delta_{f(t)})
\end{equation}
where \(f\dd (0,\infty) \to \NNN_1\) is an arbitrary function, uniquely
determined by \((Z,\varrho)\). Conversely, each space of the form
\eqref{eqn:ultra-prod} is non-empty, ultrametric and AB-HO.
\end{cor}
\begin{proof}
Of course, each ultrametric space has \ITP\ as a GEC. Using the notation
introduced in \DEF{ord}, it is clear that in any ultametric space \(\Gg =
(X,d)\) the relation \(a \prec b\) (between positive reals as colors) may hold
only when \(b < a\). In particular, in such spaces we never have \(a \prec b
\prec a\), which implies that in case the above \(\Gg\) is non-degenerate and
AB-HO:
\begin{itemize}
\item the set \(\Omega(\Gg)\) consists of one-point sets;
\item the order \(\ORD_{\Gg}\) on \(\Omega(\Gg)\) coincides with the reverse
 order to the natural one (more precisely: \(\{a\} \ORD_{\Gg} \{b\} \iff b \leq
 a\));
\item the primary decomposition of \(\Gg\) has the form
 \(\PROd{t \in R}{\geq}{*} (I_{u(t)},t \delta_{u(t)})\) where \(R \df d(X \times
 X) \setminus \{0\}\) and \(u(t) > 1\) (for \(t \in R\)) denotes the size of
 a maximal \(t\)-monochromatic sub-GEC of \(\Gg\);
\item the above function \(u\dd R \to \NNN_1 \setminus \{1\}\) is uniquely
 determined by \(\Gg\).
\end{itemize}
So, to finish the proof, it suffices to extend the above \(u\) to a function
\(g\dd (0,\infty) \to \NNN_1\) so that \(g(t) \df 1\) for \(t \notin R\).
Observe that then for \(\ffF \df \{(I_{g(t)},t \delta_{g(t)})\}_{t>0}\) we have
\(\SUPP{\ffF} = R\) (for the definition of \(\SUPP{\ffF}\) see the paragraph
preceding \PRO{prod-metric}). On the other hand, \LEM{0} guarantees that for any
function \(f\dd (0,\infty) \to \NNN_1\) the GEC of the form
\eqref{eqn:ultra-prod} is an ultrametric space. All these remarks are sufficient
to conclude the whole assertion of this corollary. We skip the details.
\end{proof}

\begin{cor}{ultrametric-size}
Every AB-HO ultrametric space is complete and either finite or countable, or has
size \(2^{\aleph_0}\). In particular, countable AB-HO ultrametric spaces are
topologically discrete.
\end{cor}
\begin{proof}
Assume \(X\) is an infinite AB-HO ultrametric space. Consider the primary
decomposition \(\PROd{r \in R}{\geq}{*} (I_{f(r)},r \delta_{f(r)})\) (see
\COR{ultrametric}) of \(X\) into finite non-de\-gen\-erate factors \(I_{f(r)}\)
where \(R\) is the set of all positive distances attainable in \(X\) and \(f\dd
R \to \NNN \setminus \{0,1\}\) is a function uniquely determined by \(X\). In
what follows, we will identify the space \(X\) with the underlying space of
the aforementioned product. That \(X\) is complete follows from part (B) of
\PRO{prod-metric}.\par
If \((R,\leq)\) is well-ordered, then it is countable (as a well-ordered subset
of \(\RRR\)) and then each well-ordered subset of \((R,\geq)\) is finite and,
consequently, \(X\) is countable (note also that in that case \(X\) is
discrete). Otherwise (that is, if \((R,\leq)\) is not well-ordered) \((R,\geq)\)
contains a countably infinite well-ordered subset \(J\). In that case \(X\)
contains a ``natural'' copy of the set \(\prod_{r \in J} I_{f(r)}\) whose size
is \(2^{\aleph_0}\). On the other hand, each well-ordered subset of \((R,\geq)\)
is countable and the number of such subsets does not exceed \(2^{\aleph_0}\),
which implies that the size of \(X\) is exactly \(2^{\aleph_0}\).
\end{proof}

\begin{rem}{isosceles-metric}
\THM{classification} combined with part (A) of \PRO{prod-metric} enables us to
characterise all AB-HO metric spaces in which each triangle is isosceles. These
are precisely the `GEC' products of finite AB-HO metric spaces each of which
attains at most two positive distances (and when two such distances are
attained, say \(a < b\), they satisfy the inequality \(b \leq 2a\) and
the factor needs to contain both configurations of triangles: \(a,a,b\) and
\(b,a,a\)). However, a precise description is much more complicated than
\eqref{eqn:ultra-prod} and hence we skip here the details.
\end{rem}

\section{Cardinal characteristics of skeletoids}\label{sec:card}

This short part is motivated by the following result due to Menger
\cite{mg1,mg2} (consult also \cite{b-b} for a modern exposition of this topic
and an interesting discussion).

\begin{thm}[Menger]{Menger}
For a metric space \((X,d)\) and a positive integer \(n\), \tfcae
\begin{enumerate}[\upshape(i)]
\item \((X,d)\) is isometrically embeddable into \((\RRR^n,d_e)\);
\item every subspace of \(X\) of size at most \(n+3\) isometrically embeds into
 \((\RRR^n,d_e)\).
\end{enumerate}
\end{thm}

To make our work as complete as possible, below we present a counterpart of
the above result for spheres and hyperbolic spaces. Its proof equally works for
the case of Euclidean spaces. (Our proof is a little bit different that the one
presented in \cite{b-b}.) We do not know whether this result is already known
(but we believe it is so).

\begin{thm}{ht-hyp-sph}
Let \(n > 0\) and \((M_n,\rho)\) denote one of \((H^n(\RRR),d_h)\) or
\((\SSS^n,d_s)\). For a metric space \((X,d)\) \tfcae
\begin{enumerate}[\upshape(i)]
\item \((X,d)\) is isometrically embeddable into \((M_n,\rho)\);
\item every subspace of \(X\) of size at most \(n+3\) embeds isometrically into
 \((M_n,\rho)\).
\end{enumerate}
\end{thm}
\begin{proof}
Of course, only the implication ``(ii)\(\implies\)(i)'' needs proving. Denoting
\(\SSS^0 \df \{-1,1\}\) (the unit sphere in \(\RRR\), equipped with \(d_s\)) and
\(H^0(\RRR) \df \{0\}\), we proceed by induction on \(n \geq 0\). Since for \(n
= 0\) the conclusion trivially holds, we now assume that \(n > 0\) and
the assertion holds for \(n-1\). In what follows, we will consider \(M_n\) as
a subspace of a real vector space \(V_n\) where \(V_n = \RRR^{n+1}\) for \(M_n =
\SSS^n\) and \(V_n = H^n(\RRR)\) otherwise. We will consider \(V_n\) with
a standard inner product (cf. \eqref{eqn:inner}). Let \(H\) stand for
the orthogonal group of \(V_n\). Recall also that \(\Iso(\SSS^n) = O_{n+1}\) and
\(\OPN{stab}_G(0) = O_n\) for \(G = \Iso(H^n(\RRR))\). Thanks to the induction
hypothesis, we may and do assume that not every at most \(n+2\) subspace of
\(X\) embeds isometrically into \(M_{n-1}\). So, there is a set \(A \subset X\)
such that \(\card(A) \leq n+2\) and \(A\) does not embed into \(M_{n-1}\).
However, it follows from our assumption in (ii) that \(A\) does embed into
\(M_n\). So, let \(\phi\dd A \to M_n\) be an isometric map. Fixing \(b \in A\)
and using the homogeneity of \(M_n\), we may (and do) additionally assume that
\(\phi(b) = 0\), provided \(M_n = H^n(\RRR)\). A key point now is that the set
\(\phi(A)\) must contain a linear basis of \(V_n\). For if not, then the linear
span \(W\) of \(\phi(A)\) is a proper linear subspace of \(V_n\) and hence there
exists \(u \in H\) such that \(u(W) \subset V_{n-1}\). But \(u\restriction{M_n}
\in \Iso(M_n)\) and hence \(u \circ \phi\dd A \to M_{n-1}\) is a correctly
defined isometric map, which contradicts our assumption about \(A\). So, there
are \(a_1,\ldots,a_d \in A\) such that \(\phi(a_1),\ldots,\phi(a_d)\) is
a linear basis of \(V_n\). Set
\[A_0 \df \begin{cases}\{a_1,\ldots,a_d\} & M_n = \SSS^n\\\{b,a_1,\ldots,a_d\} &
M_n = H^n(\RRR)\end{cases},\]
\(m \df \card(A_0)\ (=n+1)\) and \(\phi_0 \df \phi\restriction{A_0}\dd A_0 \to
M_n\). Note that each at most \((m+2)\)-point subset of \(X\) embeds
isometrically into \(M_n\). What is more,
\begin{itemize}
\item[\((*)\)] for any \(p,q \in X\), there exists a \textbf{unique} isometric
 map \(\phi_{p,q}\dd A_0 \cup \{p,q\} \to M_n\) that extends \(\phi_0\).
\end{itemize}
Indeed, the existence follows from the absolute homogeneity of \(M_n\) (and
the last statement preceding \((*)\)), and to show the uniqueness, assume
\(\psi\) and \(\xi\) are two (suitable) extensions of \(\phi_0\). Then
\(\xi \circ \psi^{-1}\) is a well defined partial isometry and, as such, it
extends to a global isometry \(u\dd M_n \to M_n\). Then \(\xi = u \circ \psi\)
and, in particular, \(u\) fixes all points of \(\phi_0(A_0)\). It follows from
the previous remarks that either \(u\) extends to an element of \(H\) (if \(M_n
= \SSS^n\)), which we will still denote by \(u\), or \(u \in H\) (if \(M_n =
H^n(\RRR)\)---because then \(u(0) = 0\)). So, \(u\) is a linear mapping. But
\(\phi_0(A_0)\) contains a linear basis of \(V_n\), and therefore \(u\)
coincides with the identity map. Consequently, \(\xi = \psi\) and the proof of
\((*)\) is finished.\par
It follows from the uniqueness in \((*)\) that \(\phi_{x,y}\) extends both
\(\phi_{x,x}\) and \(\phi_{y,y}\) for all \(x, y \in X\). So, the assignment
\(x \mapsto \phi_{x,x}(x)\) defines an isometric map from \(X\) into \(M_n\) and
we are done.
\end{proof}

The above two results inspired us to introduce

\begin{dfn}{ht}
Let \(\ffF\) be a skeletoid. The \emph{height} of \(\ffF\), denoted by
\(\OPN{ht(\ffF)}\), is:
\begin{itemize}
\item a minimal integer \(n \geq 0\) such that a finite (non-empty) GEC \(\Gg\)
 belongs to \(\ffF\) iff every its (non-empty) sub-GEC on at most \(n+3\)
 vertices belongs to \(\ffF\)---if at least one such \(n\) exists;
\item \(\infty\), if no such \(n\) as specified above exists.
\end{itemize}
For any \(n > 0\), we denote by \(\ffF_{[n]}\) the class of all GEC's \(\Gg \in
\ffF\) which have exactly \(n\) vertices. So, up to isomorphism, \(\ffF_{[1]}\)
contains at most one GEC, members of \(\ffF_{[2]}\) are in a natural one-to-one
correspondece with colors from the palette of \(\ffF\) and \(\ffF_{[3]}\)
contains all triangles in \(\ffF\).\par
The \emph{rank} of \(\ffF\), denoted by \(\OPN{rk}(\ffF)\), is defined as
the supremum of the sizes decreased by 1 of all monochromatic GEC's in \(\ffF\),
unless \(\ffF = \varempty\):
\[\OPN{rk}(\ffF) \df \sup\{|\Gg|-1\dd\ \Gg \in \ffF\UP{ monochromatic}\}\ \in
\NNN \cup \{\infty\}.\]
Additionally, we set \(\OPN{rk}(\varempty) \df 0\).
\end{dfn}

\begin{pro}{ht-rk}
For any skeletoid \(\ffF\), one has \(\OPN{rk}(\ffF) \leq \OPN{ht}(\ffF)+1\).
\end{pro}
\begin{proof}
The conclusion is significant only when \(h \df \OPN{ht}(\ffF)\) is finite. In
such a case if \(\ffF\) contained a \(c\)-monochromatic graph (for some color
\(c\)) of size \(h+3\), then each finite (non-empty) \(c\)-monochromatic graph
would belong to \(\ffF\) (by the definition of the height), which would
contradict \MBP\ for \(\ffF\). So, any monochromatic graph from \(\ffF\) has
size at most \(h+2\), and the conclusion follows.
\end{proof}

\begin{exm}{ht}
\begin{enumerate}[(A)]
\item Let \(\ffF\) be a skeletoid. Observe that (below \(\Gg\) stands for
 a non-empty finite GEC):
 \begin{itemize}
 \item if the statement ``\(\Gg \in \ffF\)'' is equivalent to ``any sub-GEC of
  \(\Gg\) of size 1 belongs to \(\ffF\)'', then \(\ffF = \varempty\);
 \item if the statement ``\(\Gg \in \ffF\)'' is equivalent to ``any sub-GEC of
  \(\Gg\) of size 1 or 2 belongs to \(\ffF\)'', then, \(\ffF = \ffF_{[1]}\) (so,
  \(\ffF\) is a skeleton generating a degenerate AB-HO GEC).
 \end{itemize}
 The above remarks explain why it is reasonable to define the height of
 a skeletoid in a way such that height zero corresponds to sub-GEC's of size 3
 (or less).
\item It follows from \THM[s]{Menger} and \ref{thm:ht-hyp-sph} that
 \(\OPN{ht}(\FIN{\Gg}) \leq n\) where \(\Gg\) is any of \((\RRR^n,d_e)\),
 \((\SSS^n,d_s)\) and \((H^n(\RRR),d_h)\). Menger himself proved that actually
 \(\OPN{ht}(\FIN{\Gg}) = n\) (he used different terminology). It is quite
 an easy exercise to verify that \(\OPN{rk}(\RRR^n,d_e) = n\) and
 \(\OPN{rk}(\SSS^n,d_s) = n+1\). The latter equality, combined with \PRO{ht-rk},
 yields \(\OPN{ht}(\SSS^n,d_s) = n\). At this moment we are unable to give
 a proof that the same equation holds for the hyperbolic spaces. We plan to
 study it in a close future.\par
 The example of spheres witnesses that the inequality in \PRO{ht-rk} is optimal.
\item The previous example may suggest that the height of the skeleton of
 an AB-HO metric space is closely related to the dimension of this space. To
 convince oneself that this is not the case, it suffices to consider, e.g.,
 the GEC product \(K \df \PROd{n\in\NNN_1}{\geq}{*} (I_n,2^{-n} \delta_n)\) (cf.
 \COR{ultrametric}). It is a compact AB-HO ultrametric space (and hence its
 topological dimension is zero), whereas \(\OPN{rk}(\FIN{K}) = \infty\) and,
 consequently, \(\OPN{ht}(\FIN{K}) = \infty\) as well (by \PRO{ht-rk}).
\item It follows from the very definition of the rank that the equality
 \(\OPN{rk}(\ffF) = 1\) is equivalent to the statement that \(\COL(\ffF) \neq
 \varempty\) and that \(\ffF\) contains no monochromatic triangles. When we
 restrict our attention only to skeletons of the AB-HO GEC's with \ITP, we can
 conclude from \THM{classification} that such a skeleton has rank 1 iff it is
 isomorphic to the product of ultrahomogeneous finite GEC's each of which is
 either monochromatic of size 2 or bi-connected without monochromatic triangles.
 It is an interesting combinatorial exercise that, up to structural equivalence,
 there is exactly one such bi-connected GEC. It is shown in Figure~\ref{fig:1}.
 (As an ordinary graph, it is isomorphic to its complement graph). Since this
 example is only a digression, we skip the proof of the above statement about
 uniqueness.
\end{enumerate}
\begin{figure}
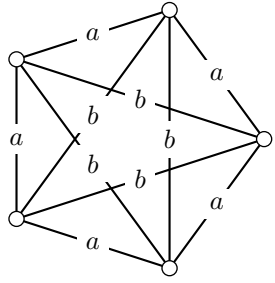

\DIAGQ\caption{A unique bi-connected ultrahomogenoues graph without
monochromatic triangles}\label{fig:1}
\end{figure}
\end{exm}

\section{Directed graphs}\label{sec:directed}

The aim of this section is to generalise the concepts and results from previous
parts to the context of directed graphs. In order not to bore the reader, we
will focus exclusively on those aspects that contribute something significantly
new to the theory.\par
Recall that a \underline{directed} graph is a pair \((V,E)\) where \(V\) is
a set (of vertices) and \(E\) is a subset of \(V \times V\) disjoint from
the diagonal \(\Delta_V\) of \(V\) (again, \(E\) is the set of edges). Since
\(E\) consists of \underline{ordered} pairs, each edge has its beginning and
end and between any two distinct vertices a directed graph may have zero, one or
two edges (in different directions). In directed graphs we speak about
\emph{outcoming} (from a vertex) and \emph{incoming} (to a vertex) edges;
namely: an edge of the form \((x,y)\) is outcoming from \(x\) (and \(x\) is
the beginning of this edge) and incoming to \(y\) (and \(y\) is its end).
Undirected graphs (discussed in the previous parts) are considered as directed
ones usually as graphs that have zero or two edges between any two distinct
vertices (and these two edges are treated as equally valuable).\par
In a similar manner as we have done it for undirected graphs, we will now define
(full) directed GEC's.

\begin{dfn}{dGEC}
A \emph{directed GEC} (briefly: a \emph{dGEC}) is a pair \((V,\kappa)\) where
\(V\) is a set and \(\kappa\) is a function defined on \(V \times V\) which
satisfies (GEC1) from \DEF{GEC}.\par
In a similar manner as for GEC's, one defines: \emph{degenerate} (and
\emph{non-degenerate}) and \emph{monochromatic} dGEC's;  the \emph{palette},
the \emph{pseudocolor}, the \emph{size} and a \emph{sub-dGEC} of a dGEC as well
as \emph{morphisms}, \emph{partial morphisms} and \emph{isomorphisms} between
two dGEC's (and an \emph{automorphism} of a dGEC). Finally, we verbatim
\DEF{absGEC} to define \emph{AB-HO} dGEC's. Next come the definitions of
a \emph{skeleton}, \emph{productable collections} and their \emph{products}.\par
Similarly as for GEC's, an \emph{incomplete} dGEC is a quadruple \(\Gg = (V,a,b,
\kappa)\) where \(\kappa\) is defined everywhere on \(V \times V\) apart from
two edges: \((a,b)\) and \((b,a)\), and satisfies (GEC1). To \emph{complete}
the above incomplete dGEC means to assign values (that is, colors) to the two
unpainted edges. The above \(\Gg\) is \emph{non-symmetric} if there exists \(c
\in V \setminus \{a,b\}\) such that
\[\kappa(a,c) \neq \kappa(b,c) \qquad \UP{or} \qquad \kappa(c,a) \neq
\kappa(b,a).\]
Next one defines \emph{finitary structures} of dGEC's, incomplete dGEC's
\emph{affiliated with classes of dGEC's}; and \MBP, \AP\ and its variants.\par
For a non-symmetric incomplete dGEC \(\Gg = (V,a,b,\kappa)\) affiliated with
a proper finitary structure \(\ffF\) one can also define \(\Delta_{\ffF}(\Gg)\)
as the set of all possible pairs \((c,d)\) such that setting \(\kappa(a,b) \df
c\) and \(\kappa(b,a) \df d\) we obtain a completion of \(\Gg\) that belongs to
\(\ffF\).
\end{dfn}

Directed GEC's may be visualised as undirected full graphs in which every edge
is two-color or monochromatic: visualising an (undirected!) edge \(\{x,y\}\) as
a line segment joining two its ends (\(x\) and \(y\)), we paint the half of this
segment between \(x\) and the middle point of the segment color of the directed
edge \((x,y)\) of the given dGEC and the other half color of \((y,x)\). In this
way one obtains a one-to-one correspondence between dGEC's and GEC's where each
edge has two colors (or one).\par
As in dGEC's a new phenomenon occurs (related to \MBP), we will define
skeletoids after introducing a new notion and establishing an important property
related to this phenomenon.

\begin{dfn}{OBP}
Let \(\Gg = (V,\kappa)\) be a dGEC and \(c\) and \(d\) be two distinct colors
from the palette of \(\Gg\). A sub-dGEC \(\Hh = (W,\kappa)\) of \(\Gg\) and its
set \(W\) of vertices are said to be \emph{\((c,d)\)-ordered} if there exists
a total order \(\preceq\) on \(W\) such that for any two distinct vertices \(u,
v \in W\):
\begin{equation}\label{eqn:OBP}
u \preceq v \iff \kappa(u,v) = c \iff \kappa(v,u) = d.
\end{equation}
Notice that the order \(\preceq\) is uniquely determined by \eqref{eqn:OBP} and
thus we call it the \emph{\((c,d)\)-order of \(W\)} or \emph{of \(\Hh\)}.
\(\Hh\) is said to be \emph{bi-ordered} if it is \((a,b)\)-ordered for some pair
\((a,b)\) of distinct colors. (If this is so, then \(\COL{\Hh} = \{a,b\}\).)\par
We say that a finitary structure \(\ffF\) satisfies \emph{order boundedness
principle} (briefly: \OBP) if for any two distinct colors \(c\) and \(d\) from
the palette of \(\ffF\) there exists a positive integer \(\mu = \mu(c,d)\) such
that any \((c,d)\)-ordered dGEC belonging to \(\ffF\) has size bounded by
\(\mu\).
\end{dfn}

And now the result just announced.

\begin{thm}{OBP}
\begin{enumerate}[\upshape(A)]
\item If a proper finitary structure \(\ffF\) of dGEC's satisfies both
 the conditions \MBP\ and \OBP, then there exists a cardinal \(\mM\) such that
 each dGEC finitely represented in \(\ffF\) has size not greater than \(\mM\).
\item All skeletons satisfy \MBP\ and \OBP. The size of an AB-HO dGEC whose
 palette has size \(\alpha \geq \aleph_0\) does not exceed \(2^{\alpha}\).
 An AB-HO dGEC is of finite size iff so is its palette.
\end{enumerate}
\end{thm}
\begin{proof}
Even though part (A) follows from \THM{r-e-r} (similarly as \LEM{bdd}), we give
the proofs of both the parts, as (A) requires new ideas.\par
Let \(\ffF\) be as specified in (A) and \(\Gg = (V,\kappa)\) be a non-degenerate
dGEC finitely represented in \(\ffF\). Let \(\alpha\) stand for the cardinality
of the palette \(Y\) of \(\ffF\). We claim that:
\begin{enumerate}[(c{a}rd1)]
\item if \(\alpha\) is finite, then so is the set \(V\);
\item if \(\alpha \geq \aleph_0\), then \(\card(V) \leq 2^{\alpha}\).
\end{enumerate}
To this end, denote by \(e\) the pseudocolor of \(\ffF\) and fix any total order
\(\ORD\) on \(V\). We define a new edge-coloring \(\mu\dd V \times V \to Y
\times Y\) as follows: \(\mu(x,x) \df e\) and for distinct \(x, y \in V\) such
that \(x \ORD y\), \(\mu(x,y) = \mu(y,x) \df (\kappa(x,y),\kappa(y,x))\). In
this way we obtain a GEC \((V,\mu)\). Since \(\mu\) is a symmetric function, we
can make use of \THM{r-e-r}.\par
First we assume \(\alpha < \aleph_0\). If (card1) is false (that is, if the set
\(V\) is infinite), it follows from the Ramsey's theorem that there exists
an infinite set \(Z \subset V\) that is monochromatic with respect to \(\mu\).
Denote a common color (w.r.t.\ \(\mu\)) of all edges of \(Z\) by \((a,b)\). If
\(a = b\), \(Z\) is \(a\)-monochromatic w.r.t.\ \(\kappa\). Since \(Z\) is
infinite and finitely represented in \(\ffF\), we get a contradiction with \MBP.
On the other hand, if \(b \neq a\), then \(Z\) is \((a,b)\)-ordered (w.r.t.\
\(\kappa\)). Indeed, observe that for two distinct vertices \(p, q \in Z\) we
either have \(p \ORD q\) (and then \(\kappa(p,q) = a\) and \(\kappa(q,p) = b\))
or \(q \ORD p\) (and then \(\kappa(p,q) = b\) and \(\kappa(q,p) = a\)), which
shows that \(\ORD\) on \(Z\) satisfies \eqref{eqn:OBP}. So, \OBP\ does not hold
in \(\ffF\), which means that (card1) is true.\par
In the remaining case (to obtain (card2)) we argue similarly, but we would apply
the Erd\H{o}s-Rado theorem instead of Ramsey's if (card2) were false.\par
We pass to (B). Observe that the second and third statements of (B) follow from
the combination of (card1)--(card2) and the first statement (of (B)) and thus we
focus only on showing that skeletons satisfy \MBP\ and \OBP. As the proof of
part (a) of \THM{card} works also for AB-HO dGEC's, we only need to check \OBP.
To this end, we fix an AB-HO dGEC \(\Gg = (V,\kappa)\) and two different colors
\(c\) and \(d\) from its palette. \OBP\ for the skeleton of \(\Gg\) is
equivalent to the statement that all \((c,d)\)-ordered sub-dGEC's of \(\Gg\)
have sizes uniformly bounded by a finite number. To show the latter propery, we
argue by a contradiction. So, we assume that there are finite \((c,d)\)-ordered
sub-dGEC's \(\Ff_1,\Ff_2,\ldots\) of \(\Gg\) such that \(|\Ff_n| < |\Ff_{n+1}|\)
for all \(n\). Similarly as in the proof of part (a) of \THM{card}, we infer
from the absolute homogeneity of \(\Gg\) that then there exists
a \((c,d)\)-ordered set \(A \subset V\) such that \(A\) equipped with its
\((c,d)\)-order is order-isomorphic to \(\NNN\). So, \(A\) has the least element
(w.r.t.\ this order), which we denote by \(a\). There also exists a strictly
increasing (w.r.t.\ the \((c,d)\)-order of \(A\)) function \(\psi\dd A \to A\)
such that
\begin{equation}\label{eqn:notin}
\psi(A) = A \setminus \{a\}.
\end{equation}
It follows from Zorn's lemma that in the family \(\LLl\) of all
\((c,d)\)-ordered sets \(B \subset V\) such that
\begin{itemize}
\item \(A \subset B\); and
\item if \(x \in B\) and \(\kappa(a,x) = c\), then \(x \in A\) (that is, \(A\)
 must coincide with the set of all elements from \(B\) that are greater than or
 equal to \(a\) w.r.t.\ the \((c,d)\)-order of \(B\))
\end{itemize}
there exists a maximal set, say \(C\). Denote by \(\ORD\) the \((c,d)\)-order of
\(C\). The latter property of \(C\) implies that the function \(\phi\dd C \to
C\) given by
\[\phi(x) \df \begin{cases}\psi(x) & x \in A\\x & x \notin A\end{cases}\]
is strictly increasing w.r.t.\ \(\ORD\). Consequently, \(\phi\) is a partial
morphism in \(\Gg\). Observe that \(\phi(C) = C \setminus \{a\}\) (thanks to
\eqref{eqn:notin}). So, we can find a vertex \(b \in V \setminus C\) such that
\(\phi\) can be extended to a partial morphism by setting \(\phi(b) \df a\).
Then \(\phi(D) = C\) where \(D \df C \cup \{b\}\). Hence \(D \equiv C\), which
implies that \(D\) is \((c,d)\)-ordered. Moreover, \(\phi\dd D \to C\) (as
an isomorphism of two \((c,d)\)-ordered dGEC's) is an order-isomorphism w.r.t.
the \((c,d)\)-orders. Since \(C \subset D\), these two orders are compatible and
hence we will use (again) \(\ORD\) to denote the \((c,d)\)-order on \(D\). Then
\(\phi(b) = a \ORD \psi(a) = \phi(a)\), which implies that \(b \ORD a\). Since
\(b \neq a\), we see that \(D \in \LLl\) and thus \(C\) is not a maximal set in
\(\LLl\). This contradiction finishes the proof of the theorem.
\end{proof}

The above result shows that in the realm of dGEC's a counterpart of \MBP\ (from
the realm of GEC's) is the tandem of \MBP\ and \OBP. Thus, we put the following

\begin{dfn}{skeletoid}
A \emph{skeletoid} is any proper finitary structure of dGEC's that satisfies
both \MBP\ and \OBP, and has \AP.\par
The \emph{rank} of a non-empty skeletoid \(\ffF\) of dGEC's, denoted by
\(\OPN{rk}(\ffF)\), is defined as:
\[\OPN{rk}(\ffF) \df \sup\{|\Gg|-1\dd\ \Gg \in \ffF\UP{ monochromatic or
bi-ordered}\}\ \in \NNN \cup \{\infty\}.\]
\end{dfn}

The \emph{height} of a skeletoid is defined in the very same way as for
GEC's.\par
Now we list results that can be proven in almost the same way as for GEC's. That
is why we omit their proofs and leave them to interested readers.

\begin{thm}{embedd}
Let \(\Gg\) be an AB-HO dGEC. A dGEC admits a morphism into \(\Gg\) iff every
its finite (non-empty) sub-dGEC admits a morphism into \(\Gg\).\par
Two AB-HO dGEC's \(\Gg_1\) and \(\Gg_2\) are isomorphic iff \(\FIN{\Gg_1} =
\FIN{\Gg_2}\).
\end{thm}

\begin{thm}{dskeleton}
A class is a skeleton iff it is a skeletoid with \AP[\(*\)-].
\end{thm}

\begin{pro}{d*AP}
A skeletoid with infinite palette has \AP[\(*\)-] iff it has \AP[\(\mM\)-] where
\(\mM\) is the size of its palette. In particular, a skeletoid with countable
palette is a skeleton iff it has \AP[\(\omega\)-].
\end{pro}

\begin{pro}{dsphere}
Let \(\Gg = (V,\kappa)\) be an AB-HO dGEC, \(S\) be a non-empty subset of \(V\),
and let \(r_{in},r_{out}\dd S \to \COL{\Gg}\) be two arbitrary functions. Let
\(W\) consist of all \(x \in V\) such that \(\kappa(x,y) = r_{in}(y)\) and
\(\kappa(y,x) = r_{out}(y)\) for any \(y \in S\). Then \((W,\kappa)\) is
a (possibly degenerate) AB-HO dGEC.
\end{pro}

\begin{thm}{dproduct}
The product of a productable collection of AB-HO dGEC's is an AB-HO dGEC as
well.
\end{thm}

\begin{pro}{dht-rk}
For any skeletoid \(\ffF\), one has \(\OPN{rk}(\ffF) \leq \OPN{ht}(\ffF)+1\).
\end{pro}

As we announced in Section~3, all groups can naturally be considered as dGEC's.
We will now discuss this idea.

\begin{dfn}{group-dGEC}
Let \((G,\cdot)\) be any group. Each of the following four functions \(L_G,
L^*_G,R_G,R^*_G\dd G \times G \to G\) may serve as an edge-coloring of \(G\)
(with the neutral element of \(G\) as the pseudocolor):
\[L_G(a,b) \df a^{-1}b, \quad L^*_G(a,b) \df b^{-1}a, \quad R_G(a,b) \df
ab^{-1}, \quad R^*_G(a,b) \df ba^{-1}.\]
We call them (respectively) \emph{left}, \emph{left\(*\)}, \emph{right} and
\emph{right\(*\) dGEC-structures of \(G\)}.
\end{dfn}

\begin{pro}{dgroup}
Let \((G,\cdot)\) be a group and let \(\Xi\dd G \to G\) be given by \(\Xi(x) \df
x^{-1}\).
\begin{enumerate}[\upshape(a)]
\item \(\Xi\) is an isomorphism from the left dGEC-structure onto the right one
 of \(G\) (and vice versa) as well as from the left\(*\) dGEC-structure onto
 the right\(*\) one (and vice versa).
\item \(\Xi\) is a structural isotopy from the left dGEC-structure onto
 the left\(*\) one (and vice versa) as well as from the right dGEC-structure
 onto the right\(*\) one (and vice versa).
\item All these four dGEC-structures of \(G\) coincide iff \(G\) is Boolean.
\item All these four dGEC-structures of \(G\) are isomorphic iff \(G\) is
 Abelian.
\item In all these four dGEC-structures of \(G\), all edges outcoming from
 a single vertex (as well as all incoming to it) have different colors.
\item In each of these four dGEC-structures, \(G\) is an AB-HO dGEC with
 the property that each morphism of a non-empty sub-dGEC of \(G\) into \(G\)
 extends to a \textbf{unique} automorphism of the dGEC \(G\). Automorphisms of
 both the left and left\(*\) (resp. of both the right and right\(*\))
 dGEC-structures are precisely left (resp. right) translations of \(G\).
\item Fixing a type of dGEC-structures of groups among the above four and
 denoting by \(S\) the respective edge-coloring among \(L,L^*,R,R^*\), two
 groups \(G\) and \(H\) are isomorphic as groups iff the dGEC's \((G,S_G)\) and
 \((H,S_H)\) are structurally equivalent. Moreover, structural isotopies between
 these two dGEC's are precisely group isomorphisms.
\end{enumerate}
\end{pro}
\begin{proof}
As all parts of the proposition other than (d) and (g) are trivial (or may
simply be deduced from the proofs of respective parts of \THM{skton-grp}), we
will give proofs of (d) and (g) only.\par
We start from (d). If \(G\) is Abelian, then \(L_G = R^*_G\) and the conclusion
follows from (a). Conversely, assume \(\phi\dd (G,L_G) \to (G,L^*_G)\) is
an isomorphism of dGEC's. This means that
\begin{equation}\label{eqn:aux9}
\forall x, y \in G\dd\quad \phi(y)^{-1}\phi(x) = x^{-1}y.
\end{equation}
Setting \(c \df \phi(e_G)\) (where \(e_G\) is the neutral element of \(G\))
and substituting \(e_G\) for \(y\) in \eqref{eqn:aux9}, we get that \(\phi(x) =
c x^{-1}\). Now substituting the last formula into \eqref{eqn:aux9} yields
\(yx^{-1} = x^{-1}y\), which means that \(G\) is Abelian.\par
Now we turn to (g). If \(\Phi\dd G \to H\) is a group isomorphism, then it is
also an isomorphism between dGEC's \((G,\Phi \circ S_G)\) and \((H,S_H)\), which
implies that \(\Phi\) is a structural isotopy. Conversely, assume \(\Phi\dd G
\to H\) is a structural isotopy and take an isomorphism \(\phi\dd (G,\Phi \circ
S_G) \to (H,S_H)\) of dGEC's. For clarity, below we assume \(S = L\) (arguments
for the remaining three cases closely mirror the following). Then we obtain:
\[\phi(x)^{-1} \phi(y) = \Phi(x^{-1}y)\]
for all \(x, y \in G\). Setting \(c \df \phi(e_G)\) and substituting in
the above equation \(e_G\) for \(x\) we get \(\phi(y) = c \Phi(y)\).
Substitution of the last formula to the equation above yields \(\Phi(x^{-1}y) =
\Phi(x)^{-1}\Phi(y)\), which is equivalent to the property that \(\Phi\) is
a group homomorphism.
\end{proof}

Parts (b) and (d) of the above result motivate us to introduce

\begin{dfn}{dual}
The \emph{dual} dGEC of a dGEC \(\Gg = (V,\kappa)\) is a dGEC \(\Gg^* =
(V,\kappa^*)\) where \(\kappa^*(x,y) \df \kappa(y,x)\). The dGEC \(\Gg\) is said
to be \emph{group-like} if \(\Gg\) and \(\Gg^*\) are structurally equivalent.
It is called \emph{Abelian-like} if these dGEC's are isomorphic.
\end{dfn}

Observe that the equation \(\Gg = \Gg^*\) characterises GEC's \(\Gg\) among all
dGEC's.\par
As for GEC's (cf. \THM{skton-grp}), property (e) formulated in \PRO{dgroup}
characterise groups among AB-HO dGEC's, as shown by

\begin{thm}{dskton-grp}
For a non-empty dGEC \(\Gg = (V,\kappa)\) \tfcae
\begin{enumerate}[\upshape(i)]
\item \(\Gg\) is an AB-HO dGEC in which all edges outcoming from a single vertex
 have different colors;
\item \(\Gg\) is an AB-HO dGEC in which all edges incoming to a single vertex
 have different colors;
\item \(\Gg\) is isomorphic (as a dGEC) to a group equipped with the left
 dGEC-structure.
\end{enumerate}
\end{thm}

It is worth noticing that the proof presented below works for all homogeneous
dGEC's that satisfy (i) or (ii). In particular, all such dGEC's are
automatically AB-HO (by (iii)).

\begin{proof}[Proof of \THM{dskton-grp}]
We have already established (in \PRO{dgroup}) that both (i) and (ii) are implied
by (iii).\par
Assume (i) holds. We will show (iii). To this end, we denote by \(e\)
the pseudocolor of \(\Gg\), set \(G \df \COL{\Gg} \cup \{e\}\), \(H \df
\AUT{\Gg}\) and fix a vertex \(a\) from \(V\). It follows from (i) that
the function \[\Psi\dd V \ni x \mapsto \kappa(a,x) \in G\] is a bijection.
Moreover, it may easily be verified that in that case also the function
\(\Theta\dd H \ni \phi \mapsto \phi(a) \in V\) is bijective. For simplicity, for
\(x \in V\) we denote by \(\phi_x\) the automorphism \(\Theta^{-1}(x)\), and set
\(\Phi \df \Psi \circ \Theta\). For any \(x,y \in V\) we have
\begin{equation}\label{eqn:aux8}
\kappa(x,y) = \kappa(\phi_x(a),\phi_y(a)) =
\kappa(a,(\phi_x^{-1} \circ \phi_y)(a)) = \Phi(L_H(\phi_x,\phi_y)).
\end{equation}
So, it suffices to transfer the binary operation from \(H\) into \(G\) by
\(\Phi\) to get a group \((G,*)\) such that \(\Phi\dd (H,\circ) \to (G,*)\) is
a group isomorphism (precisely, for \(g,h \in G\) we set \(g * h \df
\Phi(\Phi^{-1}(g) \circ \Phi^{-1}(h))\)) and to conclude that \(\Psi\dd
(V,\kappa) \to (G,L_G)\) is an isomorphism of dGEC's. Indeed, \eqref{eqn:aux8}
yields:
\[\kappa(x,y) = \Phi(\phi_x^{-1} \circ \phi_y) = \Phi(\phi_x)^{-1} *
\Phi(\psi_y) = L_G(\Phi(\phi_x),\Phi(\phi_y)) = L_G(\Psi(x),\Psi(y)).\]
So, \(\Gg\) is isomorphic to a group with the left dGEC-structure, as we
wished.\par
In order to show that also (ii) is followed by (iii), it is enough to observe
that if (i) holds, then the dual dGEC \(\Gg^*\) satisfies (i) and therefore
\(\Gg^*\) is isomorphic to a group with the left dGEC-structure. So, we infer
from \PRO{dgroup} that all edges incoming to a single vertex have different
colors in \(\Gg^*\). Equivalently, \(\Gg\) satisfies (i) and we are done.
\end{proof}

Comparing the above result with its counterpart for GEC's---\THM{skton-grp}---we
see two things:
\begin{itemize}
\item as all Boolean groups are fully classified, unlike, e.g., all finite
 groups, the classification of AB-HO dGEC's is apparently much harder than that
 of AB-HO GEC's;
\item the property formulated in part (a) of \THM{dskton-grp} in the realm of
 AB-HO GEC's is equivalent to a unique way of extending partial morphisms
 defined on non-empty sub-GEC's to global automorphisms (see item (C) in
 \THM{skton-grp})---and this equivalence is false for AB-HO dGEC's (see
 \EXM{unique-non} below); however, the former property is stronger than
 the latter. 
\end{itemize}

\begin{figure}
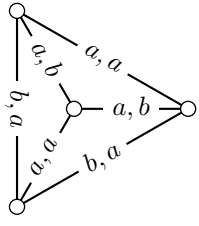

\DIAG{0.3}{\Vertices{circle}{Z,W,Y}\VERT[,Lpos=30,Ldist=-1mm]\WE(Z){X}%
\EDGE[labelstyle=sloped,]{a,a}{Y}{X}\EDGE[labelstyle=sloped,]{a,b}{X}{Z}%
\EDGE[labelstyle=sloped,]{a,b}{W}{X}\EDGE[labelstyle=sloped,]{b,a}{Y}{Z}%
\EDGE[labelstyle=sloped,]{b,a}{W}{Y}\EDGE[labelstyle=sloped,]{a,a}{W}{Z}}
\caption{A uniquely homogeneous AB-HO dGEC without the property of different
colors of outcoming edges}\label{fig:uniq-non}
\end{figure}

\begin{exm}{unique-non}
Figure~\ref{fig:uniq-non} shows a \emph{uniquely homogeneous} AB-HO dGEC (i.e.,
an AB-HO dGEC with the property that each partial morphism defined on
a non-empty sub-dGEC extends to a unique global automorphism), but some two
edges outcoming from a common vertex do have the same color. The color of
an outcoming edge from a given vertex is represented by a letter \(a\) or \(b\)
that is written on this edge on a position closer to that vertex. A verification
of the aformentioned property is left to the reader.
\end{exm}

The above example makes it interesting to find a characterisation of uniquely
homogeneous AB-HO dGEC's. To this end, we introduce

\begin{dfn}{reverse}
A dGEC \(\Gg = (V,\kappa)\) is said to be \emph{reversive} if there exists
a function \(u\dd \COL{\Gg} \to \COL{\Gg}\) such that for any two distinct
vertices \(x\) and \(y\) from \(V\):
\begin{equation}\label{eqn:rev}
\kappa(x,y) = u(\kappa(x,y)).
\end{equation}
In other words, a dGEC is reversive if the color of each edge is determined by
the color of the reverse edge.\par
The \emph{reversification} of \(\Gg\) is a dGEC \(\Gg_{rev} = (V,\kappa_{rev})\)
with the same pseudocolor as \(\Gg\) and where \(\kappa_{rev}(x,y) \df
(\kappa(x,y),\kappa(y,x))\) for distinct \(x\) and \(y\). Note that
\(\Gg_{rev}\) is reversive, since \eqref{eqn:rev} is satisfied for
\(\kappa_{rev}\) and \(u(p,q) \df (q,p)\).
\end{dfn}

Reversification has a very important property, formulated in a proposition
below. Since its proof is straightforward, we omit it.

\begin{pro}{rev}
For a dGEC \(\Gg = (V,\kappa)\) and a function \(\phi\dd A \to V\) with \(A
\subset V\) \tfcae
\begin{enumerate}[\upshape(i)]
\item \(\phi\) is a partial morphism of \(\Gg\);
\item \(\phi\) is a partial morphism of \(\Gg_{rev}\).
\end{enumerate}
In particular, \(\Gg\) is AB-HO (resp. uniquely homogeneous and AB-HO) iff so is
\(\Gg_{rev}\).
\end{pro}

The above result reduces, in a sense, study of AB-HO dGEC's to AB-HO reversive
dGEC's. With the aid of this new notion, we can now characterise uniquely
homogeneous AB-HO dGEC's as follows:

\begin{pro}{uniq-homo}
A dGEC is both uniquely homogeneous and AB-HO iff its reversification is
isomorphic to a group with left dGEC-structure.
\end{pro}
\begin{proof}
The ``if'' part follows from \PRO{dgroup} combined with \PRO{rev}. To see
the ``only if'' part, consider a uniquely homogeneous AB-HO dGEC \(\Gg =
(V,\kappa)\) and note that then \(\Gg_{rev}\) has these two properties as well
(by \PRO{rev}). Moreover, since \(\Gg\) has both these properties, we infer
that for any three distinct vertices \(x,y,z \in V\) one has:
\[\kappa(x,y) \neq \kappa(x,z) \qquad \UP{or} \qquad \kappa(y,x) \neq
\kappa(z,x)\]
(because otherwise we will be able to find an automorphism of \(\Gg\) that fixes
\(x\) and sends \(y\) onto \(z\) and hence \(\Gg\) would not be uniquely
homogeneous). But the above condition means that \(\Gg_{rev}\) satisfies
condition (a) of \THM{dskton-grp}, from which the asertion follows.
\end{proof}

The above result enables us to give a more intrinsic characterisation of
uniquely homogeneous AB-HO dGEC's:

\begin{thm}{uniq-homo}
For a non-empty dGEC \(\Gg = (V,\kappa)\) \tfcae
\begin{enumerate}[\upshape(i)]
\item \(\Gg\) is a uniquely homogeneous AB-HO dGEC;
\item there exist a binary operation \(*\) on \(V\) and a function \(\mu\dd V
 \to Z\) (where \(Z\) is some set) such that:
 \begin{itemize}
 \item \((V,*)\) is a group;
 \item the function \(V \ni v \mapsto (\mu(v),\mu(v^{-1})) \in Z \times Z\) is
  one-to-one;
 \item \(\kappa(u,v) = \mu(u^{-1}*v)\) for any \(u, v \in V\).
 \end{itemize}
\end{enumerate}
Moreover,
\begin{enumerate}[\upshape(gr1)]
\item if \UP{(ii)} holds, then \(\AUT{\Gg}\) consists precisely of all left
 translations of the group \((V,*)\); in particular, \((V,*)\) is isomorphic to
 \(\AUT{\Gg}\);
\item if \UP{(i)} holds, then for any vertex \(u \in V\) there exists a unique
 pair \((*,\mu)\) such that \UP{(ii)} holds and \(u\) is the neutral element of
 \((V,*)\).
\end{enumerate}
\end{thm}
\begin{proof}
First assume (i) holds. By \PRO{uniq-homo}, there exists a group \((G,\cdot)\)
and an isomorphism \(\phi\dd V \to G\) of \(\Gg_{rev}\) onto \((G,L_G)\). Recall
that the palette of \(\Gg_{rev}\) consists of (certain) pairs of colors from
the palette of \(\Gg\). We denote by \(\pi\dd \COL{\Gg} \times \COL{\Gg} \to
\COL{\Gg}\) the projection onto the first coordinate; that is, \(\pi(c,d) \df
c\). Additionally, we set \(\pi(e) \df e\) where \(e\) is the pseudocolor of
both \(\Gg\) and \(\Gg_{rev}\). Further, since \(\phi\) is a bijection, we may
and do transfer the group structure of \(G\) into \(V\) to obtain a binary
operation \(*\) on \(V\) such that \((V,*)\) is a group and \(\phi\dd (V,*) \to
(G,\cdot)\) is a group isomorphism.\par
Note that \(G\) is contained (as a set) in the domain of \(\pi\) (because
\(\COL{G,L_G} = G \setminus \{e_G\}\) coincides with the palette of
\(\Gg_{rev}\) and the pseudocolor \(e_G\) of \(G\) is the same as that of
\(\Gg_{rev}\)). Thus, it makes sense the formula \(\mu \df \pi \circ \phi\dd V
\to Z\) (where \(Z \df \COL{\Gg} \cup \{e\}\)). Finally, observe that:
\begin{itemize}
\item for any \(x, y \in V\), 
 \begin{align*}
 \kappa(x,y) &= \pi(\kappa_{rev}(x,y)) = \pi(L_G(\phi(x),\phi(y))) \\&=
 \pi(\phi(x)^{-1} \phi(y)) = \pi(\phi(x^{-1} * y)) = \mu(x^{-1}*y);
 \end{align*}
\item for a fixed \(a \in V\), we already know that all outcoming edges from
 \(a\) have different colors in \(\Gg_{rev}\), which implies that the function
 \(V \ni v \mapsto \kappa_{rev}(a,v) \in \COL{\Gg_{rev}}\) is one-to-one, but
 \(\kappa_{rev}(a,v) = (\kappa(a,v),\kappa(v,a)) = (\mu(a^{-1}*v),
 \mu(v^{-1}*a)),\)
\end{itemize}
and hence it is sufficient to substitute the neutral element of \((V,*)\) for
the above \(a\) to get the whole assertion of (ii).\par
We pass to the reverse implication. Assume (ii) is fulfilled. Let \(\phi\dd A
\to V\) be a partial morphism of \(\Gg\) defined on a non-empty set \(A \subset
V\). We fix \(a \in A\) and set \(b \df \phi(a)\). Then, for all \(x \in A\):
\begin{align*}
\mu(a^{-1}*x) &= \kappa(a,x) = \kappa(b,\phi(x)) = \mu(b^{-1}*\phi(x)),\\
\mu((a^{-1}*x)^{-1}) &= \mu(x^{-1}*a) = \kappa(x,a) = \kappa(\phi(x),b) \\&=
\mu(\phi(x)^{-1}*b) = \mu((b^{-1}*\phi(x))^{-1}).
\end{align*}
So, it follows from the property of \(\mu\) postulated in (ii) that \(a^{-1}*x =
b^{-1}*\phi(x)\), which yields \(\phi(x) = c*x\) for all \(x \in A\) and \(c \df
b*a^{-1}\). On the other hand, a straightforward calculation shows that all left
translations of \((V,*)\) are automorphisms of \(\Gg\), and therefore
\(\AUT{\Gg}\) consists precisely of all left translations and the above \(\phi\)
extends to a unique automorphism of \(\Gg\), which finishes the proof of this
implication and, incidentally, of (gr1).\par
The uniqueness in (gr2) follows from (gr1): if \((V,*)\) is a group such that
(for some \(\mu\)) condition (ii) is fulfilled and \(u\) is the neutral element
of this group, then (thanks to (gr1)) the function \(\AUT{\Gg} \ni \phi \mapsto
\phi(u) \in V\) is a group isomorphism that allows us to determine the binary
operation \(*\) on \(V\). And knowing \(*\), we know \(\mu\) as \(\mu(v) =
\kappa(u,v)\) (by the last property in (ii)). Finally, to prove the existence in
(gr2), we use previous parts of the proof---we already know that there exist
a certain group structure \((V,\cdot)\) (on \(V\)) and a function \(\mu_0\) such
that (ii) holds. Then we define \(*\) and \(\mu\) for fixed \(u\) as follows:
\(x*y \df x \cdot u^{-1} \cdot y\) and \(\mu(x) \df \mu_0(u^{-1} \cdot x)\). We
leave it as an exercise that this pair has all properties postulated in (gr2).
\end{proof}

We end the section with the following

\begin{exm}{smallest}
We now give a simple example of an AB-HO dGEC that is not a GEC. To this end,
let \(V \df \{0,1,2\}\) be the set of vertices, \(0\) stand for the pseudocolor
and let \(-1\) and \(1\) serve as colors. We define the edge-coloring \(\kappa\)
as follows: \(\kappa(0,1) = \kappa(1,2) = \kappa(2,0) \df 1\) and \(\kappa(1,0)
= \kappa(2,1) = \kappa(0,2) \df -1\). It may easily be verified that:
\begin{itemize}
\item \(\Gg_{\min} \df (V,\kappa)\) is an AB-HO dGEC (that is not a GEC);
\item each AB-HO dGEC that is not a GEC is either structurally equivalent to
 \(\Gg_{\min}\) or else has size greater than three;
\item \(\Gg_{\min}\) is structurally equivalent to the group \(\ZZZ/3\ZZZ\)
 considered as a dGEC---in particular, \(\Gg_{\min}\) is Abelian-like.
\end{itemize}
We leave the details as an exercise.
\end{exm}

\section{Remarks on structural equivalence}\label{sec:ges}

When studying structures of AB-HO GEC's (or dGEC's), the palette does not really
matter---that is, specific colors of edges are irrelevant; the only thing that
matters is which edges share the same color. From this perspective isomorphisms
between dGEC's lose significance in favor of structural isotopies. In this
section we discuss most important aspects of AB-HO dGEC's when we consider them
as structures without colors. As such, their structure acquires a more inward
nature. What exactly we have in mind is explained in the following

\begin{dfn}{similar}
By an \emph{edge similarity} on a set \(V\) (which serves, as usual, as a set of
vertices) we mean any equivalence relation \(\sim\) on the set \(V \times V\)
(that is, between pairs from \(V \times V\)) with the following property:
\[\forall x, y, z \in V\dd\qquad (x,y) \sim (z,z) \iff x = y.\]
An edge similarity \(\sim\) on \(V\) is said to be \emph{reflective} iff:
\[\forall x, y \in V\dd\qquad (x,y) \sim (y,x).\]
A \emph{dGES} (a \emph{directed graph with edge similarity}) is a set (of
vertices) equipped with an edge similarity. If this relation is reflective, we
speak about a \emph{GES} (a \emph{graph with edge similarity}).
\end{dfn}

\begin{exm}{dGEC=dGES}
Each dGEC (resp. GEC) \((V,\kappa)\) possesses a canonical structure of a dGES
(resp. a GES) consisting of the edge similarity \(\sim_{\kappa}\) given by:
\[(x,y) \sim_{\kappa} (z,w) \iff \kappa(x,y) = \kappa(z,w).\]
Conversely, each dGES (resp. GES) \((V,\sim)\) possesses a canonical structure
of a dGEC (resp. a GEC): it is enough to declare the diagonal of \(V\) (which
coincides with a single equivalence class w.r.t. \(\sim\)) as the pseudocolor,
the quotient set \(Y \df (V \times V) / \sim\) reduced by the pseudocolor as
the palette of colors and the quotient map \(V \times V \ni (x,y) \mapsto
[(x,y)]_{\sim} \in Y\) as an edge-coloring.
\end{exm}

In order to speak about absolute homogeneity, we need to introduce partial
morphisms of dGES'es, which we now do.

\begin{dfn}{part-morph-dGES}
Let \(\Ee \df (V,\sim)\) be a dGES. A \emph{partial morphism} of \(\Ee\) is any
function \(f\dd A \to V\) defined on a subset \(A\) of \(V\) such that:
\[\forall x, y \in A\dd\qquad (f(x),f(y)) \sim (x,y).\]
(Notice that the above condition forces \(f\) to be one-to-one.) If, in
addition, \(A = V\), the above \(f\) is said to be a \emph{morphism} of \(\Ee\).
A bijective morphism is called an \emph{automorphism}.\par
A dGES is said to be \emph{absolutely homogeneous} (briefly: \emph{AB-HO}) if
every its partial morphism extends to an automorphism. Homogeneity and other its
variants (like two-point homogeneity or ultrahomogeneity) one defines likewise.
\end{dfn}

We wish to underline that in this approach there is no way of defining morphisms
between two different dGES'es (at least in a way such that this new definition
will be compatible with the old one introduced for GEC's and dGEC's). However,
we are able to speak about isotopies (in this approach we skip the adjective
`structural' as we do not need to make any changes in the structures of dGES'es,
in contrast to the context of dGEC's where we allow repainting):

\begin{dfn}{isotopy}
An \emph{isotopy} between dGES'es \((V,\sim_V)\) and \((W,\sim_W)\) is any
bijection \(f\dd V \to W\) such that:
\[\forall x,y,z,w \in V\dd\qquad (x,y) \sim_V (z,w) \iff (f(x),f(y)) \sim_W
(f(z),f(w)).\]
Two dGES'es are said to be \emph{isotopic} if there exists an isotopy between
them.
\end{dfn}

Finally, with no difficulty we can introduce products of (totally arbitrary
collections of non-empty) dGES'es:

\begin{dfn}{prod-dGES}
Let \(\eeE = \{\Ee_s = (V_s,\sim_s)\}_{s \in S}\) be an arbitrary collection of
non-empty dGES'es, \(\aA = (a_s)_{s \in S}\) be an arbitrary element of
\(\prod_{s \in S} V_s\) and let \(\ORD\) be a total order on \(S\). We define
a set \(W\) in the very same way as in \DEF{product}, that is, \(W\) consists of
all elements \(\xX = (x_s)_{s \in S} \in \prod_{s \in S} V_s\) such that the set
\(\DIF{\xX}{\aA}\) is well-ordered w.r.t. \(\ORD\). Finally, the \emph{product
of \(\eeE\) w.r.t. \(\ORD\)}, to be denoted by \(\PROd{s \in S}{\ORD}{\aA}
\Ee_s\), is a dGES \((W,\sim_W)\) where \(\sim_W\) is defined as follows:
\begin{multline*}
(\xX,\yY) \sim_W (\zZ,\wW) \iff (\xX = \yY\ \land\ \zZ = \wW)\ \lor\\
(\xX \neq \yY\ \land\ \zZ \neq \wW\ \land\ \ell \df \LAB{\xX}{\yY} =
\LAB{\zZ}{\wW}\ \land\ (x_{\ell},y_{\ell}) \sim_{\ell} (z_{\ell},w_{\ell}))
\end{multline*}
where \(x_{\ell},y_{\ell},z_{\ell}\) and \(w_{\ell}\) are the \(\ell\)-th
coordinates of \(\xX,\yY,\zZ\) and \(\wW\), respectively.
\end{dfn}

Similarly as in case of dGEC's, one easily proves that if the collection
\(\{\Ee_s\}_{s \in S}\) consists only of homogeneous dGES'es, then all
the products of the form \(\PROd{s \in S}{\ORD}{\aA}\) where \(\ORD\) is a fixed
total order on \(S\) (and \(\aA\) varies) are pairwise isotopic. (Even more, if
\(\ORD\) is a well-order on \(S\), all they coincide, even for collections of
non-homogeneous dGES'es.) Thus, in such cases we may write \(\PROD{s \in S}
\Ee_s\) to denote an isotopic type of this product. If, in addition to
the above, all dGES'es coincide, say \(\Ee_s = \Ee\), instead of
\(\PROD{s \in S} \Ee_s\) we write \(\Ee^{(S,\ORD)}\).

Passing to canonical dGEC's structures of dGES's (introduced in
\EXM{dGEC=dGES}), one concludes from \THM{prod-AB-HO} the following result,
whose proof is omitted.

\begin{thm}{prod-dGES}
The product of AB-HO dGES'es is AB-HO as well.
\end{thm}

We leave as an exercise to define \ITP\ for dGES'es and formulate classification
of all dGES'es with \ITP\ (based on \THM{classification}).\par
We want to pay special attention to compact topological dGES'es which we call
\emph{cdGES'es} and define as follows:

\begin{dfn}{cdGES}
A \emph{cdGES} (a \emph{compact} dGES) is a compact topological Hausdorff space
\(X\) equipped with an edge similarity that is a closed subset of \(X \times
X\). A \emph{cGES} is a cdGES that is a GES (that is, whose edge similarity is
reflective).
\end{dfn}

The above requirement that the edge similarity must be closed in the Cartesian
square of the space is typical in a topological context, especially when working
with compact Hausdorff spaces as in such spaces the aforementioned closedness
guarantees (and is equivalent to) that the quotient space (that is, \((X \times
X)/\sim\) where \(\sim\) is the edge similarity) is a (compact) Hausdorff space
as well.\par
Below we present a little bit surprising result on cdGES'es. It may be
considered as an abstract generalisation of well-known facts about partial
isometries of compact metric spaces.

\begin{thm}{cdGES}
Let \(\Xx = (X,\sim)\) be a cdGES.
\begin{enumerate}[\upshape(a)]
\item Each partial morphism of \(X\) is a continuous function and extends to
 a partial morphism defined on a closed set.
\item For any closed set \(Z \subset X\) the set \(\OPN{Mor}(Z,X)\) of all
 partial morphisms \(f\dd Z \to X\) of \(\Xx\) is compact in the compact-open
 topology of the space \(C(Z,X)\) of all continuous functions from \(Z\) into
 \(X\).
\item Each morphism of \(\Xx\) is an automorphism.
\item \(\AUT{\Xx}\) is a compact topological group in the compact-open topology.
\item \(\Xx\) is AB-HO iff every partial morphism defined on a finite subset
 \(A\) of \(X\) extends to a partial morphism defined on \(A \cup \{a\}\) for
 each \(a\) from a fixed dense subset \(D\) of \(X\).
\end{enumerate}
\end{thm}
\begin{proof}
Before passing to specific parts of the theorem, we make a preliminary
preparations. Denote by \(R\) the set \(\{(x,y,z,w) \in X^4\dd\ (x,z) \sim
(y,w)\}\) where \(X^4 \df X \times X \times X \times X\). It follows from
the definition of a cdGES that \(R\) is a closed set in \(X^4\).\par
(a): Let \(f\dd A \to X\) (where \(A \subset X\)) be a partial morphism. Since
\(f\) is a partial morphism, we infer that its graph \(\Gamma \df \{(a,f(a))\dd\
a \in A\}\) satisfies: \(\Gamma \times \Gamma \subset R\). Consequently, also
\begin{equation}\label{eqn:aux10}
\bar{\Gamma} \times \bar{\Gamma} \subset R
\end{equation}
where \(\bar{\Gamma}\) is the closure of \(\Gamma\) in \(X \times X\). Inclusion
\eqref{eqn:aux10} implies that \(\bar{\Gamma}\) coincides with the graph of some
function. To convince onelself about it, take any two pairs \((a,b), (c,d) \in
\bar{\Gamma}\) such that \(c = a\). It then follows from \eqref{eqn:aux10} that
\((a,c) \sim (b,d)\) and hence \(d = b\), which is equivalent to the statement
that \(\bar{\Gamma}\) is the graph of some \(F\dd B \to X\). Finally, since
\(F\) has closed graph and \(X\) is compact, it follows that \(B\) is closed and
\(F\) is continuous, which finishes the proof of (a).\par
(b): We already know (from (a)) that \(K \df \OPN{Mor}(Z,X)\) consists of
continuous functions. Note that \(K\) is closed in the topology of pointwise
convergence in \(C(Z,X)\), thanks to the closedness of \(R\) in \(X^4\). So,
since both \(Z\) and \(X\) are compact, we can make use of an abstract version
of the Arzela-Ascoli type theorem (consult, e.g., \cite[Theorem~3.4.20]{eng})
which states that \(K\) (being closed) is compact (in the compact-open topology)
iff it is an \emph{evenly continuous} family of functions. In other words, we
only need to show that:
\begin{itemize}
\item[\((*)\)] For any \(x \in Z\), \(y \in X\) and an open neighbourhood \(V\)
 of \(y\) in \(X\) there are open neighbourhoods \(U \subset Z\) and \(W \subset
 X\) of \(x\) and \(y\), respectively, with the following property: if \(u \in
 K\) and \(u(x) \in W\), then \(u(U) \subset V\).
\end{itemize}
To prove \((*)\), we argue by a contradiction: we fix \(x\), \(y\) and \(V\) as
specified in \((*)\), choose any neighbourhood \(W\) of \(y\) whose closure is
contained in \(V\), and suppose that for any open (in \(Z\)) neighbourhood \(U
\subset Z\) of \(x\) there exist a function \(f_U \in K\) and a point \(x_U \in
U\) such that for \(\sigma \df U\):
\begin{equation}\label{eqn:aux11}
f_{\sigma}(x) \in W \qquad \UP{and} \qquad f_{\sigma}(x_{\sigma}) \notin V.
\end{equation}
Since \(K\) is compact in the pointwise convergence topology (by the Tychonoff
theorem on products of compact spaces), we may and do assume (after passing to
a subnet) that the functions \(f_{\sigma}\ (\sigma\in\Sigma)\) converge
pointwise on \(Z\) to a function \(g \in K\) (and still satisfy
\eqref{eqn:aux11}) and the points \(f_{\sigma}(x_{\sigma})\ (\sigma\in\Sigma)\)
converge to \(z \in X\). Then it follows that \(g(x) = \lim_{\sigma\in\Sigma}
f_{\sigma}(x)\) belongs to the closure of \(W\) and hence \(g(x) \in V\). On
the other hand, we infer from \eqref{eqn:aux11} that \(z \notin V\). Finally,
observe that \(\lim_{\sigma\in\Sigma} x_{\sigma} = x\) and \((x_{\sigma},x) \sim
(f_{\sigma}(x_{\sigma}),f_{\sigma}(x)) \stackrel[\sigma\in\Sigma]{}{\to}
(z,g(x))\). So, it follows from the closedness of the relation \(\sim\) that
\((x,x) \sim (z,g(x))\) as well. But the latter relation implies that \(g(x) =
z \notin V\), which contradicts the property established earlier.\par
Since we need (d) in the proof of (c), now we will briefly show (d) (which
simply follows from (b)). To this end, just observe that the set
\[P \df \{(u,v) \in \OPN{Mor}(X,X) \times \OPN{Mor}(X,X)\dd\ u \circ v = v \circ
u = \OPN{id}_X\}\] (where \(\OPN{id}_X\) is the identity map on \(X\)) is
a closed set in \(\OPN{Mor}(X,X)\) (as the operation of composing functions is
continuous on \(C(X,X) \times C(X,X)\) for compact \(X\)) and \(\AUT{\Xx}\)
coincides with the image of \(P\) by a projection onto the first coordinate. So,
\(\AUT{\Xx}\) is compact, thanks to (b). And as a subgroup of the homeomorphism
group of a compact Hausdorff space, it is a topological group.\par
Now we pass to (c). Let \(f\dd X \to X\) be a morphism of \(\Xx\). It follows
from (b) that the sequence \((f^n)_{n=1}^{\infty}\) of all iterated powers of
\(f\) has a convergent subnet to some morphism \(g\dd X \to X\). For simplicity,
we set \(K \df \bigcap_{n=1}^{\infty} f^n(X)\). Observe that \(f(K) = K\) and
\begin{equation}\label{eqn:aux12}
g(X) \subset K.
\end{equation}
In particular, \(f\restriction{K}\) is an automorphism of the sub-cdGES \(\Kk
\df (K,\sim)\) of \(\Xx\). So, it follows from (d) that \(g\restriction{K}\) is
an automorphism of \(\Kk\) and, consequently, \(g(K) = K\). Since \(g\) is
one-to-one, the last equality, combined with \eqref{eqn:aux12} yields that \(K =
X\), which finishes the proof of (c).\par
(e): We only need to check the sufficiency of the condition formulated in (b).
To this end, we assume that \(X\) has the property of extending partial
morphisms on finite sets by one point (as specified in (b)) and fix a partial
morphism \(f\dd A \to X\) (with \(A \subset X\)). It follows from our assumption
that for any two finite sets \(S \subset A\) and \(F \subset D\) the set
\(E(S,F)\) consisting of all (possibly discontinuous) functions \(g\dd X \to X\)
such that:
\begin{itemize}
\item \(g\restriction{S} = f\restriction{S}\);
\item \(g\restriction{S \cup F}\dd S \cup F \to X\) is a morphism
\end{itemize}
is non-empty. What is more, \(E(S,F)\) is closed in the pointwise convergence
topology of \(X\) (because \(R\) is closed in \(X^4\)). So, we infer from
the Tychonoff theorem (on products of compact spaces) that \(E(S,F)\) is
compact. Observe that all such sets form a centered family, as \(\bigcap_{k=1}^n
E(S_k,F_k) \supset E(\bigcup_{k=1}^n S_k,\bigcup_{k=1}^n F_k) \neq \varempty\)
for any finite sets \(S_1,\ldots,S_n \subset A\) and \(F_1,\ldots,F_n \subset
D\) (and each \(n > 0\)). So, the compactness of these sets implies that their
intersection is non-empty. Now any function \(g_0\) that belongs to that
intersection is a densely defined partial morphism that extends \(f\). We infer
from (a) that \(g_0\) extends to a (unique) morphism \(g\) and from (c) that
\(g\) is an automorphism, and we are done.
\end{proof}

Part (e) of the above result contains, in particular, a probably well-known
property of compact metric spaces which says that in such spaces
ultrahomogeneity is equivalent to absolute homogeneity.

\begin{exm}{comp-grp}
Let \(G\) be a compact topological Hausdorff group. Then \(G\) becomes an AB-HO
cdGES when considered with its canonical dGEC structure induced by (e.g.) its
left dGEC-structure, that is, when considered with the relation \(\sim\) given
by: \[(a,b) \sim (c,d) \stackrel{\UP{def}}{\iff} a^{-1}b = c^{-1}d.\]
It follows from \THM{skton-grp} that the above \(G\) is a cGES iff it is
a Boolean group. It is worth noticing here that all compact Boolean groups are
fully classified: each such a group is topologically isomorphic to
the multiplicative group \(\{-1,1\}^{\alpha}\) (equipped with the product
topology) where \(\alpha\) is an arbitrary cardinal and is uniquely determined
by the group under consideration.
\end{exm}

The above example shows, among other things, that the class of cdGES contains
infinite-dimensional metrizable spaces which are simultaneously arcwise
connected, locally arcwise connencted and simply connected (take, e.g.,
\(SU(2)^{\omega}\)). We are also able to give similar examples (but not simply
connected) of AB-HO cGES'es---indeed, each compact Abelian group becomes a cGES
w.r.t. certain quite natural structure, as shown by

\begin{pro}{CA}
Let \((K,+)\) be a compact Abelian group equipped with the following edge
similarity:
\[(x,y) \sim_{Ab} (z,w) \iff x-y = z-w\ \lor\ x-y = w-z.\]
Then \((K,\sim_{Ab})\) is an AB-HO symmetric compact GES.
\end{pro}
\begin{proof}
Instead of giving a detailed proof, we restrict here only to giving the reader
the main hint: show that each partial morphism \(\phi\dd A \to K\) has one of
the two possible forms: \(\phi(x) = g+x\ (x \in A)\) or \(\phi(x) = g-x\ (x \in
A)\) where \(g\) is an (arbitrarily) fixed element of \(K\).
\end{proof}

\begin{exm}{never}
In order to clarify any uncertainties, let us show that if two non-degenerate
cdGES'es \(\Xx = (X,\sim_X)\) and \(\Zz = (Z,\sim_Z)\) are such that their
product \(\Ww = (W,\sim_W) \df \Xx \times^* \Zz\) is a cdGES as well (in
the product topology), then the space \(X\) is finite (that is, otherwise
the relation \(\sim_W\) will not be closed in \(W \times W\)). To this end,
suppose \(X\) is infinite and fix a net \(\{x_{\sigma}\}_{\sigma\in\Sigma}\) in
\(X\) that converges to some \(x \in X\) and consists of points distinct from
\(x\). Choose arbitrarily two distinct points \(y\) and \(z\) of \(Z\) and
observe that for points \(w_{\sigma} \df (x_{\sigma},y)
\stackrel[\sigma\in\Sigma]{}{\to} w \df (x,y)\) and \(v \df (x,z)\) from \(W\)
one obtains (from the very definition of \(\sim_W\)) that \((w_{\sigma},w)
\sim_W (w_{\sigma},v)\) and simultaneously \((w,w) \not\sim_W (w,v)\) and hence
\(\sim_W\) is not closed in \(W^4\).
\end{exm}

\subsection{From dGES'es to \textit{abstrans}}\label{sec:abstr}
Assume \(\Xx = (X,\sim)\) is a two-point homogeneous dGES and set \(G \df
\AUT{\Xx}\). Then, for any \(x,y,z,w \in X\):
\begin{equation}\label{eqn:sim}
(x,y) \sim (z,w) \iff \exists \phi \in G\dd\ (z,w) = (\phi(x),\phi(y)).
\end{equation}
So, the dGES structure is reproducible from the automorphism group \(G\).\par
Conversely, having any group \(G\) of permutations of a set \(X\) that acts
transitively on \(X\), we can define \(\sim\) by \eqref{eqn:sim} to obtain
a dGES \(\Xx \df (X,\sim)\) that is automatically two-point homogeneous. Indeed,
in that case \(G\) is a subgroup of \(\AUT{\Xx}\). So, when studying two-point
homogeneous dGES'es, instead of pairs of the form \((X,\sim)\) one may start
from \((X,G)\) where \(G\) is a group of permutations of \(X\) acting
transitively on \(X\) (ultimately, we want \(G\) to coincide with the full
automorphism group of the dGES it induces by \eqref{eqn:sim}). With such
an approach, absolute homogeneity translates as follows:
\begin{itemize}\itshape
\item[\UP{(at)}] If \(u\dd A \to X\) (where \(A \subset X\)) is a function such
 that
 \begin{equation}\label{eqn:at}
 \forall a, b \in A\ \exists \phi \in G\dd\ (u(a),u(b)) = (\phi(a),\phi(b)),
 \end{equation}
 then there exists \(v \in G\) for which \(u = v\restriction{A}\).
\end{itemize}
The above condition is actually a property of \(G\). By an analogy to absolute
homogeneity, we call the above condition (at) \emph{absolute transitivity} of
\(G\) (on \(X\)) and say that \(G\) acts \emph{absolutely transitively} on \(X\)
if (at) holds (but firsly we require the transitivity of the action of
\(G\)).\par
Going further, as \(G\) acts transitively on \(X\), \(X\) may naturally be
identified with the set of all left cosets of the isotropy group \(H \df
\OPN{stab}_G(a)\) where \(a\) is an arbitrarily chosen point from \(a\). Then
\(H\) contains no non-trivial normal subgroups of \(G\) and under such
an identification (of points from \(X\) with left cosets of \(H\)) (at) takes
the form formulated in the following

\begin{lem}{at}
Under the above settings, \UP{(at)} is equivalent to:
\begin{itemize}\itshape
\item[\UP{(AT)}] If \(f\dd Z \to G\) (where \(Z \subset G\)) is a function such
 that
 \begin{equation}\label{eqn:AT}
 \forall x,y \in Z\ \exists g \in G\dd\ x^{-1}g^{-1}f(x), y^{-1}g^{-1}f(y) \in
 H,
 \end{equation}
 then there exists \(g \in G\) for which \(z^{-1}g^{-1} f(z) \in H\) for all \(z
 \in Z\).
\end{itemize}
\end{lem}
\begin{proof}
(at)\(\implies\)(AT): Let \(f\dd Z \to G\) satisfy \eqref{eqn:AT}. We set \(A
\df \{zH\dd\ z \in Z\} \subset X\). Observe that:
\[\forall x, y \in Z\dd\qquad (xH = yH \implies f(x)H = f(y)H),\]
which simply follows from \eqref{eqn:AT}. The above property enables us to
correctly define a function \(u\dd A \to X\) by \(u(zH) \df f(z)H\). Moreover,
\(u\) satisfies \eqref{eqn:at} (again thanks to \eqref{eqn:AT}). So,
an application of (at) gives us \(g \in G\) such that \(u(zH) = gzH\) for any
\(z \in Z\), which is equivalent to the conlusion of (AT).\par
(AT)\(\implies\)(at): Now take any \(u\dd A \to X\) (with \(A \subset X\)) that
satisfies \eqref{eqn:at}. (Recall that \(X\) consists of all left cosets of
\(H\).) Let \(Z\) be the union of all cosets that are elements of \(A\)
(actually we could take as \(Z\) any set that is contained in this union and
intersects each coset from \(A\)). For any \(z \in Z\) we choose an arbitrary
element of \(u(zH)\) and denote it by \(f(z)\). In this way we obtain a function
\(f\dd Z \to G\). It is easy to check that \(f\) satisfies \eqref{eqn:AT}, since
\(u\) fulfills \eqref{eqn:at}. So, we infer from (AT) that there is \(g \in G\)
such that \(f(z)H = gzH\) for any \(z \in Z\). Since \(f(z)H = u(zH)\),
the former equation yields (at).
\end{proof}

Observe that (AT) is purely algebraic (that is, it makes sense also for abstract
groups). In this way we have arrived at the following

\begin{dfn}{abstran}
An \emph{abstran} is a pair \((G,H)\) such that:
\begin{itemize}
\item \(G\) is a group; and
\item \(H\) is a subgroup of \(G\) containing no non-trivial normal subgroup of
 \(G\); and
\item (AT) is fulfilled (see \LEM{at}).
\end{itemize}
(This notion is named after `absolute transitivity'.)
\end{dfn}

Abstrans are outcomes of AB-HO dGES'es. (Actually, both these object are in
a one-to-one correspondence.) So, a natural question arises of which abstrans
correspond to AB-HO GES'es. Similarly as presented in the proof of \LEM{at}, we
can easily check that these are those abstrans \((G,H)\) that satisfy:
\[\forall a, b \in G\ \exists g \in G\dd\qquad b^{-1}ga, a^{-1}gb \in H.\]
Equivalently:
\[\forall a, b \in G\dd\qquad bHa^{-1} \cap aHb^{-1} \neq \varempty.\]
Substituting in the above formula \(g\) for \(a^{-1}b\), we obtain another
equivalent condition:
\begin{equation*}\label{eqn:sym}\tag{sym}
\forall g \in G\dd\qquad gHg \cap H \neq \varempty.
\end{equation*}
Abstrans satisfying \eqref{eqn:sym} we call \emph{symmetric}.\par
Now if we start the whole story of this subsection from an AB-HO cdGES \(\Xx\),
then \(G = \AUT{\Xx}\) becomes a compact topological group (by \THM{cdGES}) and
\(H\), as the isotropy group, becomes a closed subgroup of \(G\). So, it is
quite reasonable to call such abstrans \emph{compact}. More generally,
an abstran \((G,H)\) is said to be \emph{compact} if both \(G\) and \(H\) are
compact groups.\par
A (weaker) counterpart of condition (e) of \THM{cdGES} for pairs of groups is
formulated below. Its proof is left as an easy exercise.

\begin{pro}{fin-abstran}
For a compact group \(G\) and its closed subgroup \(H\) containing no
non-trivial normal subgroup of \(G\) \tfcae
\begin{enumerate}[\upshape(i)]
\item \((G,H)\) is an abstran;
\item condition \UP{(AT)} holds whenever \(Z\) is a finite set.
\end{enumerate}
\end{pro}

As Euclidean spheres are AB-HO, we conclude that for all \(n > 1\) the pairs
\((O_n,O_{n-1})\) (where each matrix from \(O_{n-1}\) is supplemented by adding
one row and one column whose all entries are zero apart from the diagonal cell
where the entry is 1) are compact symmetric abstrans corresponding to connected
spaces. As we have seen (in \PRO{CA}), also pairs of the form \((K \rtimes
\FFF_2,\{e_K\} \times \FFF_2)\) where \(K\) is a connected compact Abelian group
and \(\FFF_2\) is the two-point group, are compact symmetric abstrans that
correspond to connected spaces. We know no other examples of such abstrans. To
be correctly understood, we explain that a compact abstran \((G,H)\) corresponds
to a connected space if the homogeneous quotient space \(G/H\) (that is,
the AB-HO cdGES the abstran corresponds to) is connected. It is well-know that
in this situation (that is, if both the groups are compact) this condition is
equivalent to \(G = G_o H\) where \(G_o\) is the connected component of \(G\) at
the neutral element.

\section{Absolute homogeneity: categorical approach}\label{sec:cat}

It is a natural temptation to try to generalise the results that have been
already established in a narrow context to a wider variety. In this section we
fulfill this desire and present a categorical approach to absolute homogeneity.
To this end, we begin with fixing the notation. For any category \(\ccC\),
\(\OB\) and \(\ARR{X}{Y}\) will stand for, respectively, the class of all
objects of \(\ccC\) and the set of all arrows from \(X\) to \(Y\) in this
category where \(X,Y \in \OB\). Additionally, we use \(\AUT[\ccC]{X}\) to denote
the automorphism group of the object \(X\).\par
We can now introduce a central notion of this section.

\begin{dfn}{cat-ab-ho}
Let \(X \in \OB\). \(X\) is said to be \emph{absolutely homogeneous} (again
abbreviated as \emph{AB-HO}) if the following condition is fulfilled:
\begin{quote}\itshape
For any \(A \in \OB\) and any two arrows \(u, v \in \ARR{A}{X}\) there exists
\(\phi \in \AUT[\ccC]{X}\) such that \(v = \phi \circ u\).
\end{quote}
The above condition says that each object essentially admits at most one arrow
into \(X\).
\end{dfn}

A similar concept (of \(\kkK\)-homogeneous objects where \(\kkK\) is
a subcategory of the given one) was earlier studied by Kubi\'{s} \cite{kub}.\par
To obtain results that will resemble main theorems of Section~\ref{sec:gec}, we
need to impose certain conditions on the category under considerations, which we
will now do. Recall that an arrow \(u \in \ARR{X}{Y}\) is said to be
a (categorical) \emph{monomorphism} if it satisfies the following condition:
whenever \(Z \in \OB\) and \(v, w \in \ARR{Z}{X}\) are such that \(u \circ v = u
\circ w\), then \(v = w\). For two objects \(X, Y \in \OB\) we will write:
\begin{itemize}
\item \(X \equiv_{\ccC} Y\) if there exists an isomorphism (in \(\ccC\)) between
 them;
\item \(X \prec_{\ccC} Y\) to denote that \(\ARR{X}{Y} \neq \varempty\).
\end{itemize}
When all arrows in \(\ccC\) are monomorphisms (as we will be assuming in
a while), \(X \prec_{\ccC} Y\) can be interpreted as that \(X\) \emph{embeds}
into \(Y\). Observe that \(\prec_{\ccC}\) is reflexive and transitive.\par
Note that if an object \(X\) of a category \(\ccC\) is AB-HO, then \(\ARR{X}{X}
= \AUT[\ccC]{X}\). Thus, when studying absolute homogeneity, it is natural to
consider only such categories in which arrows are monomorphisms.

\begin{dfn}{axioms}
Let \(\ccC\) be a category. A class \(\ffF\) of certain objects from \(\ccC\) is
said to be:
\begin{itemize}
\item \emph{closed under isomorphic copies} if the following condition is
 fulfilled:
\item[(Iso)] whenever \(X, Y \in \OB\) are such that \(X \equiv_{\ccC} Y\) and
 \(Y \in \ffF\), then \(X \in \ffF\);
\item \emph{hereditary} if it has the following property:
\item[(Sub)] whenever \(X, Y \in \OB\) are such that \(X \prec_{\ccC} Y\) and
 \(Y \in \ffF\), then \(X \in \ffF\);
\item \emph{proper} if
\item[(Set)] there exists a \underline{set} \(W\) of objects from  \(\OB\) such
 that any object from \(\ffF\) is isomorphic to some object from \(W\).
\end{itemize}
We call a category \(\ccC\) \emph{\CARD} if all the following conditions are
satisfied:
\begin{itemize}
\item[(Arr)] All arrows in \(\ccC\) are monomorphisms and \(\ARR{X}{Y}\) is
 a \underline{set} for any two objects \(X, Y \in \OB\).
\item[(Sys)] All directs systems in \(\ccC\) have (direct) limits.
\item[(Card)] To any object \(X \in \OB\) it is assigned a cardinal number, to
 be called \emph{size} of \(X\) and denoted by \(|X|\), in a way that isomorphic
 objects have the same size.
\item[(Lim)] Any object from \(\ccC\) with infinite size \(\alpha\) is
 isomorphic to the object of the direct limit of some direct system consisting
 of at most \(\alpha\) objects each of which has finite size.
\item[(Size)] If a direct system in \(\ccC\) consists of \(\alpha \geq
 \aleph_0\) objects each of which has finite size, then the object of its direct
 limit has size not greater than \(\alpha\).
\item[(Dir)] If \(X, Y \in \OB\) have finite size and \(Z \in \OB\), \(u \in
 \ARR{X}{Z}\) and \(v \in \ARR{Y}{Z}\) are arbitrary, then there exist an object
 \(W \in \OB\) with finite size and three morphisms \(\alpha \in \ARR{X}{W}\),
 \(\beta \in \ARR{Y}{W}\) and \(\gamma \in \ARR{W}{Z}\) such that \(u = \gamma
 \circ \alpha\) and \(v = \gamma \circ \beta\).
\item[(Fin)] For any object \(X \in \OB\) the class \(\FIN[\ccC]{X}\) consisting
 of all objects \(A \in \OB\) of finite size such that \(A \prec_{\ccC} X\) is
 proper (that is, axiom (Set) holds for \(\ffF = \FIN[\ccC]{X}\)).
\item[(Upp)] For any object \(X \in \OB\) there exists a cardinal number \(\mM\)
 with the propery that if \(Y \in \OB\) satisfies \(Y \prec_{\ccC} X\), then
 \(|Y| \leq \mM\).
\end{itemize}
and any object \(X \in \OB\) of finite size has the following property:
\begin{itemize}
\item[\upshape(Emb)] Whenever
 \SYS{Z_{\alpha}}{p_{\alpha,\beta}}{\alpha\leq\beta} is a direct system in
 \(\ccC\) such that the object \(W\) of its direct limit satisfies \(X
 \prec_{\ccC} W\), then \(X \prec_{\ccC} Z_{\alpha}\) for some index \(\alpha\).
\end{itemize}
A \CARD\ category \(\ccC\) is said to be \emph{regular} if for any two its
objects \(X\) and \(Y\) the following condition is satisfied:
\begin{itemize}
\item[(Reg)] If \(X \prec_{\ccC} Y\) and \(|Y| < \aleph_0\), then \(|X| <
 \aleph_0\) as well.
\end{itemize}
\end{dfn}

\begin{exm}{fin-reg}
Below we give a few examples illustrating the above notions and axioms.
In the categories indicated in examples (B)--(D) below arrows are defined as
one-to-one homomorphisms between respective groups.
\begin{enumerate}[(A)]
\item The category of all dGEC's is \CARD\ and regular.
\item The category of all locally finite groups is \CARD\ and regular when
 the size is defined as the cardinality of a given group. (Recall that a group
 is \emph{locally finite} if all its finite subsets generate finite groups.)
\item The category of all Abelian groups is \CARD\ and regular when the size
 is defined as the least cardinality among cardinalities of all subsets
 generating a given group.
\item The category of all groups is \CARD\ \textbf{but not regular} (because
 the free group on two generators contains the free group on infinitely many
 generators as a subgroup) when the size is defined as the least cardinality
 among cardinalities of all subsets generating a given group.
\item Now let \(\ccC\) be a category where:
 \begin{itemize}
 \item \(\OB\) is a class of all compact groups;
 \item for \(G, H \in \OB\), \(\ARR{G}{H}\) consists of all surjective
  continuous homomorphisms from \(H\) onto \(G\);
 \item the composition of two arrows is defined as the composition of respective
  homomorphisms but in a reverse order, that is: the composition of \(u \in
  \ARR{G}{H}\) with \(v \in \ARR{H}{K}\) equals \(u \circ v \in \ARR{G}{K}\)
  (note that then each arrow is a categorical monomorphism and direct limits
  of direct systems in this category coincide with inverse limits of inverse
  systems in a natural direction of arrows as maps);
 \item the size of an object \(G \in \OB\) is defined as follows: \(|G|\) is
  the topological weight of \(G\) if \(G\) is not a Lie group, or else \(|G|\)
  is the topological dimension of a Lie group \(G\).
 \end{itemize}
 One proves that the above category is \CARD\ and regular.
\item Let us give a negative example. If the size is defined as the topological
 weight of a topological space (or the density character of the space), usually
 one of axioms (Lim) or (Emb) will fail to hold. Let us discuss this issue on
 an example of complete separable metric spaces. Their weight is countable, so
 (Lim) forces that any such a space \(M\) has to be (isometric to) the direct
 limit of a countable direct system of finite spaces, which in turn means that
 \(M\) will be identified with some/all of its own dense subsets as they
 coincide with the set-theoretic direct limits of countable direct systems
 mentioned above. But then (Emb) may fail---if for \(X\) (therein) one takes
 a finite subset of \(M\) that is disjoint from the above dense subset \(D\) and
 does not embed isometrically into \(D\), we will lose (Emb).
\end{enumerate}
\end{exm}

Similarly as for GEC's, we define:

\begin{dfn}{fin-str}
Let \(\ccC\) be a \CARD\ category. A class \(\ffF\) of objects from \(\ccC\)
is said to be:
\begin{itemize}
\item \emph{hereditary w.r.t. finite size} if it has the following property:
\item[(Sub\({}_\omega\))] whenever \(X, Y \in \OB\) are such that
 \(X \prec_{\ccC} Y\), \(Y \in \ffF\) and \(X\) has finite size, then \(X \in
 \ffF\) as well;
\item a \emph{\textbf{finitary structure}} if \(\ffF\) consists of objects of
 finite size, is closed under isomorphic copies and hereditary w.r.t. finite
 size.
\end{itemize}
Let \(\ffF\) be a finitary structure. An object \(X \in \OB\) is said to be
\emph{finitely represented in \(\ffF\)} if every object \(A \in \OB\) of finite
size such that \(A \prec_{\ccC} X\) belongs to \(\ffF\). We denote by
\(\grp{\ffF}\) the class of all objects that are finitely represented in
\(\ffF\) and by \(\grp{\ffF}_{\alpha}\) (where \(\alpha\) is an infinite
cardinal) the class of all objects \(X \in \grp{\ffF}\) with \(|X| \leq
\alpha\). Observe that \(\grp{\ffF}\) is always closed under isomorphic copies
and hereditary.\par
For any object \(X \in \OB\), \(\FIN[\ccC]{X}\) (defined in axiom (Fin)) is
always a \underline{proper} finitary structure (thanks to axiom (Fin)) and we
call it \emph{finitary structure induced by \(X\)}. Note that \(X\) is finitely
represented in \(\ffF\) iff \(\FIN[\ccC]{X} \subset \ffF\).
\end{dfn}

To familiarize the reader more with \CARD\ categories, below we present
a surprising (but quite simple) result which says that in a regular \CARD\
category the class of all objects of finite size as well as the size of all
other objects are uniquely determined by the axioms of such categories. Namely:

\begin{pro}{card-cat}
Let \(X\) be an object in a regular \CARD\ category \(\ccC\).
\begin{itemize}
\item[\upshape(a)] \(X\) has finite size iff it satisfies axiom \UP{(Emb)}.
\item[\upshape(b)] If \(X\) does not satisfy \UP{(Emb)}, then the size of \(X\)
 is the least infinite cardinal number \(\alpha\) such that \(X\) is isomorphic
 to the object of the direct limit of a direct system consisting of at most
 \(\alpha\) objects of finite size.
\end{itemize}
\end{pro}
\begin{proof}
Instead of giving a detailed proof, we will give the reader all necessary hints
to prove both the parts. The `if' part of (a) follows immediately from axioms
(Lim), (Emb) and (Reg) (which is here assumed), whereas (b) is an immediate
consequence of the tandem (Lim) and (Size).
\end{proof}

\begin{rem}{reg-card}
The above result gives us a method of verification whether a specified category
that fulfills axioms (Arr) and (Sys) is a regular \CARD\ category with respect
to some \emph{size}: first we define the class \(\ffF\) of (all) objects of
finite size as consisting of those which satisfy (Emb). Next we verify whether
axioms (Dir), (Fin) and (Reg) are satisfied. Finally, we need to check whether
for any object \(X \notin \ffF\) there exists an infinite cardinal number
\(\alpha\) such that \(X\) is the direct limit consisting of at most \(\alpha\)
objects from \(\ffF\). And if so, we define \(|X|\) as the least such
\(\alpha\). In this way we `force' (Lim) and (Size) to hold. Finally, we have to
verify if (Upp) holds.
\end{rem}

Before passing to main results of the section, we establish certain auxiliary
results that will be used later. To simplify statements, we fix the terminology:
\begin{itemize}
\item if a direct system is indexed by all ordinals less than an infinite
 ordinal \(\xi\) (and the order between indices is compatible with the order of
 ordinals), we call it a \emph{\(\xi\)-transfinite direct sequence};
\item a \(\xi\)-transfinite direct sequence \(\Zz =
 \SYS{Z_{\alpha}}{p_{\alpha,\beta}}{\alpha\leq\beta}\) is said to be
 \emph{continuous} if for each limit ordinal \(\alpha < \xi\), the pair
 \((Z_{\alpha},\{p_{\alpha,\beta}\}_{\beta<\alpha})\) is the direct limit of
 the (direct) subsystem of \(\Zz\) consisting of all objects whose indices
 precede \(\alpha\);
\item for simplicity, each arrow \(p_{\alpha,\beta}\) with \(\alpha < \beta\)
 in a \(\xi\)-transfinite direct sequence we call \emph{essential} and all other
 are called \emph{inessential}.
\end{itemize}

\begin{lem}{transfin}
Let \(X\) be an object of infinite size \(\mM\) in a \CARD\ category \(\ccC\).
Let \(\xi\) be an initial ordinal of size \(\mM\). Then \(X\) is isomorphic to
the object of the direct limit of a continuous \(\xi\)-transfinite direct
sequence consisting of objects each of which has size less than \(\mM\).
\end{lem}
\begin{proof}
In what follows, we denote by \(\omega\) the first infinite ordinal and for any
ordinal \(\alpha\), \(|\alpha|\) will stand for its cardinality.\par
It follows from axioms (Lim) and (Size) that there exists a direct system
\[\SYS{Z_{\sigma}}{p_{\sigma,\tau}}{\sigma\ORD\tau}\] constisting of objects of
finite size where indices run over a directed set \((\Sigma,\ORD)\) of size
\(\mM\). We distinguish two cases.\par
First assume \(\xi = \omega\). Write \(\Sigma = \{\sigma_0,\sigma_1,\ldots\}\)
and inductively define a sequence \(\tau_0,\tau_1,\ldots \in \Sigma\) such that:
\begin{itemize}
\item \(\tau_0 = \sigma_0\); and
\item \(\tau_{n-1} \ORD \tau_n\) as well as \(\sigma_n \ORD \tau_n\) for each
 \(n > 0\).
\end{itemize}
Then \(\Zz = \SYS{Z_{\tau_n}}{p_{\tau_n,\tau_m}}{n\leq m}\) is the direct
sequence we searched for.\par
Now assume \(\omega < \xi\). To avoid confusion, set \(I \df \{\alpha\dd\ \alpha
< \xi\}\) and \(J \df \{\alpha \in I\dd\ \alpha \geq \omega\}\). Since the set
\(\Sigma\) has size \(\mM\), there exists a bijection \(\kappa\dd I \to
\Sigma\). Using transfinite induction, it is easy to construct a collection
\(\{S_{\alpha}\dd\ \alpha \in J\}\) of subsets of \(\Sigma\) with the following
five properties (below \(\alpha\) and \(\beta\) are arbitrary ordinals from
\(J\)):
\begin{itemize}
\item \((S_{\alpha},\ORD)\) is a directed set;
\item \(S_{\beta} \subset S_{\alpha}\) whenever \(\beta \leq \alpha\);
\item \(\kappa(\{\eta \in I\dd\ \eta < \alpha\}) \subset S_{\alpha}\);
\item \(\card(S_{\alpha}) \leq |\alpha|\);
\item \(S_{\alpha} = \bigcup_{\beta<\alpha} S_{\beta}\) for all limit ordinals
 \(\alpha \neq \omega\).
\end{itemize}
For any \(\alpha \in J\) denote by \(\Zz_{\alpha}\) a direct system obtained
from \(\Zz\) by restricting indices only to the set \(S_{\alpha}\). Now, again
using transfinite induction (as well as axiom (Sys)) we define
a \(\xi\)-transfinite direct sequence \(\Ww =
\SYS{W_{\alpha}}{q_{\alpha,\beta}}{\alpha\leq\beta}\) as follows:
\begin{itemize}
\item \(W_{\alpha}\) is the object of the direct limit of
 \(\Zz_{\omega+\alpha}\),
\item \(q_{\alpha,\alpha}\) is the identity of \(W_{\alpha}\)
\end{itemize}
and for \(\beta < \alpha\) the arrow \(q_{\beta,\alpha}\) is constructed as
follows. For any \(\mu \in I\) and \(\sigma \in S_{\omega+\mu}\) we have
an arrow \(r_{\sigma,\mu} \in \ARR{Z_{\sigma}}{W_{\mu}}\) that is an ingredient
of the direct limit. And now it follows directly from the condition defining
direct limit (for \(\Zz_{\omega+\beta}\)) that whenever \(\beta < \alpha\),
there exists a \underline{unique} arrow \(q_{\beta,\alpha} \in
\ARR{W_{\beta}}{W_{\alpha}}\) such that \(r_{\sigma,\alpha} = q_{\beta,\alpha}
\circ r_{\sigma,\beta}\) for any \(\sigma \in S_{\omega+\beta}\). This
uniqueness implies that all the arrows \(q_{\beta,\alpha}\) satisfy
compatibility condition. Moreover, one checks that the direct system \(\Ww\) is
continuous, which follows from the last property of the sets \(S_{\alpha}\). In
the very same way one shows that the object of the direct limit of \(\Zz\)
(which is isomorphic to \(X\)) is isomorphic to the object of the direct limit
of \(\Ww\). So, it remains to check that \(|W_{\alpha}| < \mM\), which simply
follows from the construction and axiom (Size).
\end{proof}

\begin{lem}{closure}
Let \(\ffF\) be a finitary structure in a \CARD\ category \(\ccC\).
\begin{enumerate}[\upshape(a)]
\item \(\grp{\ffF}\) consists precisely of those objects that are isomorphic to
 objects of direct limits of direct systems consisting of objects from \(\ffF\).
\item Objects of (direct) limits of direct systems consisting of objects from
 \(\grp{\ffF}\) belong to \(\grp{\ffF}\) as well.
\item If \(\ffF\) is proper, then for any cardinal \(\alpha \geq \aleph_0\),
 the class \(\grp{\ffF}_{\alpha}\) is proper as well.
\end{enumerate}
\end{lem}
\begin{proof}
If \(X \in \grp{\ffF}\), it follows from (Lim) that \(X\) satisfies
the condition specified in (a). Conversely, if \(X \in \OB\) is isomorphic to
the object of the direct limit of a certain direct system consisting of objects
from \(\grp{\ffF}\) and \(A \in \FIN[\ccC]{X}\), then it follows from (Emb) that
\(A \prec_{\ccC} B\) where \(B \in \grp{\ffF}\) is an object from this direct
system. This yields \(B \in \ffF\) and, consequently, \(X \in \grp{\ffF}\).
Since \(\ffF\) is hereditary w.r.t. finite size, we get \(\ffF \subset
\grp{\ffF}\), which finishes the proof of both (a) and (b).\par
(c): We involve transfinite induction. We fix infinite cardinal \(\alpha\) and
assume that for all infinite \(\beta < \alpha\) the class \(\grp{\ffF}_{\beta}\)
is proper. It follows from \LEM{transfin} that each member of
\(\grp{\ffF}_{\alpha}\) is isomorphic to the object of the direct limit of
a \(\xi\)-transfinite direct sequence consisting of objects of size less than
\(\alpha\) where \(\xi\) is an initial ordinal of size \(\alpha\). But, since
\(\ffF\) is proper (if \(\alpha = \aleph_0\)) as well as
\(\bigcup_{\beta<\alpha} \grp{\ffF}_{\beta}\) (if \(\alpha > \aleph_0\)), we see
that, up to isomorphism, all such transfinite sequences form a set (thanks to
(Arr)) and hence \(\ffF_{\alpha}\) is proper.
\end{proof}

And now we will see another application of the Erd\H{o}s-Rado theorem.

\begin{lem}{long}
Let \(\ccC\) be a category that satisfies axioms \UP{(Arr)} and \UP{(Sys)} and
such that the class \(\OB\) is proper. Then there exists an ordinal \(\eta\)
with the following property:
\begin{quote}
If \SYS{Z_{\alpha}}{p_{\alpha,\beta}}{\alpha\leq\beta} is a \(\xi\)-transfinite
direct sequence in \(\ccC\) with non-isomorphic essential arrows, then \(\xi <
\eta\).
\end{quote}
\end{lem}
\begin{proof}
Let \(\Ww\) be a set of objects from \(\ccC\) that contains an isomorphic copy
of any object from \(\ccC\). Let \(\mM\) be an infinite cardinal number such
that \(\card(\Ww) \leq \mM\) and \(\card(\ARR{X}{X}) \leq \mM\) for all \(X \in
\Ww\). Denote by \(\eta\) the least ordinal whose size is greater that
\(2^{\mM}\). We will now show that \(\eta\) is the ordinal we search for. To
this end, we fix a \(\xi\)-transfinite direct sequence
\SYS{Z_{\alpha}}{p_{\alpha,\beta}}{\alpha\leq\beta} with all essential arrows
non-isomorphic. With no loss on generality, we may and do assume that
\(Z_{\alpha} \in \Ww\) for any \(\alpha < \xi\). We argue by a contradiction.
Reducing \(\xi\) if necessary, we assume \(\xi = \eta\). Since the cardinality
of \(\eta\) is a successor cardinal (and hence a regular cardinal), passing to
a transfinite subsequence (in our direct system) of the same size, we may (and
do) further assume that \(Z_{\alpha} = Z\) for all \(\alpha < \eta\) and certain
\(Z \in \Ww\). Now it follows from \THM{r-e-r} that there exists an infinite
set, say \(A\), of indices such that \(p_{\alpha,\beta} = q\) whenever \(\alpha
< \beta\) and both \(\alpha\) and \(\beta\) are in \(A\) (here \(q\) is
a certain fixed arrow from \ARR{Z}{Z}). Finally, choose any \(\alpha,\beta,
\gamma \in A\) such that \(\alpha < \beta < \gamma\). Then \(q =
p_{\alpha,\gamma} = p_{\beta,\gamma} \circ p_{\alpha,\beta} = q \circ q\). Since
all arrows in \(\ccC\) are monomorphisms (by (Arr)), we conclude that \(q\) is
the identity on \(Z\), which contradicts the assumption about the direct system
and finishes the proof.
\end{proof}

The last thing needed for our further purposes is amalgamation. Although it is
a classical categorical notion, below we formulate it in a way convenient for
us.

\begin{dfn}{amalg-cat}
Let \(\ffF\) be a class of certain objects from a category \(\ccC\). We say
\(\ffF\) has the \emph{amalgamation property} in \(\ccC\) (briefly: has \AP) if
the following two conditions are fulfilled:
\begin{itemize}\itshape
\item Whenever \(X, Y, Z\) are objects from \(\ffF\) and \(u\dd X \to Y\) and
 \(v\dd X \to Z\) are two arrows from \(\ccC\), then there exists an object
 \(W\) \textbf{from} \(\pmb{\ffF}\) and two arrows \(p\dd Y \to W\) and \(q\dd Z
 \to W\) from \(\ccC\) such that \(p \circ u = q \circ v\).
\item For any two objects \(X, Y\) from \(\ffF\) there exists a third object
 \(Z \in \ffF\) such that \(X \prec_{\ccC} Z\) and \(Y \prec_{\ccC} Z\).
\end{itemize}
We say a finitary structure \(\ffF\) in a \CARD\ category \(\ccC\) \emph{has
\AP[\(*\)-]} if \(\grp{\ffF}\) has \AP.
\end{dfn}

We are now ready to present main results of this section. The first of them
is a counterpart of \THM{embed}, whereas the second of \THM{skeleton}.

\begin{thm}{AB-HO-cat}
Let \(X \in \OB\) be AB-HO in a \CARD\ category \(\ccC\) and let \(\ffF\) denote
\(\FIN[\ccC]{X}\). Then:
\begin{enumerate}[\upshape(A)]
\item For any object \(Y \in \OB\), \(Y \prec_{\ccC} X\) iff \(Y \in
 \grp{\ffF}\).
\item The class \(\ffF\) is a proper finitary structure with \AP.
\item The class \(\grp{\ffF}\) is proper and has \AP.
\end{enumerate}
\end{thm}
\begin{proof}
We start from (A). We only need to show the `if' part. To this end, we will use
transfinite induction on \(\mM \df |Y|\). If \(\mM\) is finite, we have nothing
to do. Thus, we assume \(\mM\) is infinite and (A) holds for objects \(Z \in
\OB\) with \(|Z| < \mM\). It follows from \LEM{transfin} that \(Y\) is
isomorphic to the object of the direct limit \(Y'\) of a continuous
\(\xi\)-transfinite sequence \SYS{Z_{\alpha}}{p_{\alpha,\beta}}{\alpha\leq\beta}
consisting of objects of size less than \(\mM\) (for some ordinal \(\xi\)). We
may and do assume that \(Y = Y'\). We infer from the transfinite induction
hypothesis that \(Z_{\alpha} \prec_{\ccC} X\) (because \(\FIN[\ccC]{Z_{\alpha}}
\subset \FIN[\ccC]{Y}\)). Now we can repeat the idea presented in the proof of
\THM{embed} to construct a family of arrows \(u_{\alpha} \in
\ARR{Z_{\alpha}}{X}\) such that \(u_{\beta} = u_{\alpha} \circ
p_{\alpha,\beta}\) whenever \(\beta < \alpha < \xi\) (here it is important that
the transfinite direct sequence is continuous---for constructing the arrow for
limit \(\alpha\)). In this way, after passing to the direct limit, we will get
also an arrow from \(Y\) into \(X\). We skip technical details.\par
Since (B) follows from (C) (and axiom (Dir)), first we give a proof of (C). It
follows from axioms (Fin) and (Upp) that \(\ffF\) is proper and there exists
a cardinal \(\mM\) such that \(\grp{\ffF} = \grp{\ffF}_{\mM}\). So, we conclude
from \LEM{closure} that \(\grp{\ffF}\) is proper as well. To show \AP\ for that
class, we only need to check the first condition of \AP\ (as the second
immediately follows from (A)). So, assume \(u\dd C \to A\) and \(v\dd C \to B\)
are two arrows from \(\ccC\) between objects from \(\grp{\ffF}\). It follows
from (A) that there are arrows \(\alpha\dd A \to X\) and \(\beta\dd B \to X\)
from \(\ccC\). Then it follows from the fact that \(X\) is AB-HO that there
exists \(\phi \in \AUT[\ccC]{X}\) for which \(\beta \circ v = \phi \circ (\alpha
\circ u)\). As \(X \in \grp{\ffF}\), the arrows \(\beta\dd B \to X\) and \(\phi
\circ \alpha\dd A \to X\) are those searched in the amalgamation condition.\par
Finally, we pass to (B). That \(\ffF\) is a finitary structure we have already
commented in \DEF{fin-str}. As we announced it earlier, \AP\ is a consequence
of (C) and (Dir). We leave the details to the reader.
\end{proof}

Our next aim is to formulate a theorem that characterises finitary structures
induced by AB-HO objects. To this end, we will now introduce skeletons and
skeletoids (as for GEC's).

\begin{dfn}{skton-cat}
A finitary structure \(\ffF\) of certain objects of a \CARD\ category \(\ccC\)
satisfies \emph{size boundedness principle} (abbreviated: \SBP) if there is
a cardinal number \(\mM\) such that \(|X| \leq \mM\) for any \(X \in
\grp{\ffF}\).\par
A \emph{skeletoid} in a \CARD\ category \(\ccC\) is a proper finitary structure
\(\ffF\) that satisfies \SBP\ and has \AP. A \emph{skeleton} in \(\ccC\) is
a class of the form \(\FIN[\ccC]{X}\) for some AB-HO object \(X \in \OB\).
\end{dfn}

The following result is a counterpart of \THM{skeleton}.

\begin{thm}{skton-cat}
A class of objects of a \CARD\ category is a skeleton iff it is a skeletoid with
\AP[\(*\)-].
\end{thm}
\begin{proof}
Thanks to \THM{AB-HO-cat}, we only need to prove the `if' part. To this end, we
fix a skeletoid \(\ffF\) with \AP[\(*\)-] in a \CARD\ category \(\ccC\). It
follows from \SBP\ and \LEM{closure} that \(\eeE \df \grp{\ffF}\) is proper and
is closed under passing to direct limits. So, the class \(\eeE\) considered as
a full subcategory of \(\ccC\) (i.e., with all arrows from \(\ccC\) between
objects from \(\eeE\)) satisfies all the assumptions of \LEM{long} and therefore
there exists an infinite ordinal \(\eta\) with a property specified there. Now
we will mimic the proof of \THM{skeleton}. We fix a \underline{set} \(\Ww\) of
objects from \(\eeE\) that contains an isomorphic copy of each object from
\(\eeE\). It follows from Zorn's lemma that in the set of all
\(\xi\)-transfinite direct sequences consisting of objects from \(\Ww\) and with
non-isomorphic essential arrows where \(\xi \leq \eta\) there exists a maximal
one (w.r.t. a partial order consisting of extending a system by adding new
objects at its end), say \(\Zz =
\SYS{Z_{\alpha}}{p_{\alpha,\beta}}{\alpha\leq\beta}\) with indices less than
some \(\xi \leq \eta\). We infer from \LEM{long} that \(\xi < \eta\). From now
on, we again consider \(\eeE\) as a class of objects from \(\ccC\). Observe that
\(\xi\) is not limit---otherwise we could extend \(\Zz\) by adding at its end
its direct limit (with its object from \(\Ww\); the limit arrows will still not
be isomorphic as the composition of two monomorphism can be an isomorphism only
if both these arrows are isomorphisms). So, \(\Zz\) has its last object, say
\(X\). We will now show that \(X\) is AB-HO and \(\FIN[\ccC]{X} = \ffF\) (which
will finish the proof). First of all, observe that for any \(V \in \eeE\), each
arrow from \(X\) into \(V\) is an isomorphism (because otherwise we could extend
\(\Zz\) by a suitable copy of \(V\) in \(\Ww\)). In particular, \(\ARR{X}{X} =
\AUT[\ccC]{X}\). Hence we conclude that:
\begin{itemize}
\item[(\(*\))] If \(V, W\) are objects from \(\eeE\) and \(u\dd V \to X\) and
 \(v\dd V \to W\) are two arrows from \(\ccC\), then there exists an arrow
 \(w\dd W \to X\) from \(\ccC\) such that \(u = w \circ v\).
\end{itemize}
One shows the above property similarly as in the proof of \THM{skeleton} and
thus we skip the proof of (\(*\)) (this is a moment where \AP[\(*\)-] is
involved). Substituting \(X\) for \(W\) in (\(*\)) we get that \(X\) is AB-HO.
Indeed, if \(V \in \OB\) and there exists an arrow in \(\ccC\) from \(V\) into
\(X\), then \(V \in \eeE\) (as \(\eeE\) is hereditary and contains \(X\)) and we
may apply (\(*\)). So, it remains to show that \(\FIN[\ccC]{X} = \ffF\).
The inclusion ``\(\subset\)'' is immediate, since \(X \in \eeE = \grp{\ffF}\).
To verify the reverse one, take any object \(A \in \ffF\). It follows from
the second condition defining \AP\ that there is an object \(B \in \eeE\) such
that \(A \prec_{\ccC} B\) and \(X \prec_{\ccC} B\). But then \(X \equiv_{\ccC}
B\) and, consequently, \(A \prec_{\ccC} X\), which finishes the proof.
\end{proof}

Although we have obtained a result that resembles \THM{skeleton}, one can feel
unsatisfied due to the form of \SBP\ (which is a counterpart of both \MBP\ and
\OBP) as it is not formulated in terms of intrinsic properties of the skeletoid
under considerations, whereas both \MBP\ are \OBP\ are so. We consider this
defect as a major challenge for our future research in this topic. As both \MBP\
and \OBP\ are direct consequences of the Erd\H{o}s-Rado theorem, any criterion
(that is, a sufficient condition) for \SBP\ may be considered as
a \emph{categorical Erd\H{o}s-Rado phenomenon}. Thus, we end this section with
the following

\begin{prb}{E-R}
Find convenient/applicable conditions for a proper finitary structure with
 \AP[\(*\)-] in a \CARD\ category to satisfy \SBP.
\end{prb}

\section{Concluding remarks and open problems (revisited)}\label{sec:fin}

We believe our work will open new research areas and give a solid background for
further research. We hope to find another new examples of absolutely homogeneous
structures in a close future. For the reader's convenience, below we list all
open problems that were posed in previous parts and add a few new.
\begin{enumerate}[(Prb1)]
\item Does there exist an absolutely homogeneous non-locally compact complete
 metric space that is both connected and locally connected?
\item Is it true in ZFC that for any two infinite cardinals \(\alpha\) and
 \(\beta\) such that \(\alpha < \beta < 2^{\alpha}\) there exists an absolutely
 homogeneous graph with edge-coloring whose palette and set of vertices have
 size, respectively, \(\alpha\) and \(\beta\)?
\item Is it possible to find a formula for the size of an infinite absolutely
 homogeneous graph with edge-coloring based purely on its skeleton?
\item How to reconstruct the absolutely homogeneous graph with edge-coloring
 generated by a skeleton with countable palette from the Fra\"{\i}ss\'{e} limit
 of this skeleton?
\item How to `encode' (or `decode') topological properties of an absolutely
 homogeneoue metric space in terms of its skeleton?
\item Find a class of graphs with edge-coloring larger than consisting of those
 with isosceles triangles for which a classification of graphs from that class
 will be possible (in the spirit of \THM{classification}).
\item Does there exist a connected compact Hausdorff space equipped with
 reflective edge similarity that is absolutely homogeneous (as such a structure)
 and homeomorphic to no Euclidean sphere, but its isotropy group is not totally
 disconnected?
\item Find convenient/applicable conditions for a proper finitary structure with
 \AP[\(*\)-] in a \CARD\ category to satisfy \SBP.
\item Resolve whether each of the axioms of \CARD\ categories (listed in
 \DEF{axioms}) is independent of all the others at once.
\item In reference to part (C) of \PRO{*AP}, is it true that a proper finitary
 structure \(\ffF\) in a \CARD\ category \(\ccC\) has \AP[\(*\)-] iff
 the amalgamation property holds for objects from \(\grp{\ffF}\) of size not
 greater than \(\mM\) where \(\mM\) is the cardinality of \(\ffF\) after
 identifying isomorphic objects?
\end{enumerate}

\end{document}